\documentclass[11pt]{amsart}

\usepackage[margin=3cm]{geometry}

\usepackage[foot]{amsaddr}
\usepackage{amssymb}
\usepackage{amsthm}
\usepackage{amsmath}
\usepackage{amsmath,color}
\usepackage{amstext}
\usepackage{amsopn}
\usepackage{graphicx}
\usepackage{float}
\usepackage{color}
\usepackage{dsfont}
\usepackage{epstopdf}
\usepackage{hyperref}
\usepackage{comment}
\usepackage{enumitem}
\usepackage{bm}
\allowdisplaybreaks

\numberwithin{equation}{section}

\newtheorem{theorem}{Theorem}[section]
\newtheorem{corollary}[theorem]{Corollary}

\newtheorem{lemma}[theorem]{Lemma}
\newtheorem{proposition}[theorem]{Proposition}

\theoremstyle{definition}

\newtheorem{remark}{Remark}[section]

\title[] 
      {
  On
 principal eigenvalues of linear time-periodic parabolic systems: symmetric mutation case
}

 \author[ 
 ]{Shuang Liu
 }

 \thanks{{S. Liu}: School of Mathematics and Statistics, Beijing Institute of Technology, Beijing, 100081,
  China.
  }

 \email{liushuangnqkg@bit.edu.cn
}



\begin{document}
\maketitle

%
%

\begin{abstract}
    The paper  is concerned with  the effect of the spatio-temporal heterogeneity on the principal eigenvalue of  some linear time-periodic parabolic system.  
    Various asymptotic behaviors of
the principal  eigenvalue and its  monotonicity, as a function of the diffusion rate and frequency, are first derived. In particular,
some singular behaviors 
of the principal eigenvalues are observed when both diffusion rate and frequency approach zero, 
with some  scalar time-periodic Hamilton-Jacobi equation as the limiting equation.
Furthermore, we completely classify  
the topological structures of
the level sets for the principal eigenvalues  in the plane of frequency and diffusion rate.
Our results not only
generalize
 most of the findings  in \cite{LL2022} 
 for scalar periodic-parabolic operators,
 but also reveal more rich
 global information,  for time-periodic parabolic systems,
 on the dependence
of the principal eigenvalues upon the spatio-temporal heterogeneity.
\end{abstract}

\section{\bf Introduction}\label{S1}

Consider the  coupled  periodic-parabolic  eigenvalue problem
\begin{equation}\label{Liu1}
\left\{
\begin{array}{ll}
\medskip
\omega\partial_t{\bm \varphi}-\rho {\bf D}\Delta{\bm \varphi}-{\bf A}(x,t){\bm \varphi}=\lambda{\bm \varphi} &\mathrm{in}~\Omega\times\mathbb{R},\\
\medskip
  \nabla {\bm \varphi}\cdot\nu =0 & \mathrm{on}~\partial\Omega\times\mathbb{R},\\
  {\bm \varphi}(x,t)={\bm \varphi}(x,t+1), &\mathrm{in}~\Omega\times\mathbb{R},
  \end{array}
  \right.
 \end{equation}
where  $\Omega$ is a bounded domain in $\mathbb{R}^n$ with smooth boundary $\partial\Omega$ and $\nu(x)$ denotes the unit outward
normal vector at $x\in\partial\Omega$.
The operator $\Delta=\sum_{i=1}^n \partial_{x_ix_i}$ is the  Laplace operator in $\mathbb{R}^n$
and ${\bf D}={\rm diag}(d_1,\cdots,d_n)$ is an  $n\times n$ diagonal matrix with  constant $d_i>0$, $i=1,\cdots,n$.
For each $(x,t)\in \Omega\times \mathbb{R}$,  ${\bf A}(x,t)
$ is an  essentially positive (cooperative) 
matrix (its off-diagonal entries are nonnegative), which is assumed to be symmetric and time-periodic with unit period.  Inspired by \cite{BH2020,S1992}, we assume
${\bf A}=(a_{ij})_{n\times n}$ is fully coupled, 
i.e. the index set $\{1,\cdots,n\}$ cannot be split up in two disjoint nonempty
sets $\mathcal{I}$ and $\mathcal{J}$ such that $a_{ij}(x, t)\equiv 0$ in $\Omega\times\mathbb{R}$ for $i\in\mathcal{I}$ and $j\in\mathcal{J}$.
Parameters $\omega,\rho>0$ represent the frequency and diffusion rate, respectively.


By the Krein-Rutman theorem \cite{KR1950}, it is shown in \cite{BH2020, S1992} that
 problem \eqref{Liu1} admits a real and simple eigenvalue, called {\it principal eigenvalue}. 
 It
 has the smallest real part among all eigenvalues of \eqref{Liu1} and corresponds to a positive eigenfunction, 
  for which each entry is a positive time-periodic function. 
  Our study is focused on the dependence of principal eigenvalue
on frequency
and diffusion rate. 

\subsection{
Background and motivations}\label{subsection1.1}
There has recently been considerable interest in issues related to qualitative analysis of  the principal eigenvalue for problem \eqref{Liu1}; see for example \cite{BH2020,BLSS2024, D2009,LL2016,ZZ2021}.
 Such interest stems mainly from 
the study of reaction-diffusion 
systems
in spatio-temporally heterogeneous media. In particular, the principal eigenvalue can often be regarded as a threshold
value  in determining the dynamics of the associated nonlinear systems 
\cite{BHN2022,CC2003,Hutson2001,LL2022book}.

Our motivation comes from the following linear system for  $n$ phenotypes of a species with densities ${\bf u}(x,t)=(u_1(x,t),\cdots,u_n(x,t))$: 
 \begin{equation}\label{liu-20240724-1}
\left\{
\begin{array}{ll}
\medskip
\omega\partial_t{\bf u}=\rho {\bf D}\Delta{\bf u}+{\bf M}(x,t){\bf u}+ {\rm diag}(c_i(x,t)){\bf u}&\mathrm{in}~\Omega\times(0,+\infty),\\
  \nabla {\bf u}\cdot\nu =0 & \mathrm{on}~\partial\Omega\times(0,+\infty),
  \end{array}
  \right.
 \end{equation}
 where $ {\rm diag}(c_i(x,t))$ is a diagonal matrix, 
 with each
$c_i$ being a time-periodic function denoting  the birth-death rate of phenotype $i$. The 
essentially positive matrix function
${\bf M}(x,t)=(m_{ij}(x,t))_{n\times n}$ satisfying
$$
 m_{ii}(x,t)=-\sum_{j\not= i} m_{ij}(x,t),\quad \forall (x,t)\in \Omega\times\mathbb{R},\,\,i=1,\cdots,n,
 $$
represents the mutation of phenotypes.

The
long-time behaviors of the solution to \eqref{liu-20240724-1} is determined by the signs of the principal eigenvalue
of
problem \eqref{Liu1} with
\begin{equation*}
     {\bf A}(x,t)= {\bf M}(x,t)+{\rm diag}(c_i(x,t)), \quad \forall (x,t)\in \Omega\times\mathbb{R}.
 \end{equation*}
When ${\bf M}\equiv 0$,
 i.e. there is no mutation, problem \eqref{Liu1} can be decoupled as the following $n$ scalar  time-periodic parabolic problems:
\begin{equation}\label{Liu-20240719-1}
 \begin{cases}
 \begin{array}{ll}
 \omega\partial_t \varphi-\rho d_i\Delta\varphi-c_i(x,t)\varphi=\lambda\varphi &\mathrm{in}~\Omega\times \mathbb{R},\,\,i=1,\cdots,n ,\\
   \nabla \varphi\cdot\nu=0 &\mathrm{on}~\partial\Omega\times \mathbb{R},\\
   \varphi(x,t)=\varphi(x,t+1) &\mathrm{in}~\Omega\times \mathbb{R},
   \end{array}
  \end{cases}
  \end{equation}
  which admits the unique principal eigenvalue, denoted by $\lambda_i(\omega,\rho)$, for each $i=1,\cdots,n$;
 see \cite[Proposition 16.1]{Hess1991}.
The outcomes of phenotypes in \eqref{liu-20240724-1} are independent in this case,  and phenotype $i$ is  able to persist if and only if $\lambda_i(\omega,\rho)<0$.
This motivates extensive researches  on  scalar problem  \eqref{Liu-20240719-1} in order to clarify the effect of spatio-temporal heterogeneity on the persistence of phenotypes,
 mainly focusing on the   asymptotic behaviors of
$\lambda_i(\omega,\rho)$ with respect to frequency $\omega$
  and diffusion rate $\rho$; see e.g.  \cite{Hess1991,HSV2000,LL2022,LLPZ2019,N2009}. These findings turn out to have wide range of applications in  studies of reaction-diffusion equations \cite{BHN2022,Hess1991,Hutson2001,PZ2015}.

A natural question is 
how the spatio-temporal heterogeneity may affect the dynamics of \eqref{liu-20240724-1}  when the mutation exists.
We assume ${\bf M}$ is a time-periodic and fully coupled matrix function. 
In the simplest
case, ${\bf M}$ is a discrete Laplacian, characterized as a symmetric and essentially positive constant matrix.  
Then all phenotypes will exhibit the same outcome and their persistence can be determined by the principal eigenvalue $\lambda(\omega,\rho)$ of  coupled problem \eqref{Liu1}.
   To 
    address the question, we are motivated to investigate the 
   qualitative properties of $\lambda(\omega,\rho)$ as a function of  $\omega$
  and  $\rho$.  For
technical reasons we  specifically focus on the case of symmetric mutation, where matrix ${\bf M}$ (and thus ${\bf A}$) is assumed to be symmetric 
 throughout this paper.  We left the general cases for further studies.

When
matrix ${\bf A}
$ is  independent of $x$  variable, problem \eqref{Liu1} can  be reduced to a time-periodic system in  patchy environment \cite{MG2007,RHB2005}  and   the corresponding principal eigenvalue  
is independent of the diffusion rate $\rho$.  
The monotonicity of principal eigenvalue  with respect to frequency $\omega$ was established in \cite{LLS2022} and subsequently applied in \cite{K2022} to investigate the dispersal-induced
growth phenomenon\footnote{This phenomenon, which is of particular interest, occurs when
populations, that would become extinct when either isolated or well mixed, are able to persist
by dispersing in the habitats.}. We refer to
 \cite{BLSS2024,JY1988,K2022} for details.
In contrast,
when matrix ${\bf A}
$ is time-independent, problem \eqref{Liu1} becomes a  cooperative elliptic problem \cite{S1992} 
and  principal eigenvalue  
remains unaffected by frequency $\omega$.  The asymptotic limits of principal eigenvalue 
as diffusion rate
$\rho$ approaches zero  was  proved by Dancer \cite{D2009} and 
was extended by Lam and Lou in \cite{LL2016} to more general cases,  with applications
to 
nonlinear competition population models. 

When matrix ${\bf A}
$ depends on both variables $x$ and $t$, 
much less  is known about the dependence of principal eigenvalue on parameters.
A few existing results can be found in \cite{BH2020,GM2021,ZZ2021}, which
  generalized the results in  \cite{D2009,LL2016} and established the asymptotic behaviors of principal eigenvalue for small or large diffusion rate.   Among  others, the following result holds. 


\begin{theorem}[\cite{BH2020}]\label{Bai-He-2020}
    Let
$\lambda(\omega, \rho)$ be the principal eigenvalue of  \eqref{Liu1}, then for each $\omega>0$,
\begin{equation*}
\lim_{\rho\to 0}\lambda(\omega,\rho)=\underline{h}(\omega):
=\min_{x\in\overline{\Omega}} h(x,\omega), 
\end{equation*}
where for any  $x\in \Omega$ and $\omega>0$, $h(x,\omega)$ denotes the principal eigenvalue of the problem
\begin{equation}\label{liu-20240317-1}
    \omega\frac{{\rm d}{\bm\phi}}{{\rm d} t}-{\bf A}(x,t){\bm\phi}= h{\bm\phi},
\qquad
  {\bm\phi}(t)={\bm\phi}(t+1), \,\,\,\,t\in\mathbb{R}.
\end{equation}
\end{theorem}

In the present
paper,  we will establish some monotonicity and  asymptotic behaviors of the principal eigenvalue
for problem \eqref{Liu1}. They lead to the complete classification for the topological structures of its level sets,
as a function of frequency $\omega$ and diffusion rate $\rho$.  This will help us
better understand the  combined effects of $\omega$ and $\rho$ on the principal eigenvalue. 




\subsection{
Main results I: asymptotics and monotonicity
}\label{subsection1.2} 
We define
\begin{equation}\label{def_underlineC}
    \underline{C}:=-\int_0^1 \max_{x\in \overline{\Omega}}\mu({\bf A}(x,t)){\rm d}t \quad\text{and}\quad C_*:=-\max_{x\in \overline{\Omega}}\int_0^1\mu({\bf A}(x,t)){\rm d}t, 
\end{equation}
where $\mu({\bf A}(x,t))$ denotes the maximal eigenvalue of matrix ${\bf A}(x,t)$ associated with a nonnegative eigenvector.
By
        Theorem \ref{Bai-He-2020} and \cite[Theorem 2.1]{LLS2022},
 it follows that

       $$ \lim_{\omega\to 0} \lim_{\rho\to 0} \lambda(\omega,\rho)=\lim_{\omega\to 0}\underline{h}(\omega)=C_*.
$$
In contrast,
       Proposition \ref{TH-liu-20240227} and \cite[Theorem 1]{D2009} imply that
       $$ \lim_{\rho\to 0}\lim_{\omega\to 0}\lambda(\omega,\rho)=\lim_{\rho\to 0}\underline{\lambda}(\rho)=\underline{C},
$$
where $\underline{\lambda}(\rho)$ is defined in Proposition \ref{TH-liu-20240227} to serve
as the limit of $\lambda(\omega,\rho)$ as $\omega\to 0$.
        This implies that the double limit of $\lambda(\omega,\rho)$ as $(\omega,\rho)\to (0,0)$  does  not exist  once $\underline{C}\neq C_*$.
This motivates us to 
consider the asymptotic behaviors of $\lambda(\omega,\rho)$ near $(\omega,\rho)=(0,0)$
 and  study 
 the transition between the two  regimes  described above.

\begin{theorem}\label{TH-liu-20240106}
Let
$\lambda(\omega, \rho)$ denote the principal eigenvalue of problem \eqref{Liu1}.
Then 
$$\lim_{(\omega,\rho)\to (0,0) \atop \frac{\omega}{\sqrt{\rho}}\to 0}\lambda(\omega, \rho)=\underline{C} 
\quad\text{and}\quad
\lim_{(\omega,\rho)\to (0,0) \atop \frac{\omega}{\sqrt{\rho}}\to +\infty}\lambda(\omega, \rho)=C_*. 
$$

Furthermore, for any ${\bf p}\in\mathbb{R}^n$, let  ${\rm H} ({\bf p},x,t)\in \mathbb{R}$ be the  maximal eigenvalue of matrix
$${\rm diag}\{d_1|{\bf p}|^2,\cdots, d_n|{\bf p}|^2\}+{\bf A}(x,t)$$ 
associated with a nonnegative eigenvector. 
Then there holds
$$\lim_{(\omega,\rho)\to (0,0) \atop \frac{\omega}{\sqrt{\rho}}\to \vartheta}\lambda(\omega, \rho)=C(\vartheta),$$
whereas $C(\vartheta)$ is the unique value for which the following time-periodic Hamilton-Jacobi
equation admits a Lipschitz viscosity solution:
\begin{equation}\label{Liu01-20210716}
 \begin{cases}
\begin{array}{ll}
\vartheta\partial_t U+{\rm H} (\nabla U,x,t)=-C(\vartheta) &\mathrm{in}~\Omega\times \mathbb{R},\\
  \nabla U\cdot\nu =0 &\mathrm{on}~\partial\Omega\times \mathbb{R},\\
  U(x,t)=U(x,t+1) &\mathrm{in}~\Omega\times \mathbb{R}.
  \end{array}
 \end{cases}
 \end{equation}
 Moreover,
 $C(\vartheta)$ is continuous 
 in $\vartheta$, $C(\vartheta)\to \underline{C}$ as $\vartheta\to 0$, and $C(\vartheta)
\to C_*$ as $\vartheta\to+\infty$.
 \end{theorem}


Theorem \ref{TH-liu-20240106} indicates that the transition 
of various asymptotic behaviors for principal eigenvalue  $\lambda(\omega,\rho)$ in the regime $(\omega,\rho)\to(0,0)$ 
 occurs at $\omega\approx \sqrt{\rho}$. This is connected by 
 the time-periodic Hamilton-Jacobi equation \eqref{Liu01-20210716} with some convex and coercive Hamiltonian  that represents the maximal eigenvalue of some nonnegative matrix. It is somewhat interesting and surprising that the limiting behaviors of 
 system \eqref{Liu1} can be determined by a scalar Hamilton-Jacobi equation.
 This is a generalization of  \cite[Theorem 1.2]{LL2022} for 
 periodic-parabolic operators and the proof is
much more involved for the present case of systems.

 The problem of  identifying a pair $(C(\vartheta), U_\vartheta)
 $ for which $U_\vartheta$ is a viscosity solution of \eqref{Liu01-20210716} is known as an additive eigenvalue problem.
 This is related to  weak KAM theory \cite{I2011,T2021} and homogenization theory \cite{EG2001,LPV1988}, both of which have extensive applications in the asymptotic propagation \cite{BN2022,ES1989}  and the evolution of dispersal \cite{LLP2022} for reaction diffusion equations.   Problem \eqref{Liu01-20210716} has a viscosity solution only if $C(\vartheta)$ is assigned a specific value, for which the uniqueness was established in the  pioneering work \cite{LPV1988}. 
Our next result  is 
to provide more insight on the connections between the principal eigenvalue and such critical value.


\begin{theorem}\label{liutheorem-0229-1}
Let $C(\omega/\sqrt{\rho})$ be the critical value of 
\eqref{Liu01-20210716} with $\vartheta=\omega/\sqrt{\rho}$.
Then  $\lambda(\omega,\rho)\geq C(\omega/\sqrt{\rho})$ for all $\omega,\rho>0$.
Moreover,   the  critical value $C(\vartheta)$ 
is non-decreasing in $\vartheta>0$.
\end{theorem}

The inequality in Theorem \ref{liutheorem-0229-1} is a reduced version of  Theorem \ref{liutheorem-0229},  in which more precise estimation is provided based on the associated  eigenfunctions. The monotonicity in Theorem \ref{liutheorem-0229-1} is potentially of interest in understanding the effect of temporal heterogeneity on ``effect Hamiltonians" as studied in \cite{EG2001,EG2002}, where the asymptotic behaviors of $C(\vartheta)$ for more general time-periodic Hamiltonian  in \eqref{Liu01-20210716} are discussed.

A corollary of Theorem \ref{liutheorem-0229-1} is that $\lambda(\vartheta\sqrt{\rho},\rho)\geq C(\vartheta)$ for all $\rho,\vartheta>0$. In view of $\lambda(\vartheta\sqrt{\rho},\rho)\to C(\vartheta)$ as $\rho\to 0$ as shown in Theorem \ref{TH-liu-20240106},
this implies  $\lambda(\vartheta\sqrt{\rho},\rho)$ attains its global minimal at $\rho=0$. A natural conjecture is that $\lambda(\vartheta\sqrt{\rho},\rho)$ is monotone increasing in $\rho$ for any $\vartheta>0$.
The following result gives a positive answer.

%

\begin{theorem}\label{TH1-1}
Let $\lambda(\omega,\rho)$ denote the principal eigenvalue of \eqref{Liu1}.
Suppose that $\rho=\rho(\omega)\in C^1((0,\infty))$,
$\rho'(\omega)\geq 0$, and $\left[\rho(\omega)/\omega^2\right]'\leq 0$ 
in \eqref{Liu1}. Then $\lambda(\omega,\rho(\omega))$ is non-decreasing in $\omega$.   
In particular, for each fixed $\rho>0$, $\omega\mapsto\lambda(\omega, \rho)$ is non-decreasing. 
\end{theorem}

Theorem \ref{TH1-1} generalizes the results presented in \cite[Theorem 1.1]{LL2022} for the scalar periodic-parabolic case and in \cite[Theorem 1.1]{LLS2022} for the spatially homogeneous case. Our proof offers a more straightforward and simplified approach in contrast to those in \cite{LL2022,LLS2022}, based on the equality provided in Lemma \ref{L2}. The necessary and sufficient condition on the strict monotonicity is detailed  in Theorem \ref{TH1}.  We refer to \cite{FLRX2024} for a related result on the
time-periodic nonlocal dispersal cooperative systems. However,  the question of  monotonicity in  situations where matrix ${\bf A}$ is not necessarily symmetric remains open.




  \subsection{\bf 
  Main results II: level sets and applications} 
As applications of
the above results, we are able to  
provide a complete classification of
the level sets for  principal eigenvalue of problem \eqref{Liu1} as a function of frequency $\omega$  and diffusion rate $\rho$  in Theorem \ref{liu-levelset}.
It identifies  five distinct types of topological structures for the level sets, which   is a more intricate finding compared to the results in 
\cite[Theorem 1.3]{LL2022} concerning the scalar time-periodic parabolic operators, where only two different topological structures were found. An interesting consequence of
 Theorem \ref{liu-levelset} is the implication on the non-monotone dependence of principal eigenvalue
 on the diffusion rate $\rho$ as shown in Corollary \ref{liucor-1}. This is in strong contrast to the time-independent scenario,  wherein the principal eigenvalue is non-decreasing in 
 $\rho$. To shorten the introduction, we  move the materials for Theorem \ref{liu-levelset} to section \ref{Sect.5} and refer to Fig. \ref{liufig2} for illustrations.

In what follows, we shall apply Theorem \ref{liu-levelset} to system \eqref{liu-20240724-1} and characterize the parameter regions for the persistence and
extinction of  phenotypes. Assume the time-periodic mutation matrix ${\bf M}$ is  essentially positive and fully coupled, then all phenotypes will interact with each other and the persistence of each phenotype  is completely determined by the principal eigenvalue $\lambda(\omega,\rho)$ of system \eqref{Liu1} with ${\bf A}={\bf M}+{\rm diag}(c_i)$. In particular, 
the persistence region for problem \eqref{liu-20240724-1} can be defined as 
 $$
 \mathbb{E}:=\{(\omega,\rho)\in\mathbb{R}_+^2: \,\,
 \lambda(\omega,\rho)<0\},
 $$
 where $\mathbb{R}_+=(0,+\infty)$. This means that all phenotypes of the species will persist when $(\omega,\rho)\in \mathbb{E}$, whereas all phenotypes will become extinction  when $(\omega,\rho)\not\in \mathbb{E}$.
As a direct consequence of Theorem \ref{liu-levelset}, the persistence region  $\mathbb{E}$ can be characterized
 as follows.
\begin{corollary}\label{liu-levelset-1}
Let $\underline{C}\leq C_*\leq \underline{C}^+,C_*^+\leq \overline{C}$ be defined in \eqref{def_underlineC} and \eqref{def_underlineC2} with ${\bf A}= {\bf M}+{\rm diag}(c_i)$ (see also Lemma {\rm \ref{liu-20240526}}). Then $\mathbb{E}=\emptyset$ if $\underline{C}\geq 0$, and $\mathbb{E}=\mathbb{R}_+^2$ if $\overline{C}\leq 0$. Otherwise, if $\underline{C}<0<\overline{C}$, then
there exists uniquely a continuous function $\overline\omega:{\rm dom}(\overline\omega)\mapsto (0,+\infty)$ such that
$$\mathbb{E}=\{(\omega,\rho)\in \mathbb{R}_+^2: \, \omega<\overline\omega(\rho), \,\, \,\rho\in {\rm dom}(\overline\omega) \}.$$
 The domain ${\rm dom}(\overline\omega)$ and
  asymptotic behaviors of function $\overline\omega$  can be characterized as follows.


\begin{itemize}
    \item[{\rm(1)}] If
    $\underline{C}<0<C_*$, then
    $ {\rm dom}(\overline\omega)= (0,\overline\rho)$ for some $\overline\rho>0$, and
 $$\overline\omega(\rho)\to 0 \quad\text{as } \,\, \rho\searrow 0 \text{ and } \rho\nearrow \overline{\rho}.$$


\item[{\rm(2)}] If $C_*<0<\min\{C_*^+,\underline{C}^+\})$, then
    $ {\rm dom}(\overline\omega)= (0,\overline\rho)$ for some $\overline\rho>0$, and 
$$\overline\omega(\rho)\to \underline{h}^{-1}(0)>0 \text{ as } \rho\searrow 0 \quad \text{ and } \quad \overline\omega(\rho)\to 0 \text{ as } \rho\nearrow \overline{\rho}.$$

\item [{\rm(3)}]  If $C_*^+<0<\underline{C}^+$, then
$ {\rm dom}(\overline\omega)=(\underline{\rho},\overline\rho)$ for some $0<\underline{\rho}<\overline\rho$, and
$$\overline\omega(\rho)\to +\infty \text{ as } \rho\searrow \underline{\rho} \quad \text{ and } \quad \overline\omega(\rho)\to 0 \text{ as } \rho\nearrow \overline{\rho}.$$

\item [{\rm(4)}]   If $\underline{C}^+<0<C_*^+$,then $ {\rm dom}(\overline\omega)= (0,+\infty)$, and
$$\overline\omega(\rho)\to \underline{h}^{-1}(0)>0 \text{ as } \rho\searrow 0 \quad \text{ and } \quad \overline\omega(\rho)\to \overline{h}^{-1}(0)>0  \text{ as } \rho\nearrow +\infty.$$

\item[{\rm (5)}] If $\max\{C_*^+,\underline{C}^+\}<0< \overline{C}$, then $ {\rm dom}(\overline\omega)=(\underline{\rho}, +\infty)$, and
$$\overline\omega(\rho)\to +\infty \text{ as } \rho\searrow \underline{\rho} \quad \text{ and } \quad \overline\omega(\rho)\to \overline{h}^{-1}(0)>0 \text{ as } \rho\nearrow +\infty.$$
\end{itemize}
Here  functions $\underline{h}(\omega)$ and $\overline{h}(\omega)$ are defined in Theorem {\rm\ref{Bai-He-2020}} and Proposition {\rm\ref{TH-liu-20240429}}, respectively, both of which are non-decreasing and satisfy $\overline{h}^{-1}(0)<\underline{h}^{-1}(0)$.
\end{corollary}

 \begin{figure}[htb]
    \centering
    \includegraphics[width=1.0\linewidth]{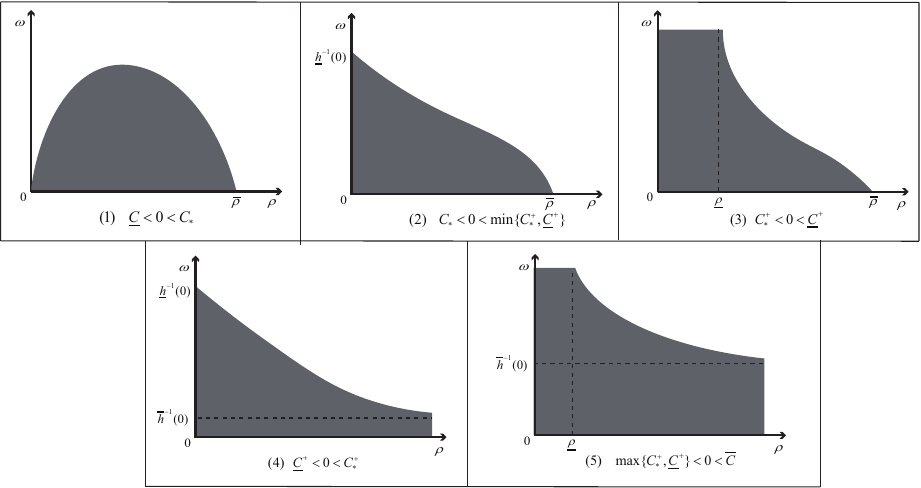}
    \caption{\small
    {Illustrations of the
persistence region $\mathbb{E}$ in the $\rho$-$\omega$ plane  for the essentially positive and fully coupled mutation matrix ${\bf M}$,
which is
marked by the shaded areas, while the blank areas correspond to the region where $\lambda(\omega,\rho)\geq 0$.
    } }
  \label{liufig-2}
   \end{figure}

 Corollary \ref{liu-levelset-1} suggests  that the persistence region $\mathbb{E}$ may exhibit five distinct types of  topological structures in the $\rho$-$\omega$ plane,  according to the combined effects of mutation matrix ${\bf M}$ and birth-death rate $c_i$.
These structures are illustrated by the shaded areas in Fig. \ref{liufig-2}, which are  separated with
blank areas  by the level set $\lambda(\omega,\rho)=0$ as characterized by Theorem \ref{liu-levelset}.  This in fact suggests much more complicated dynamics of \eqref{liu-20240724-1} compared to the scenario where ${\bf M}\equiv 0$,
indicating the absence of mutations in \eqref{liu-20240724-1}. Indeed, as aforementioned, when ${\bf M}\equiv 0$, phenotype $i$ can persist if and only if the principal eigenvalue $\lambda_i(\omega,\rho)$ for scalar problem \eqref{Liu-20240719-1} is of negative sign. According to the characterization of the level sets for $\lambda_i(\omega,\rho)$
in  \cite[Theorem 1.3]{LL2022},
only two distinct types of topological structures, specifically types (1) and (3) in Fig. \ref{liufig-2}, are observed for the persistence region $\mathbb{E}$ under different conditions on  the birth-death rate  $c_i$.
This finding suggests that our results in the present paper reveal some new phenomena,
which cannot be observed  from the decoupled system  without mutations. Therefore, the analysis of the principal eigenvalue for system \eqref{Liu1} can capture some intricate effects of mutations on the persistence of a species.

\subsection{Organization of the paper}
In section \ref{Sect.2}
we present some preliminary results on  asymptotic behaviors of the principal eigenvalue for large/small frequency or diffusion rate.
Section \ref{Sect.3} is devoted to  the asymptotic
analysis of the principal eigenvalue
near the origin  and establishing Theorem \ref{TH-liu-20240106}.
The monotone dependence of the principal eigenvalue   is studied in section \ref{monotonicity-section}, where Theorems \ref{liutheorem-0229-1} and \ref{TH1-1} are proved. Finally,  in section \ref{Sect.5} we classify the topological
structures of level sets for the principal eigenvalue and establish some non-monotone dependence of principal eigenvalue on diffusion rate.

\medskip

\section{\bf Preliminaries}\label{Sect.2}
In this section, we prepare some asymptotic behaviors of principal eigenvalue for problem \eqref{Liu1}, which generalize some results on the scalar time-periodic parabolic eigenvalue problems as studied in \cite{Hutson2001,HSV2000,LLPZ2019,N2009}.
 Throughout the paper, for each $(x,t)\in\overline{\Omega}\times\mathbb{R}$,
we assume  ${\bf A}(x,t)=(a_{ij}(x,t))_{n\times n}$ is a
symmetric and essentially positive matrix, with each  $a_{ij}\in C(\overline{\Omega}\times\mathbb{R})$
being a time-periodic function with unit period.

The symmetric assumption on matrix ${\bf A}$ is not necessary throughout this section.

\begin{proposition}\label{TH-liu-20240227}
Let $\lambda(\omega,\rho)$ be the principal eigenvalue of problem \eqref{Liu1}. Then 
\begin{itemize}
    \item[{\rm(i)}] $\lambda(\omega,\rho)\to \overline{\lambda}(\rho)$ as $\omega\to+\infty$, where $\overline{\lambda}(\rho)$ denotes the principal eigenvalue of 
\begin{equation}\label{liu-20240227-1}
-\rho {\bf D}\Delta {\bm \phi}-\widehat{\bf A}(x){\bm \phi}=\lambda{\bm \phi} \,\,\,\,\mathrm{in}~\,\,\,\Omega,\qquad
  \nabla  {\bm \phi}\cdot\nu =0~~\mathrm{on}\,\,\,~\partial\Omega,
 \end{equation}
 with $\widehat{\bf A}$ being the temporally averaged matrix with entries $(\widehat{\bf A})_{ij}(x)=\int_0^1 a_{ij}(x,t){\rm d}t$.
    \item[{\rm (ii)}] $\lambda(\omega,\rho)\to \underline{\lambda}(\rho):=\int_0^1 \lambda_0(t,\rho){\rm d}t$ as $\omega\to0$, where for any $\rho>0$ and $t\in\mathbb{R}$, $\lambda_0(t,\rho)$ denotes the principal eigenvalue of the elliptic eigenvalue problem
\begin{equation}\label{liu-20240227-2}
-\rho  {\bf D}\Delta {\bm \phi}-{\bf A}(x,t){\bm \phi}=\lambda{\bm \phi} \,\,\,\,\mathrm{in}~\,\,\,\Omega,\qquad
  \nabla  {\bm \phi}\cdot\nu =0~~\mathrm{on}\,\,\,~\partial\Omega.
 \end{equation}
\end{itemize}
\end{proposition}

\begin{proof}
{\it Step 1}. We first prove part {\rm (i)}.   The proof can follow by the ideas in  \cite[Lemma 3.10]{N2009} concerning 
scalar space-time periodic eigenvalue problems. We present the details for completeness. Let ${\bm \varphi}=(\varphi_1,\cdots,\varphi_n)>0$ be the principal eigenfunction of \eqref{Liu1} associated with $\lambda(\omega,\rho)$, which can be normalized by $\sum_{i=1}^n \int_0^1\!\int_\Omega\varphi^2_i=1$. Recall that ${\bf A}(x,t)=(a_{ij}(x,t))_{n\times n}$ with $1$-periodic function $a_{ij}\in C(\Omega\times\mathbb{R})$.  Multiplying both sides of \eqref{Liu1} by $\varphi_i$ and integrating the resulting equation over $\Omega\times(0,1)$, we have
   \begin{align*}
    \rho \sum_{i=1}^n d_i\int_0^1\!\!\!\int_\Omega |\nabla\varphi_i|^2=\lambda(\omega,\rho)+\sum_{i,j=1}^n \int_0^1\!\!\!\int_\Omega a_{ij}\varphi_i\varphi_j. 
   \end{align*}
   Hence, by the uniform boundedness of $\lambda(\omega,\rho)$ we derive that
   \begin{equation}\label{liu-20240315-1}
       \sum_{i=1}^n\int_0^1\!\!\!\int_\Omega |\nabla\varphi_i|^2\leq M/\rho, \quad \forall \omega,\rho>0.
   \end{equation}
   Hereafter $M$  denotes some  positive constant, which may vary from line to line but is always
   independent of $\omega$ and $\rho$. 
We next multiply both sides of \eqref{Liu1} by $\partial_t\varphi_i$ and integrate the resulting equation over $\Omega\times(0,1)$, then 
   \begin{align*}
    \omega\sum_{i=1}^n \int_0^1\!\!\!\int_\Omega |\partial_t \varphi_i|^2=-\frac{\rho}{2}\sum_{i=1}^n d_i\int_0^1\!\!\!\int_\Omega \partial_t(|\nabla \varphi_i|^2)+\sum_{i,j=1}^n \int_0^1\!\!\!\int_\Omega  a_{ij}\varphi_j \partial_t \varphi_i=\sum_{i,j=1}^n \int_0^1\!\!\!\int_\Omega a_{ij}\varphi_j \partial_t \varphi_i,
   \end{align*}
   which yields the following estimate
   \begin{equation}\label{liu-20240315-2}
     \sum_{i=1}^n \int_0^1\!\!\!\int_\Omega |\partial_t \varphi_i|^2\leq M/\omega, \quad \forall \omega,\rho>0.
   \end{equation}

 Combining \eqref{liu-20240315-1} and \eqref{liu-20240315-2},  we observe that for each $i=1,\cdots, n$, function $\varphi_i$ is bounded in $H^1(\Omega\times(0,1))$ and $\partial_t \varphi_i \to 0$ in $L^2(\Omega\times(0,1))$ as $\omega\to +\infty$. Hence, by passing to a subsequence if necessary, we may assume that $\varphi_i \to \phi_i$ strongly in $L^2(\Omega\times(0,1))$ and weakly in $H^1(\Omega\times(0,1))$ for some $\phi_i\in H^1(\Omega\times(0,1))$ satisfying $\sum_{i=1}^n \int_0^1\!\!\int_\Omega\phi^2_i=1$. Then it follows by \eqref{liu-20240315-2} that
 $$ \sum_{i=1}^n \int_0^1\!\!\!\int_\Omega |\partial_t \phi_i|^2\leq \liminf_{\omega\to 0}\sum_{i=1}^n \int_0^1\!\!\!\int_\Omega |\partial_t \varphi_i|^2=0,$$
 and thus $\phi_i=\phi_i(x)$ is  independent of $t$ variable.

 Assume that $\lambda(\omega,\rho)\to \overline{\lambda}(\rho)$ for some $\overline{\lambda}(\rho)\in\mathbb{R}$ as $\omega\to+\infty$ by passing to a subsequence if necessary. It suffices to show that $\overline{\lambda}(\rho)$ is the principal eigenvalue of  \eqref{liu-20240227-1}. To this end,  multiply both sides of \eqref{Liu1} by any test function in $H^1(\Omega\times(0,1))$ and
 let $\omega\to+\infty$ in  the resulting equation,  then we see that 
 ${\bm \phi}=(\phi_1,\cdots,\phi_n)$ is a weak solution  of  the equation
 \begin{equation}\label{liu-20240531-1}
     -\rho {\bf D}\Delta{\bm \phi}-{\bf A}(x,t){\bm \phi}=\overline{\lambda}(\rho){\bm \phi},\quad  \,\,\forall (x,t)\in \Omega\times\mathbb{R}.
 \end{equation}
 Due to ${\bm \phi}={\bm \phi}(x)$ is independent of $t$,
 we integrate \eqref{liu-20240531-1} with respect to $t$ over $(0,1)$, 
 then
$$-\rho {\bf D}\Delta{\bm \phi}-\widehat{\bf A}(x){\bm \phi}=\overline{\lambda}(\rho){\bm \phi},\quad  \,\,\forall (x,t)\in \Omega\times\mathbb{R}.$$
In view of ${\bm \phi}>0$, by the uniqueness of principal eigenvalue of problem \eqref{liu-20240227-1}, we conclude the desired result.  Step 1 is complete.

{\it Step 2.} We prove part {\rm(ii)}. For each $t\in \mathbb{R}$, denote by ${\bm \phi}=(\phi_1,\cdots,\phi_n)>0$  the principal eigenfunction of \eqref{liu-20240227-2}  corresponding to principal eigenvalue $\lambda_0(t,\rho)$. 
Set
$$\underline{\bm \varphi}(x,t):=\exp\left\{\frac{1}{\omega}\left(\underline{\lambda}(\rho) t-\int_0^t \lambda_0(s,\rho){\rm d}s\right)\right\}{\bm \phi}(x,t),$$
which is positive and $1$-periodic in $t$. Set $\underline{\bm \varphi}:=(\underline{\varphi}_1,\cdots,\underline{\varphi}_n)$ for $1$-periodic function $\underline{\varphi}_i\in C^{2,1}(\Omega\times\mathbb{R})$. 
By \eqref{liu-20240227-2}, direct calculations yield $\nabla\underline{\varphi}_i \cdot\nu =\underline{\varphi}_i \nabla \log\phi_i\cdot \nu=0$ for all $x\in \partial\Omega$ and 
\begin{align*}\label{liu-20240325-1}
    &\omega\partial_t\underline{\varphi}_i-\rho d_i\Delta\underline{\varphi}_i-\sum_{j=1}^n a_{ij}\underline{\varphi}_j\nonumber\\
    =&\Big[\underline{\lambda}(\rho)-\lambda_0(t,\rho)+\omega \frac{\partial_t \phi_i}{\phi_i}-\rho d_i\frac{\Delta \phi_i}{\phi_i}-\sum_{j=1}^n a_{ij}\frac{\phi_j}{\phi_i}\Big]\underline{\varphi}_i\\
     =&\left(\underline{\lambda}(\rho)+\omega \partial_t \log \phi_i\right)\underline{\varphi}_i, \qquad \forall (x,t)\in\Omega\times\mathbb{R},\,\,\,i=1,\cdots,n.\nonumber
\end{align*}
Hence, 
for any given $\epsilon>0$, we can choose $\omega$ small such that
\begin{equation}\label{fr4}
\left\{
\begin{array}{ll}
\medskip
\omega\partial_t\underline{\bm \varphi}-\rho {\bf D}\Delta\underline{\bm \varphi}-{\bf A}\underline{\bm \varphi}\leq(\underline{\lambda}(\rho)+\epsilon)\underline{\bm \varphi} &\mathrm{in}~\Omega\times\mathbb{R},\\
\medskip
\omega\partial_t\underline{\bm \varphi}-\rho {\bf D}\Delta\underline{\bm \varphi}-{\bf A}\underline{\bm \varphi}\geq(\underline{\lambda}(\rho)-\epsilon)\underline{\bm \varphi}
&\mathrm{in}~\Omega\times\mathbb{R},\\
\medskip
  \nabla \underline{\bm \varphi}\cdot\nu =0 & \mathrm{on}~\partial\Omega\times\mathbb{R},\\
  \underline{\bm \varphi}(x,t)=\underline{\bm \varphi}(x,t+1), &\mathrm{in}~\Omega\times\mathbb{R}.
  \end{array}
  \right.
\end{equation}
 Then we apply the comparison principle
established by \cite[section 2]{BH2020} to \eqref{fr4} and obtain 
\begin{equation*}
 \underline{\lambda}(\rho)-\epsilon\leq \liminf_{\omega\rightarrow0}\lambda(\omega,\rho)\leq \limsup_{\omega\rightarrow0}\lambda(\omega,\rho)\leq \underline{\lambda}(\rho)+\epsilon,
\end{equation*}
which proves part {\rm (ii)} by letting $\epsilon\to 0$. The proof is complete.
\end{proof}

As a complement to  Theorem \ref{Bai-He-2020}, we next  consider the asymptotic behavior of principal eigenvalue $\lambda(\omega,\rho)$ for large diffusion rate $\rho$.
\begin{proposition}\label{TH-liu-20240429}
For each $\omega>0$, let $\overline{h}(\omega)$ be the principal eigenvalue of the problem
\begin{equation}\label{liu-20240525-3}
    \omega\frac{{\rm d}{\bm\phi}}{{\rm d} t}-\overline{\bf A}(t){\bm\phi}= h{\bm\phi},
\qquad
  {\bm\phi}(t)={\bm\phi}(t+1), \,\,\,\,t\in\mathbb{R},
\end{equation}
where $\overline{\bf A}$ is the  spatially averaged matrix with entries $(\overline{\bf A})_{ij}(t)=\frac{1}{|\Omega|}\int_\Omega a_{ij}(x,t){\rm d}x$.
Then 
$$\lim_{\rho\to +\infty}\lambda(\omega,\rho)=\overline{h}(\omega). 
$$
  Moreover, if the off-diagonal entries of ${\bf A}$ are independent of $x$  variable, then $\lambda(\omega,\rho)\leq \overline{h}(\omega)$ for all $\omega,\rho>0$, i.e. $\lambda(\omega,\rho)$ attains its global maximal at $\rho=+\infty$ for each $\omega>0$.
\end{proposition}
\begin{proof}
{\it Step 1}. We first establish the limit of $\lambda(\omega,\rho)$ as $\rho\to+\infty$.    Let ${\bm \varphi}=(\varphi_1,\cdots,\varphi_n)>0$, normalized by $\sum_{i=1}^n \int_0^1\!\int_\Omega\varphi^2_i=1$, be the principal eigenfunction of problem \eqref{Liu1}. Set
    \begin{equation}\label{liu-20240509-1}
       f_i(x,t):=\varphi_i(x,t)-\frac{1}{|\Omega|}\int_\Omega \varphi_i(x,t){\rm d}x, \qquad i=1,\cdots,n.
    \end{equation}
    Then $\int_\Omega f_i(x,t){\rm d}x=0$ for all $t\in\mathbb{R}$ and $1\leq i\leq n$. 
    By the Poincar\'e’s inequality,  there exists some constant
$M>0$ depending only on $\Omega$ such that $\int_\Omega f_i^2 \leq M  \int_\Omega |\nabla f_i|^2$. Note from \eqref{liu-20240509-1} that $\nabla f_i=\nabla \varphi_i$, whence by  \eqref{liu-20240315-1} we derive that
\begin{equation}\label{liu-20240509-2}
       \sum_{i=1}^n\int_0^1\!\!\!\int_\Omega f_i^2\leq M/\rho, \qquad i=1,\cdots,n.
   \end{equation}

Integrate \eqref{Liu1} over $\Omega$ and substitute $\varphi_i=f_i+\frac{1}{|\Omega|}\int_\Omega \varphi_i$ into the resulting equation, then 
\begin{equation}\label{liu-20240509-3}
    \omega \partial_t \int_\Omega \varphi_i-\sum_{j=1}^n (\overline{\bf A})_{ij}(t)\int_\Omega \varphi_j=\lambda(\omega,\rho)\int_\Omega \varphi_i+\sum_{j=1}^n \int_\Omega a_{ij} f_j, \quad i=1,\cdots,n.
\end{equation}
Multiplying both sides of \eqref{liu-20240509-3} by $\int_\Omega \varphi_i$ and summing over index $i$ from $1$ to $n$, then by \eqref{liu-20240509-2} and  the boundedness of matrix $\overline{\bf A}$, 
we derive that
\begin{equation}\label{liu-20240513-1}
    \frac{\omega}{2} \partial_t\left[\sum_{i=1}^n(\int_\Omega \varphi_i)^2\right]\leq(M+\lambda(\omega,\rho))\sum_{i=1}^n(\int_\Omega \varphi_i)^2+O(\rho).
\end{equation}

The adjoint problem of \eqref{liu-20240525-3} can be written as
\begin{equation}\label{liu-20240509-6}
    -\omega\frac{{\rm d}{\bm\phi}}{{\rm d} t}-\overline{\bf A}(t){\bm\phi}= h{\bm\phi},
\qquad
  {\bm\phi}(t)={\bm\phi}(t+1), \,\,\,\,t\in\mathbb{R}.
\end{equation}
Denote by  ${\bm\phi}(t)=(\phi_1(t),\cdots,\phi_n(t))>0$  the principal eigenfunction of \eqref{liu-20240509-6} corresponding to $\overline{h} (\omega)$.
Then multiplying both sides of  \eqref{liu-20240509-3} by ${\bm \phi}$ and
integrating the resulting equation over $(0,1)$ yield
\begin{equation}\label{liu-20240509-4}
(\overline{h} (\omega)-\lambda(\omega,\rho))\sum_{i=1}^n \int_0^1 \left[\phi_i(t) \int_\Omega \varphi_i(x,t){\rm d}x\right]{\rm d}t= \sum_{i,j=1}^n\int_0^1 \left[\phi_i(t)\int_\Omega a_{ij} f_j \right]{\rm d}t.
\end{equation}
Hence, by \eqref{liu-20240509-2} we have either $\lambda(\omega,\rho)\to \overline{h} (\omega)$ or $\int_\Omega \varphi_i(x,t){\rm d}x\to 0$  in $L^1((0,1))$ as $\rho\to+\infty$. The former completes the proof. It suffices to show that the latter is impossible.

Suppose it holds, then we can derive from \eqref{liu-20240513-1} that $\int_\Omega \varphi_i(x,t){\rm d}x\to 0$  uniformly in  $[0,1]$ as $\rho\to+\infty$.  This together with \eqref{liu-20240509-1} and \eqref{liu-20240509-2}
implies that
 $$\lim_{\rho\to+\infty}\int_0^1\!\!\!\int_\Omega\varphi^2_i=0, \qquad i=1,\cdots,n. $$
This contradicts the normalization of ${\bm \varphi}$, so that the latter cannot hold.
Hence, $\lambda(\omega,\rho)\to \overline{h} (\omega)$ as $\rho\to+\infty$, which completes the proof.

{\it Step 2}. We assume that the off-diagonal entries of ${\bf A}$ are independent of $x$ and show $\lambda(\omega,\rho)\leq \overline{h}(\omega)$ for all $\omega,\rho>0$.
Recall that ${\bm \varphi}=(\varphi_1,\cdots,\varphi_n)>0$ is the principal eigenfunction of \eqref{Liu1} associated to $\lambda(\omega,\rho)$. Inspired by \cite{HSV2000}, we define
$$w_i(t):=\exp\left(\frac{1}{|\Omega|}\int_\Omega \log \varphi_i(x,t){\rm d}x\right),\quad i=1,\cdots,n.$$
Recall  ${\bf A}(x,t)=(a_{ij}(x,t))_{n\times n}$. 
By our assumption,  $a_{ij}=a_{ij}(t)$ is independent of $x$ for all $i\neq j$.
By \eqref{Liu1},
direct calculations yield
 \begin{align}\label{liu-20240530-1}
     \omega\frac{{\rm d}\log w_i}{{\rm d}t}=&
     \frac{1}{|\Omega|}\int_\Omega \frac{\omega\partial_t \varphi_i}{\varphi_i} {\rm d}x\notag\\
     =& \lambda(\omega,\rho)+
     \frac{\rho d_i}{|\Omega|}\int_\Omega \frac{\Delta\varphi_i}{\varphi_i}{\rm d}x+\frac{1}{|\Omega|} \sum_{j=1}^n  \int_\Omega \frac{a_{ij}\varphi_j}{\varphi_i} {\rm d}x\notag\\
     =& \lambda(\omega,\rho)+
     \frac{\rho d_i}{|\Omega|}\int_\Omega \frac{|\nabla\varphi_i|^2}{\varphi_i^2}{\rm d}x+\frac{1}{|\Omega|} \sum_{j=1}^n  \int_\Omega \frac{a_{ij}\varphi_j}{\varphi_i} {\rm d}x
     \\
     \geq & \lambda(\omega,\rho)+\frac{1}{|\Omega|} \sum_{j=1}^n  \int_\Omega \frac{a_{ij}\varphi_j}{\varphi_i} {\rm d}x. 
     \notag
 \end{align}
 By the definition of  spatially averaged matrix $\overline{\bf A}$, 
  using the Jensen’s inequality one has
  \begin{align*}
      \frac{1}{|\Omega|} \sum_{j=1}^n  \int_\Omega \frac{a_{ij}\varphi_j}{\varphi_i} {\rm d}x&=(\overline{\bf A})_{ii}+\frac{1}{|\Omega|} \sum_{j\neq i}  a_{ij}(t) \int_\Omega \frac{\varphi_j}{\varphi_i} {\rm d}x\\
      &= (\overline{\bf A})_{ii}+\frac{1}{|\Omega|} \sum_{j\neq i}  a_{ij}(t) \int_\Omega \exp(\log \varphi_j-\log \varphi_i){\rm d}x\\
      &\geq (\overline{\bf A})_{ii}+ \sum_{j\neq i}  a_{ij}(t)  \exp\left[\frac{1}{|\Omega|}\int_\Omega\log \varphi_j{\rm d}x-\frac{1}{|\Omega|}\int_\Omega\log \varphi_i{\rm d}x\right]\\
       &=\sum_{j=1}^n\frac{(\overline{\bf A})_{ij} w_j}{w_i}.
  \end{align*}
This together with \eqref{liu-20240530-1} gives
\begin{align*}
     \omega\frac{{\rm d}w_i}{{\rm d}t} -\sum_{j= 1}^n (\overline{\bf A})_{ij} w_j\geq \lambda(\omega,\rho)w_i \quad \text{ for  }\,\,t\in\mathbb{R},\,\,\,i=1,\cdots,n,
\end{align*}
and thus ${\bf w}=(w_1,\cdots,w_n)$ is a super-solution of problem \eqref{liu-20240525-3}. By comparison (see  Proposition 2.4 of \cite{BH2020}), we derive $\lambda(\omega,\rho)\leq {\overline{h}}(\omega)$ for all $\omega>0$. The proof is complete.
\end{proof}

We conclude this section by stating a direct corollary of Proposition \ref{TH-liu-20240429}. 
\begin{corollary}\label{cor-liu-20240429}
Let ${\bf B}(x)=(b_{ij}(x))_{n\times n}$ be any essentially positive and irreducible time-independent matrix. 
For any $\rho>0$, let $\lambda(\rho)$ be the principal eigenvalue of the elliptic problem
\begin{equation}\label{liu-20240530-2}
-\rho {\bf D}\Delta{\bm \varphi}-{\bf B}(x){\bm \varphi}=\lambda{\bm \varphi} \, \text{ in }\, \Omega,\quad
  \nabla {\bm \varphi}\cdot\nu =0 \, \text{ on }\, \partial\Omega.
 \end{equation}
 Then 
 $\lambda(\rho)\to\mu (\overline{\bf B})$ as $\rho\to +\infty$,
where  $\overline{\bf B}$ is the  matrix with entries $(\overline{\bf B})_{ij}=\frac{1}{|\Omega|}\int_\Omega b_{ij}(x){\rm d}x$.
\end{corollary}
\begin{proof}
 Let ${\bf A}={\bf B}(x)$ be independent of $t$ in \eqref{Liu1}, then problem \eqref{Liu1} is equivalent to \eqref{liu-20240530-2}. Hence, Corollary \ref{cor-liu-20240429} is
 a direct consequence of Proposition \ref{TH-liu-20240429}.
\end{proof}

\section{\bf Asymptotic analysis near the origin}\label{Sect.3}

In this section, we are concerned with the asymptotic behaviors of the principal eigenvalue when both diffusion rate and frequency approach to zero and establish Theorem \ref{TH-liu-20240106}.
To this end, we first consider
the following asymptotic regime. 

\begin{proposition}\label{liuprop-2}
Let
$\lambda(\omega, \rho)$ denote the principal eigenvalue of \eqref{Liu1}.
Then 
$$
\lim_{(\omega,\rho)\to (0,0) \atop \frac{\omega}{\sqrt{\rho}}\to +\infty}\lambda(\omega, \rho)=C_*=-\max_{x\in \overline{\Omega}}\int_0^1\mu({\bf A}(x,t)){\rm d}t.
$$
\end{proposition}

\begin{proof}

Define ${\bm \varphi}=(\varphi_1,\cdots,\varphi_n)>0$, normalized by $\max_{1\leq i\leq n}\max_{\overline\Omega\times\mathbb{R}}\varphi_i=1$,  as the principal eigenvector of \eqref{Liu1} corresponding to $\lambda(\omega,\rho)$. Set  $$W_i(x,t;\omega,\rho):=-\sqrt{\rho} \log \varphi_i\left(x,\frac{\omega t}{\sqrt{\rho}}\right) \quad\text{ for all }(x,t)\in\Omega\times\mathbb{R},\,\, 1\leq i\leq n.$$
Then $\min_{1\leq i\leq n}\min_{\overline\Omega\times\mathbb{R}} W_i(x,t;\omega,\rho)=0$ for all $\omega, \rho>0$. By \eqref{Liu1}, direct calculations yield that for any $1\leq i\leq n$,
\begin{equation}\label{liu-25}
\begin{cases}
\begin{array}{ll}
 \partial_{t}W_i-\sqrt{\rho}d_i\Delta W_i+d_i |\nabla W_i|^2+\displaystyle\sum_{j=1}^n a_{ij}(x,\tfrac{\omega t}{\sqrt{\rho}})e^{\frac{W_i-W_j}{\sqrt{\rho}}}=-\lambda(\omega,\rho)  &\text{in }\,\,\Omega\times \mathbb{R},\\
\nabla W_i\cdot\nu =0 &\text{on }\,\partial\Omega\times \mathbb{R}.
\end{array}
\end{cases}
\end{equation}
Note that $\lambda(\omega,\rho)$ is uniformly bounded for all $\omega,\rho>0$ (see also Lemma \ref{liu-20240516}).
By the same arguments as in \cite[Lemma 2.1]{ES1989}, we may construct suitable super-solutions
and apply the comparison principle to derive that 
\begin{equation}\label{liu-24}
    \max_{1\le i\leq n}\max_{ \overline{\Omega}\times\mathbb{R}}
    |W_i(x,t;\omega,\rho)| \leq C,\,\,
    \quad\forall \rho>0,\,\, \forall \omega/\sqrt{\rho}\geq 1
\end{equation}
for some constant $C$ independent of $\omega$ and $\rho$.
Given any  sequences $\{\omega_k\}_{k=1}^\infty$ and $\{\rho_k\}_{k=1}^\infty$ such that  $\omega_k,\rho_k\to 0$ and $\omega_k/\sqrt{\rho_k}\to+\infty$ as $k\to+\infty$, we assume  $\lambda(\omega_k,\rho_k)\to \lambda_\infty$ for some $\lambda_\infty\in\mathbb{R}$. 
By the arbitrariness of $\{\omega_k\}_{k=1}^\infty$ and $\{\rho_k\}_{k=1}^\infty$, it suffices to show  $\lambda_\infty=C_*$.
The proof is separated  into the following two steps.

\smallskip
{\it Step 1}. We prove $\lambda_\infty\geq C_*$.
Define the following half-relaxed limit as in \cite[Sect. 6]{B2013}:
\begin{equation*}
    U_i(x,t):=\liminf_{k\to+\infty \atop (x',t')\to(x,t)}W_i(x',t';\omega_k, \rho_k) \quad\text{and}\quad U_*(x,t):=\min_{1\leq i\leq n}U_i(x,t),
\end{equation*}
which is well-defined due to \eqref{liu-24}. Clearly,  $U_*$ is lower semi-continuous.

 Set $\overline{d}=\max_{1\leq i\leq n}d_i>0$. We shall prove that  $U_*$ is a lower semi-continuous viscosity super-solution of the Hamilton-Jacobi problem
\begin{equation}\label{liu-21}
\begin{cases}
\begin{array}{ll}
\smallskip
\partial_{t}U_*+\overline{d} |\nabla U_*|^2+\int_0^1\mu({\bf A}(x,t)){\rm d}t = -\lambda_\infty    &\text{in }\,\,\Omega\times \mathbb{R},\\
\nabla U_*\cdot\nu =0 &\text{on }\,\partial\Omega\times \mathbb{R},
\end{array}
\end{cases}
\end{equation}
where the boundary condition is understood  in the viscosity sense \cite{L1985}. Then we can apply the same arguments as in the proof of \cite[Lemma 3.3]{LL2022}  to derive  that
$$\lambda_\infty\geq -\max_{x\in \overline{\Omega}}\int_0^1\mu({\bf A}(x,t)){\rm d}t= C_*.$$

To this end,  we fix any smooth test function $\phi$ and assume  that $U_*-\phi$ has a strict minimum at a point $(x_0, t_0)\in  \overline{\Omega}\times \mathbb{R}$. By the definition of viscosity super-solution \cite{I2011,L1985}, we must verify that
at $(x_0, t_0)$, it holds
\begin{equation}\label{liu-20}
\left\{
\begin{array}{ll}
\medskip
\partial_{t}\phi+\overline{d} |\nabla \phi|^2+\int_0^1\mu({\bf A}(x_0,t)){\rm d}t\geq -\lambda_\infty  &\text{if }\,\,x_0 \in \Omega,\\
\max\left\{\nabla \phi\cdot \nu ,\,\, \partial_{t}\phi+\overline{d} |\nabla \phi|^2-\int_0^1\mu({\bf A}(x_0,t)){\rm d}t+ \lambda_\infty\right\}\geq 0
 &\text{if }\,\,x_0 \in \partial \Omega. 
\end{array}
\right.
\end{equation}

Indeed, let ${\bf v}=(v_1,\cdots,v_n)>0$ denote the principal eigenvector of ${\bf A}(x,t)$ corresponding to the maximal eigenvalue $\mu ({\bf A}(x,t))$, that is
\begin{equation}\label{liu-20240520-1}
    \sum_{j=1}^n a_{ij}(x, t)  v_j=\mu ({\bf A}(x,t)) v_i \quad \text{ for all } (x,t)\in \Omega\times\mathbb{R},\,\, i=1,\cdots,n.
\end{equation}
For any $1\leq i \leq n$ and $k\geq 1$, we define
\begin{equation}\label{liu-20231230-4}
\begin{split}
    N_{i}(x,t;\omega_k,\rho_k):=&\sqrt{\rho_k}\log v_i\left(x,\frac{\omega_k t}{\sqrt{\rho_k}}\right)\\
    &-\frac{\sqrt{\rho_k}}{\omega_k}\left[\int_0^{\frac{\omega_k t}{\sqrt{\rho_k}}}\mu ({\bf A}(x,s)){\rm d}s-\frac{\omega_k t}{\sqrt{\rho_k}}\int_0^1\mu({\bf A}(x,t)){\rm d}t\right].
    \end{split}
\end{equation}
Due to $\rho_k\to 0$ and $\omega_k /\sqrt{\rho_k}\to +\infty$ as $k\to+\infty$, it follows 
that
 \begin{equation}\label{liu-27}
 \max_{1\leq i\leq n}\max_{ \overline{\Omega}\times\mathbb{R}} (|N_{i}|+|\nabla N_{i}| +|\Delta N_{i}|) \to 0 \quad\text{as }\,\, k\to +\infty.
 \end{equation}
 Together with \eqref{liu-20240520-1} and \eqref{liu-20231230-4}, we can verify that
 \begin{equation}\label{liu-20231230-3}
    \sum_{j=1}^n a_{ij}\left(x,\frac{\omega_k t}{\sqrt{\rho_k}}\right)  e^{\frac{N_{j}-N_{i}}{\sqrt{\rho_k}}}=\mu \left({\bf A}\left(x,\omega_k t/{\sqrt{\rho_k}}\right)\right) , \quad i=1,\cdots,n.
 \end{equation}

First, we assume $x_0\in\Omega$. 
By \eqref{liu-27}, the definition of $U_*$ implies 
$$ U_*(x,t)=\liminf_{k\to+\infty \atop (x',t')\to(x,t)}\min_{1\leq i\leq n }\left(W_i(x',t';\omega_k, \rho_k)- N_{i}(x',t';\omega_k, \rho_k)\right).$$
  Since $(x_0, t_0)$  is a strict local  minimum point of $U_*-\phi$, there exist some index $\ell\in\{1,\cdots, n\}$ and a sequence $(x_k, t_k)\in\Omega\times\mathbb{R}$ such that  $(x_k, t_k)\to(x_0,t_0)$ as $k\to+\infty$, and
  \begin{equation}\label{liu-20231230-2-1}
  \begin{split}
      W_{\ell}(x_k,t_k)- N_{\ell}(x_k,t_k)
      =\min_{1\leq j\leq n }\left(W_j(x_k,t_k)- N_{j} (x_k,t_k)\right)\to U_*(x_0,t_0),
      \end{split}
  \end{equation}
  and $W_{\ell}-(N_{\ell}+\phi)$ attains a local minimum at $(x_k, t_k)$, so that $\Delta (W_{\ell}-(N_{\ell}+\phi)\geq 0$ at $(x_k, t_k)$. Hence, we evaluate \eqref{liu-25} at $(x_k,t_k)$ and obtain
\begin{equation*}
\begin{split}
  &\partial_t N_{\ell}(x_k,  t_k)+\partial_t \phi +d_{\ell}|\nabla N_{\ell}+\nabla \phi|^2+\sum_{j=1}^n a_{\ell j}\left(x_k,\frac{\omega_k t_k}{\sqrt{\rho_k}}\right)e^{\frac{W_\ell-W_j}{\sqrt{\rho_k}}}\\
  \geq& -\lambda(\omega_k,\rho_k)+\sqrt{\rho_k}(\Delta N_{\ell}+ \Delta\phi).
  \end{split}
\end{equation*}
Hence, by \eqref{liu-27} and \eqref{liu-20231230-2-1}, we derive that at $(x_k,t_k)$,
\begin{equation*}
  \partial_t N_{\ell}(x_k,  t_k)+\partial_t \phi +d_{\ell}|\nabla \phi|^2+\sum_{j=1}^n a_{\ell j}\left(x_k,\frac{\omega_k t_k}{\sqrt{\rho_k}}\right)e^{\frac{N_{j}-N_{\ell}}{\sqrt{\rho_k}}}\\
  \geq -\lambda(\omega_k,\rho_k)+o(1),
\end{equation*}
which together with \eqref{liu-20231230-3} yields
\begin{equation}\label{liu-001}
  \partial_t N_{\ell}+\partial_t \phi +d_{\ell}|\nabla \phi|^2
  \geq -\lambda(\omega_k,\rho_k)-\mu \left({\bf A}\left(x,\omega_k t/{\sqrt{\rho_k}}\right)\right)+o(1) \quad \text{at }(x_k,  t_k).
\end{equation}
Observe from \eqref{liu-20231230-4} that
\begin{equation}\label{liu-20240303-1}
    \partial_t N_{\ell}=\omega_k \partial_t(\log v_{\ell})-\mu \left({\bf A}\left(x,\omega_k t_k/{\sqrt{\rho_k}}\right)\right)+\int_0^1\mu({\bf A}(x,t)){\rm d}t,\quad\forall (x,t)\in\Omega\times\mathbb{R}.
\end{equation}
Due to $\omega_k \to 0$ as $k\to+\infty$, we deduce from \eqref{liu-001} that
\begin{equation}\label{inequality}
    \partial_t \phi +\overline{d} |\nabla \phi|^2+\int_0^1\mu({\bf A}(x_k,t)){\rm d}t\geq -\lambda(\omega_k,\rho_k)+o(1)\quad \text{at }\, (x_k, t_k),
\end{equation}
where letting  $k\to+\infty$ yields the first equation in  \eqref{liu-20}.

Next we consider the case $x_0\in\partial\Omega$ and prove the second equation in \eqref{liu-20}. Assume $\nabla \phi(x_0, t_0)\cdot \nu (x_0)<0$, since otherwise \eqref{liu-20} would hold automatically. 
As  discussed above, there exists $(x_k, t_k)\in \overline{\Omega}\times\mathbb{R}$ such that \eqref{liu-20231230-2-1} holds for some $1\leq \ell\leq n$, $W_{\ell}-(N_{\ell}+\phi)$ attains its minimum at $(x_k, t_k)$, and $(x_k, t_k)\to(x_0,t_0)$ as $k\to+\infty$. We claim  that $x_k\in\Omega$ for large $k$. Indeed,  if $x_k\in\partial\Omega$, then the boundary condition of $W_\ell$ in \eqref{liu-25} implies that 
\begin{equation}\label{liu-26}
  0=\nabla W_{\ell}\cdot \nu \leq  \nabla(N_{\ell}+\phi)\cdot \nu  =\nabla N_{\ell}\cdot \nu +\nabla\phi\cdot \nu\quad \text{at }\, (x_k, t_k).
\end{equation}
In view of $\nabla \phi(x_0, t_0)\cdot \nu (x_0)<0$, using \eqref{liu-27} and letting $k\to+\infty$ in \eqref{liu-26} give a contradiction, which implies  $x_k\in\Omega$ for large $k$. Hence, \eqref{inequality} remains true   at $(x_k, t_k)$ for large $k$.
Letting  $k\to+\infty$ in \eqref{inequality}, 
we can conclude  \eqref{liu-20}  holds.

Therefore, $U_*$ is a viscosity super-solution of \eqref{liu-21} and thus $\lambda_\infty\geq C_*$.  Step 1  is complete.

\smallskip
{\it Step 2}. We show  $\lambda_\infty\leq C_*$. Similar to Step 1, by \eqref{liu-24} we define
\begin{equation*}
    \overline{U}_i(x,t):=\limsup_{k\to+\infty \atop (x',t')\to(x,t)}W_i(x',t';\omega_k, \rho_k) \quad\text{and}\quad \overline{U}_*(x,t):=\max_{1\leq i\leq n}U_i(x,t),
\end{equation*}
and thus $\overline{U}_*$ is upper semi-continuous. We shall show that  $\overline{U}_*$ is a viscosity sub-solution of 
\begin{equation}\label{liu-20240521-3}
\begin{cases}
\begin{array}{ll}
\smallskip
\partial_{t}U+\underline{d} |\nabla U|^2+ \int_0^1\mu({\bf A}(x,t)){\rm d}t= -\lambda_\infty    &\text{in }\,\,\Omega\times \mathbb{R},\\
\nabla U_*\cdot\nu =0 &\text{on }\,\partial\Omega\times \mathbb{R},
\end{array}
\end{cases}
\end{equation}
where  $\underline{d}=\min_{1\leq i\leq n}d_i>0$. That is, given any $(x_0, t_0)\in  \overline{\Omega}\times \mathbb{R}$ and any smooth test function $\phi$,  assume  that $\overline{U}_*-\phi$ has a strict maximum at  $(x_0, t_0)$, then we must prove that
at $(x_0, t_0)$,
\begin{equation}\label{liu-20240521-4}
\left\{
\begin{array}{ll}
\medskip
\partial_{t}\phi+\underline{d} |\nabla \phi|^2+ \int_0^1\mu({\bf A}(x_0,t)){\rm d}t\leq -\lambda_\infty  &\text{if }\,\,x_0 \in \Omega,\\
\max\left\{\nabla \phi\cdot \nu ,\,\, \partial_{t}\phi+\underline{d} |\nabla \phi|^2-\int_0^1\mu({\bf A}(x_0,t)){\rm d}t+ \lambda_\infty\right\}\leq 0
 &\text{if }\,\,x_0 \in \partial \Omega. 
\end{array}
\right.
\end{equation}

We first assume $x_0\in\Omega$. Let  $N_{i}$ be defined by \eqref{liu-20231230-4}. Due to \eqref{liu-27}, we observe from
 the definition of $\overline{U}_*$ that
$$ \overline{U}_*(x,t)=\limsup_{k\to+\infty \atop (x',t')\to(x,t)}\max_{1\leq i\leq n }\left(W_i(x',t';\omega_k, \rho_k)- N_{i}(x',t'; \omega_k, \rho_k)\right).$$
Hence, by \cite[Lemma 6.1]{B2013} there exist $\ell\in\{1,\cdots, n\}$ and sequence $(x_k, t_k)\in\Omega\times\mathbb{R}$ such that  $(x_k, t_k)\to(x_0,t_0)$ as $k\to+\infty$, $W_{\ell}-(N_{\ell}+\phi)$ attains  local maximum at $(x_k, t_k)$, and
  \begin{equation}\label{liu-20231230-2-2}
  \begin{split}
      W_{\ell}(x_k,t_k)- N_{\ell}(x_k,t_k)
      =\max_{1\leq j\leq n }\left(W_j(x_k,t_k)- N_{j} (x_k,t_k)\right)\to \overline{U}_*(x_0,t_0).
      \end{split}
  \end{equation}
  This implies that $\Delta W_{\ell}(x_k,t_k)\leq \Delta (N_{\ell}+\phi)(x_k,t_k)$, so that  by \eqref{liu-25} we have 
  \begin{equation*}
\begin{split}
  &\partial_t N_{\ell}(x_k,  t_k)+\partial_t \phi +d_{\ell}|\nabla N_{\ell}+\nabla \phi|^2+\sum_{j=1}^n a_{\ell j}\left(x_k,\frac{\omega_k t_k}{\sqrt{\rho_k}}\right) e^{\frac{W_\ell-W_j}{\sqrt{\rho_k}}}\\
  \leq& -\lambda(\omega_k,\rho_k)+\sqrt{\rho_k}(\Delta N_{\ell}+ \Delta\phi) \qquad \text{ at }(x_k,t_k).
  \end{split}
\end{equation*}
Similar to \eqref{liu-001},
we can deduce from \eqref{liu-20231230-3} and \eqref{liu-20231230-2-2} that
\begin{equation*}
  \partial_t N_{\ell}(x_k,  t_k)+\partial_t \phi +d_{\ell}|\nabla \phi|^2\\
  \leq -\lambda(\omega_k,\rho_k)-\mu \left({\bf A}\left(x_k,\omega_k t_k/{\sqrt{\rho_k}}\right)\right)+o(1).
\end{equation*}
Hence, by \eqref{liu-20240303-1} and the fact that $\omega_k \to 0$ as $k\to+\infty$,  it follows  that
\begin{equation*}
    \partial_t \phi +\underline{d} |\nabla \phi|^2+\int_0^1\mu({\bf A}(x_k,t)){\rm d}t\leq -\lambda(\omega_k,\rho_k)+o(1)\quad \text{at }\, (x_k, t_k).
\end{equation*}
Letting  $k\to+\infty$ in the above inequality yields the first equation in  \eqref{liu-20240521-4}. Then we can use same arguments as in Step 1 to conclude the  second equation in \eqref{liu-20240521-4} holds for $x_0\in \partial\Omega$.

Therefore, $\overline{U}_*$ is a viscosity lower-solution of \eqref{liu-20240521-3}, so that by the arguments as in the proof of \cite[Lemma 3.3]{LL2022}  we have
   $\lambda_\infty\leq C_* 
   $.
This and Step 1 
together prove $\lambda_\infty=C_*$. 
The proof is now complete.
\end{proof}

Next, we consider the asymptotic behaviors for principal eigenvalue   in the limit of $(\omega,\rho)\to(0,0)$
 and $\omega\approx \sqrt{\rho}$. 
 To this end,
we first prepare the following comparison result for the critical value of Hamilton-Jacobi equation \eqref{Liu01-20210716}.
\begin{lemma}\label{liu-20240122}
 For any $({\bf p},x,t)\in\mathbb{R}^n\times \Omega \times \mathbb{R}$, let Hamiltonian ${\rm H} ({\bf p},x,t)$ be defined in Theorem {\rm\ref{TH-liu-20240106}}. Given any $\vartheta>0$, let $c_\vartheta\in \mathbb{R}$ be any value such that the following problem admits a lower semi-continuous 
 viscosity super-solution (resp. sub-solution):
\begin{equation}\label{liu-20240324-5}
\begin{cases}
\begin{array}{ll}
\smallskip
\vartheta\partial_{t}U+{\rm H}(\nabla U, x,t) = -c_\vartheta    &\text{in }\,\,\Omega\times \mathbb{R},\\
\nabla U\cdot\nu =0 &\text{on }\,\partial\Omega\times \mathbb{R},\\
U(x,t)= U(x,t+1) &\text{in }\,\,\Omega\times \mathbb{R}.
\end{array}
\end{cases}
\end{equation}
Then $C(\vartheta)\leq c_\vartheta$ (resp. $C(\vartheta)\geq c_\vartheta$), where $C(\vartheta)$ is  the critical value of problem \eqref{Liu01-20210716} as defined in Theorem {\rm\ref{TH-liu-20240106}}.
\end{lemma}

\begin{proof}
We assume $c_\vartheta$ be the value such that problem \eqref{liu-20240122} admits a viscosity super-solution and prove $ C(\vartheta)\leq c_\vartheta$, then the converse can  be shown by the same arguments.  Suppose on the contrary that $ C(\vartheta)>c_\vartheta$.
Let $\overline{U}_\vartheta$ be any viscosity super-solution of \eqref{liu-20240324-5}. Denote
$U_\vartheta$ by any Lipschitz viscosity solution of \eqref{Liu01-20210716} associated to $C(\vartheta)$.
 Adding  a constant to $\overline{U}_\vartheta$ if necessary, we may suppose that $U_\vartheta(x,0)<\overline{U}_\vartheta(x,0)$ for all $x\in\overline{\Omega}$. Due to $c_\vartheta> C(\vartheta)$, we can choose $\delta>0$ small such that $\vartheta \delta-C(\vartheta)\leq -c_\vartheta$, so that $\underline{U}_{\delta}:=U_\vartheta+\delta t$ satisfies
\begin{equation*}
\begin{cases}
\begin{array}{ll}
\smallskip
\vartheta\partial_{t} \underline{U}_{\delta} +{\rm H}(\nabla \underline{U}_{\delta},x,t)\leq \vartheta\partial_{t} \overline{U}_\vartheta+{\rm H}(\nabla \overline{U}_\vartheta,x,t), &x\in \Omega, \, t>0,\\
\nabla U_\delta\cdot\nu =\nabla \overline{U}_\vartheta\cdot\nu =0, &x\in \partial\Omega, \, t>0
\end{array}
\end{cases}
\end{equation*}
in the viscosity sense.
Since $\overline{U}_\vartheta(x,0)>\underline{U}_\delta(x,0)$ for all $x\in\Omega$,
By comparison (see e.g. \cite[Theorem 3.4]{I2011}),
it holds that $\overline{U}_\vartheta \geq \underline{U}_\delta $ for all $x\in\overline{\Omega}$ and $t>0$. Hence, for any  integer $n$,
$$\overline{U}_\vartheta(x,0)=\overline{U}_\vartheta(x, n)\geq \underline{U}_\delta(x,n)=U_\vartheta(x,0)+\delta n,$$
for which letting $n\to+\infty$ yields $\overline{U}_\vartheta(x,0)=+\infty$. This  is a contradiction, and thus   $ C(\vartheta)\leq c_*$, which completes the proof.
\end{proof}

\begin{proposition}\label{prop-20240521}
   Let
$\lambda(\omega, \rho)$ denote the principal eigenvalue of \eqref{Liu1}.
Then for any $\vartheta>0$, 
$$\lim_{(\omega,\rho)\to (0,0) \atop \frac{\omega}{\sqrt{\rho}}\to \vartheta}\lambda(\omega, \rho)=C(\vartheta),$$
whereas $C(\vartheta)$ is the unique value such that 
\eqref{Liu01-20210716} admits a Lipschitz viscosity solution. 
\end{proposition}
\begin{proof}
Denote by ${\bm \varphi}=(\varphi_1,\cdots,\varphi_n)>0$   the principal eigenvector of \eqref{Liu1} associated to $\lambda(\omega,\rho)$, which is normalized by $\max_{1\leq i\leq n}\max_{\overline\Omega\times\mathbb{R}}\varphi_i=1$. Set  $W_i(x,t;\omega,\rho):=-\sqrt{\rho} \log \varphi_i(x, t)$ for all $(x,t)\in\Omega\times \mathbb{R}$, $i=1,\cdots,n$. Then $$\min_{1\leq i\leq n}\min_{\overline\Omega\times\mathbb{R}} W_i(\cdot,\cdot;\omega,\rho)=0, \quad\forall \omega,\rho>0.$$  Direct calculations from \eqref{Liu1} yield that for any $1\leq i\leq n$,
  \begin{equation}\label{liu-20240110}
\left\{
\begin{array}{ll}
\frac{\omega}{\sqrt{\rho}} \partial_{t}W_i-\sqrt{\rho}d_i\Delta W_i+d_i |\nabla W_i|^2+\displaystyle\sum_{j=1}^n a_{ij}e^{\frac{W_i-W_j}{\sqrt{\rho}}}=-\lambda(\omega,\rho)  &\text{in }\,\,\Omega\times \mathbb{R},\\
\nabla W_i\cdot\nu =0 &\text{on }\,\partial\Omega\times \mathbb{R}.
\end{array}
\right.
\end{equation}
Note that $\lambda(\omega,\rho)$ is uniformly bounded for all $\omega,\rho>0$.  
By the same argument in \cite[Lemma 2.1]{ES1989},  
there exists some positive constant $C$ independent of $\omega$ and $\rho$ such that 
\begin{equation}\label{liu-24-1}
      \max_{1\leq i\leq n}\max_{ \overline{\Omega}\times\mathbb{R}}|W_i(x,t;\omega,\rho)| \leq C, \quad \forall \omega,\rho>0.
\end{equation}

Given any  sequences $\{\omega_k\}_{k\geq 1}$ and $\{\rho_k\}_{k\geq 1}$ such that $\omega_k\to+\infty$ and $\omega_k/\sqrt{\rho_k}\to \vartheta$ as $k\to+\infty$, we assume  $\lambda(\omega_k,\rho_k)\to \lambda_\vartheta$ for some $\lambda_\vartheta\in\mathbb{R}$. 

\smallskip
{\it Step 1}. We prove $\lambda_\vartheta\geq C(\vartheta)$.
Define 
\begin{equation}\label{definitionU}
   \underline{U}_i(x,t):=\liminf_{k\to+\infty \atop (x',t')\to(x,t)}W_i(x',t';\omega_k, \rho_k) \quad\text{and}\quad U_*(x,t):=\min_{1\leq i\leq n}\underline{U}_i(x,t),
\end{equation}
which are well-defined due to \eqref{liu-24-1}. Clearly,  $U_*$ is lower semi-continuous. We shall verify that $U_*$ is a viscosity super-solution of the problem
\begin{equation}\label{liu-21-1}
\begin{cases}
\begin{array}{ll}
\smallskip
\vartheta\partial_{t}U+{\rm H}(\nabla U, x,t) = -\lambda_\vartheta    &\text{in }\,\,\Omega\times \mathbb{R},\\
\nabla U\cdot\nu =0 &\text{on }\,\partial\Omega\times \mathbb{R}.
\end{array}
\end{cases}
\end{equation}
That is,
for any smooth test function $\phi$ and assume  that $U_*-\phi$ has a strict minimum at a point $(x_0, t_0)\in \overline{\Omega}\times \mathbb{R}$, then
at $(x_0, t_0)$,
\begin{equation}\label{liu-20-1}
\begin{cases}
\begin{array}{ll}
\smallskip
\vartheta\partial_{t}\phi+{\rm H}(\nabla\phi,x_0,t_0)\geq -\lambda_\vartheta  &\text{if }\,\,x_0 \in \Omega,\\
\max\{\nabla \phi\cdot \nu ,\,\, \partial_{t}\phi+{\rm H}(\nabla\phi,x,t)+ \lambda_\vartheta\}\geq 0
 &\text{if }\,\,x_0 \in \partial \Omega. 
\end{array}
\end{cases}
\end{equation}
Then it follows from Lemma \ref{liu-20240122} that $\lambda_\vartheta\geq C(\vartheta)$.

To this end, denote by  ${\bf v}=(v_1,\cdots,v_n)$  the positive eigenvector of matrix
$${\rm diag}\{d_1,\cdots, d_n\}|\nabla\phi|^2(x_0,t_0)+{\bf A}(x_0,t_0)$$ corresponding to the eigenvalue ${\rm H}(\nabla\phi(x_0,t_0),x_0,t_0)$, then at $(x_0,t_0)$,
\begin{equation}\label{liu-20240122-1}
   d_i |\nabla\phi|^2 + \sum_{j=1}^n \frac{a_{ij} v_j}{v_i}={\rm H}(\nabla\phi,x_0,t_0) \quad \text{for all } 1\leq i\leq n.
\end{equation}
Observe from \eqref{definitionU} that
\begin{equation*}
   U_*(x,t):=\min_{1\leq i\leq n}\liminf_{k\to+\infty \atop (x',t')\to(x,t)}\left(W_i(x',t';\omega_k, \rho_k)+\sqrt{\rho_k} \log v_i\right).
\end{equation*}
Then there exists some index $1\leq \ell\leq n$ and a sequence $(x_k,t_k)$ satisfying $(x_k,t_k)\to (x_0,t_0)$ as $k\to+\infty$ such that
\begin{equation}\label{liu-20240122-2}
    W_\ell(x_k,t_k)+\sqrt{\rho_k}\log v_\ell=\min_{1\leq i\leq n}\left\{W_i(x_k,t_k)+\sqrt{\rho_k}\log v_i\right\},
\end{equation}
and
$(W_\ell+\rho_k\log v_\ell) -\phi$  attains  a local minimum at $(x_k,t_k)$.

First, we assume $x_0\in\Omega$ and prove the first inequality in \eqref{liu-20-1}. Since $x_k\in\Omega$ for large $k$, we can assume $x_k\in\Omega$ for all $k\geq 1$ without loss of generality.  By  \eqref{liu-20240110} we derive that
\begin{equation}\label{liu-20240808-1}
    \begin{split}
         -\lambda(\omega_k,\rho_k)&= \frac{\omega_k}{\sqrt{\rho_k}} \partial_{t}W_\ell-\sqrt{\rho_k}d_\ell\Delta W_\ell+d_\ell |\nabla W_\ell|^2+\displaystyle\sum_{j=1}^n a_{\ell j}e^{\frac{W_\ell-W_j}{\sqrt{\rho_k}}}\\
   &\leq \frac{\omega_k}{\sqrt{\rho_k}} \partial_{t}\phi-\sqrt{\rho_k}d_\ell\Delta \phi+d_\ell |\nabla \phi|^2+\displaystyle\sum_{j=1}^n a_{\ell j}e^{\frac{W_\ell-W_j}{\sqrt{\rho_k}}} \quad \text{at } (x_k,t_k).
    \end{split}
\end{equation}
It follows from \eqref{liu-20240122-2} that
\begin{equation*}
    W_\ell(x_k,t_k)-W_j(x_k,t_k)\leq \sqrt{\rho_k}\log v_j-\sqrt{\rho_k}\log v_\ell,\quad \forall j=1,\cdots,n,
\end{equation*}
which together with \eqref{liu-20240808-1} yields
\begin{equation}\label{liu-20240122-3}
   -\lambda(\omega_k,\rho_k)\leq \frac{\omega_k}{\sqrt{\rho_k}} \partial_{t}\phi-\sqrt{\rho_k}d_\ell\Delta \phi+d_\ell |\nabla \phi|^2+\displaystyle\sum_{j=1}^n a_{\ell j}\frac{v_j}{v_\ell} \quad \text{at } (x_k,t_k).
\end{equation}
Letting $k\to +\infty$ in \eqref{liu-20240122-3} and applying \eqref{liu-20240122-1}, we can derive the first inequality in \eqref{liu-20-1}.

Next,
we consider the case $x_0\in\partial\Omega$ and prove the second equation in \eqref{liu-20-1}. Assume $\nabla \phi(x_0, t_0)\cdot \nu (x_0)<0$, since otherwise \eqref{liu-20-1} would hold automatically. 
As  discussed above, there exists $(x_k, t_k)\in\overline{\Omega}\times\mathbb{R}$ such that $(W_\ell+\rho_k\log v_\ell)-\phi$ attains a minimum at $(x_k, t_k)$ and $(x_k, t_k)\to(x_0,t_0)$ as $k\to+\infty$. We claim  that $x_k\in\Omega$ for large $k$. Indeed,  if $x_k\in\partial\Omega$, then 
\begin{equation}\label{liu-26-1}
  0=\nabla W_\ell(x_k, t_k)\cdot \nu (x_k)\leq  \nabla\phi(x_k, t_k)\cdot \nu (x_k).
\end{equation}
In view of $\nabla \phi(x_0, t_0)\cdot \nu (x_0)<0$, letting $k\to+\infty$ in \eqref{liu-26-1} gives a contradiction, which implies  $x_k\in\Omega$ for large $k$. Hence, \eqref{liu-20240122-3} remains true   at $(x_k, t_k)$ for large $k$. Hence, the second inequality in \eqref{liu-21-1} follows.

\smallskip
{\it Step 2}.  We prove $\lambda_\vartheta\leq C(\vartheta)$.
Define 
\begin{equation*}
    \overline{U}_i(x,t):=\limsup_{k\to+\infty \atop (x',t')\to(x,t)}W_i(x',t';\omega_k, \rho_k) \quad\text{and}\quad U^*(x,t):=\max_{1\leq i\leq n}\overline{U}_i(x,t),
\end{equation*}
which are also well-defined  by \eqref{liu-24-1}. Clearly,  $U^*$ is upper semi-continuous. We shall verify that $U^*$ is a viscosity sub-solution of problem \eqref{definitionU}.
For any smooth test function $\phi$ such  that $U^*-\phi$ has a strict maximum at some point $(x_0, t_0)\in \overline{\Omega}\times \mathbb{R}$, then
at $(x_0, t_0)$,
\begin{equation}\label{liu-20-1-1}
\begin{cases}
\begin{array}{ll}
\smallskip
\vartheta\partial_{t}\phi+{\rm H}(\nabla\phi,x_0,t_0)\leq -\lambda_\vartheta  &\text{if }\,\,x_0 \in \Omega,\\
\max\{\nabla \phi\cdot \nu ,\,\, \partial_{t}\phi+{\rm H}(\nabla\phi,x,t)+ \lambda_\vartheta\}\leq 0
 &\text{if }\,\,x_0 \in \partial \Omega. 
\end{array}
\end{cases}
\end{equation}
Then the desired $\lambda_\vartheta\leq C(\vartheta)$ follows from Lemma \ref{liu-20240122}.

 Let ${\bf v}=(v_1,\cdots,v_n)>0$ be defined by \eqref{liu-20240122-1}. As in Step 1, there exists some index $1\leq \ell\leq n$ and a sequence $(x_k,t_k)$ satisfying $(x_k,t_k)\to (x_0,t_0)$ as $k\to+\infty$ such that
\begin{equation}\label{liu-20240122-4}
    W_\ell(x_k,t_k)+\sqrt{\rho_k}\log v_\ell=\max_{1\leq i\leq n}\left\{W_i(x_k,t_k)+\sqrt{\rho_k}\log v_i\right\},
\end{equation}
and
$(W_\ell+\sqrt{\rho_k}\log v_\ell) -\phi$  attains  a local maximum at $(x_k,t_k)$.

If $x_0\in\Omega$, then $x_k\in\Omega$ for large $k$.
Similar to \eqref{liu-20240808-1}, we have
\begin{align*}
   -\lambda(\omega_k,\rho_k)
   &\geq \frac{\omega_k}{\sqrt{\rho_k}} \partial_{t}\phi-\sqrt{\rho_k}d_\ell\Delta \phi+d_\ell |\nabla \phi|^2+\displaystyle\sum_{j=1}^n a_{\ell j}e^{\frac{W_\ell-W_j}{\sqrt{\rho_k}}}\quad \text{at } (x_k,t_k).
\end{align*}
Due to \eqref{liu-20240122-4}, it follows that
\begin{equation*}
    W_\ell(x_k,t_k)-W_j(x_k,t_k)\geq \sqrt{\rho_k}\log v_j-\sqrt{\rho_k}\log v_\ell, \quad \forall j=1,\cdots,n,
\end{equation*}
and thus
\begin{align*}
   -\lambda(\omega_k,\rho_k)
   \geq \frac{\omega_k}{\sqrt{\rho_k}} \partial_{t}\phi-\sqrt{\rho_k}d_\ell\Delta \phi+d_\ell |\nabla \phi|^2+\displaystyle\sum_{j=1}^n a_{\ell j}\frac{v_j}{v_\ell} \quad \text{at } (x_k,t_k),
\end{align*}
for which letting $k\to +\infty$ gives the first inequality in
\eqref{liu-20-1-1}.

If $x_0\in\partial\Omega$, then the second inequality in
\eqref{liu-20-1-1} can be proved by the similar arguments in Step 1.
The proof is now complete.
\end{proof}

Finally, we investigate the asymptotic behaviors of the critical value $C(\vartheta)$ for  Hamilton-Jacobi equation \eqref{Liu01-20210716}.
We first prepare the following result. 
\begin{lemma}\label{liu-lem-20240329}
    For each $({\bf p},x,t)\in \mathbb{R}^n\times\Omega\times \mathbb{R}$, let Hamiltonian ${\rm H}({\bf p},x,t)$ be defined in the statement of Theorem {\rm\ref{TH-liu-20240106}}. For each $t\in\mathbb{R}$, consider the  Hamilton-Jacobi equation
    \begin{equation}\label{liu-20240329-1}
\begin{cases}
\begin{array}{ll}
\smallskip
{\rm H}(\nabla U, x,t) = -c(t)    &\text{in }\,\,\Omega,\\
\nabla U\cdot\nu =0 &\text{on }\,\partial\Omega.
\end{array}
\end{cases}
\end{equation}
Then  the critical value $c(t)$ such that \eqref{liu-20240329-1} admits a Lipschitz viscosity solution satisfies
\begin{equation*}
c(t)=-\max_{x\in\overline{\Omega}} \mu({\bf A}(x,t)) \quad \text{ for all }\,\, t\in\mathbb{R},
\end{equation*}
whereas 
$\mu({\bf A}(x,t))$ denotes the maximal eigenvalue of matrix ${\bf A}(x,t)$.
\end{lemma}

\begin{proof}
    let $U $ be any Lipschitz viscosity solution of \eqref{liu-20240329-1} associated with $c(t)$, which is differentiable a.e. in $\Omega$ by  Radamacher theorem. By \eqref{liu-20240329-1} we have 
\begin{equation}\label{liu-20240329-2}
  {\rm H}(\nabla U , x,t)= -c(t) \quad\text{a.e. in }\,\, x\in \Omega.
\end{equation}
Recall from the definition of Hamiltonian ${\rm H}$ that ${\rm H}(\nabla U , x,t)\geq {\rm H}(0, x,t)=\mu({\bf A}(x,t))$, so that \eqref{liu-20240329-2} implies  $c(t)\leq -\mu({\bf A}(x,t))$ a.e. in $\Omega\times(0,1)$. By the continuity of $\mu({\bf A}(x,t))$, we conclude
 $c(t)\leq -\max_{x\in \overline{\Omega}}\mu({\bf A}(x,t))$ for all $t\in\mathbb{R}$. 
 It remains to establish
\begin{equation}\label{liu-20240329-3}
c(t)\geq -\max_{x\in\overline{\Omega}} \mu({\bf A}(x,t)) \quad \text{ for all }\,\, t\in\mathbb{R}.
\end{equation}
Assume $U $ attains its  local minimal at some point $x_*\in \overline{\Omega}$. We consider  two cases:

{\rm (1)}  If $x_*\in \Omega$, then $U -0$ attains its  local minimal at $x_*$. By the definition of viscosity solution (or super-solution), we arrive at
${\rm H}(0,x_*,t)\geq -c(t)$,
which implies $c(t)\geq -\mu(x_*,t)$. Hence, \eqref{liu-20240329-3}  holds.

{\rm (2)}  If $x_*\in \partial\Omega$, 
we define $\mathcal{D}_{x_*}$ as the set of sub-differentials of $U$ at the point $x_*$.  In fact,
 $p\in \mathcal{D}_{x_*}$  if and only if there exists some $\phi\in C^{1}( \overline{\Omega})$ such that $\nabla\phi(x_*)=p$ and $x_*$ is a local minimum point of $U-\phi$. Since $x_*$ is a minimum point  and  $x_*\in\partial\Omega$, one can choose $\phi=-\sigma {\rm dist}(x, \partial\Omega)$ for any $\sigma\ge 0$ and verify
$[0,\infty)\subset \{\sigma\in\mathbb{R}: \sigma\nu (x_*)\in \mathcal{D}_{x_*}\}$.
Next, we define
\begin{equation*}
    \underline{\sigma}:=\inf\{\sigma\in\mathbb{R}: \sigma\nu (x_*)\in  \mathcal{D}_{x_*}\}\leq 0.
\end{equation*}

We first consider the case  $ \underline{\sigma}<0$.
Fix any $\sigma\in[\underline{\sigma}, 0)$ such that $\sigma\nu (x_*)\in  \mathcal{D}_{x_*}$.  
We can choose some $\phi\in C^{1}( \overline{\Omega})$ such that $\nabla\phi(x_*)=\sigma\nu (x_*)$ and  $U-\phi$ attains a local minimum at $x=x_*$.   In view of  $\nabla\phi(x_*) \cdot \nu (x_*)=\sigma<0$,  by the definition of viscosity solutions to  \eqref{liu-20240329-1},  one has
$H(\sigma \nu(x_*),x_*,t)\geq -c(t)$, for which letting $\sigma\to 0$ gives  $c(t)\geq -\mu(x_*,t)$, which proves \eqref{liu-20240329-3}.

It remains to consider the case $\underline{\sigma}=0$.
Given any $\delta>0$, 
we define  $D_\delta:=B_\delta(x_*)\cap \overline{\Omega}$.
Set $\psi(x):=4(x-x_*)^2$. Assume $U-(\delta {\rm dist}(x, \partial\Omega)-\psi)$ attains its minimum  at a point $x_\delta\in \overline{D}_\delta$.  We claim $x_\delta\in D_\delta$. Indeed,  if $x_\delta\in\partial\Omega$, then there holds
$U(x_\delta, t)+\psi(x_\delta)\leq U(x_*,t)$,  so that $x_\delta=x_*$ (due to $x_*$ being a local minimal point of $U$). This yields $\nabla \psi(x_*)-\delta \nu(x_*)=-\delta \nu(x_*)\in \mathcal{D}_{x_*}$, and thus $\underline{\sigma}\leq -\delta<0$, contradicting the assumption  $\underline{\sigma}=0$. Hence, $x_\delta\not\in\partial\Omega$. If $x_\delta\in \Omega$ and $|x_\delta-x_*|=\delta$, then
$U(x_\delta, t)- (\delta {\rm dist}(x_\delta,\partial\Omega)-\psi(x_\delta))\leq U(x_*, t)$,
and thus $\delta {\rm dist}(x_\delta,\partial\Omega)\geq\psi(x_\delta)$.
In view of $\psi(x_\delta)=4\delta^2$ and $\delta {\rm dist}(x_\delta,\partial\Omega)\leq 2\delta^2$, this is a contradiction. Therefore, $x_\delta\in D_\delta$.

Define the text function $\phi=\delta {\rm dist}(x, \partial\Omega)-\psi$.  Since $U-(\delta {\rm dist}(x, \partial\Omega)-\psi)$ attains its minimum  at $x_\delta\in D_\delta\subset\Omega$, the definition of viscosity solutions to \eqref{liu-lem-20240329} implies 
\begin{equation*}
  {\rm H}(-\delta \nu(x_\delta)-8(x_\delta-x_*),x_\delta,t)\geq -c(t) \quad\text{ for all } t\in\mathbb{R}.
\end{equation*}
Hence, letting $\delta\to 0$ yields $c(t)\geq -\mu(x_*,t)$, which implies \eqref{liu-20240329-3}.
The proof is complete.
\end{proof}

\begin{proposition}\label{liu-20240326}
    Let $C(\vartheta)$ be the critical value of Hamilton-Jacobi equation \eqref{Liu01-20210716} as defined in Theorem {\rm\ref{TH-liu-20240106}}. Then there holds
    $$\lim_{\vartheta\to 0}C(\vartheta)=\underline{C}\quad\text{and}\quad \lim_{\vartheta\to +\infty}C(\vartheta)=C_*.$$
\end{proposition}
\begin{proof}
    {\it Step 1}. We first prove  $C(\vartheta)\to\underline{C}$ as $\vartheta\to 0$.   For each $t\in \mathbb{R}$, by Lemma  \ref{liu-lem-20240329}  we denote $\overline{U}(\cdot,t)$ by any  Lipschitz viscosity solution of \eqref{liu-20240329-1} associated with 
    $c(t)=-\max_{x\in\overline{\Omega}}\mu({\bf A}(x,t))$. We define the time-periodic function
    $$V(x,t)=\frac{1}{\vartheta}\left[\int_0^ t\max_{x\in\overline{\Omega}}\mu({\bf A}(x,s)){\rm d}s-\underline{C} t\right]+\overline{U}(x,t), \quad (x,t)\in{\Omega}\times\mathbb{R}.$$
By \eqref{liu-20240329-1}, direct calculations yield
    \begin{align*}
        \vartheta\partial_t V+{\rm H} (\nabla V,x,t)=\max_{x\in\overline{\Omega}}\mu({\bf A}(x,t))-\underline{C}+\vartheta\partial_t \overline{U}+{\rm H} (\nabla \overline{U},x,t)=\vartheta\partial_t \overline{U}-\underline{C}
    \end{align*}
    in the sense of viscosity solutions. Hence, for any $\epsilon>0$, we can choose $\vartheta>0$ small such that the Lipschitz function $V$  is a viscosity solution of
    \begin{equation*}
\begin{cases}
\begin{array}{ll}
\smallskip
\vartheta\partial_{t}V+{\rm H}(\nabla V, x,t) \geq -\underline{C}-\epsilon    &\text{in }\,\,\Omega\times \mathbb{R},\\
\smallskip
\vartheta\partial_{t}V+{\rm H}(\nabla V, x,t) \leq -\underline{C}+\epsilon    &\text{in }\,\,\Omega\times \mathbb{R},\\
\smallskip
\nabla V\cdot\nu =0 &\text{on }\,\partial\Omega\times \mathbb{R},\\
V(x,t)= V(x,t+1) &\text{in }\,\,\Omega\times \mathbb{R}.
\end{array}
\end{cases}
\end{equation*}
Then by applying Lemma \ref{liu-20240122} we conclude that
$$\underline{C}-\epsilon\leq\liminf_{\vartheta\to 0} C(\vartheta)\leq\limsup_{\vartheta\to 0} C(\vartheta)\leq \underline{C}+\epsilon.$$
Letting $\epsilon\to 0$ yields the desired result. Step 1 is complete.

\smallskip

 {\it Step 2}. We next prove  $C(\vartheta)\to C_*$ as $\vartheta\to +\infty$. Fix any sequence $\{\vartheta_n\}_{n\geq 1}$ such that $\vartheta_n\to +\infty$ as $n\to \infty$.
Let $U_n$ be any Lipschitz viscosity solution of \eqref{Liu01-20210716}  associated to $C(\vartheta_n)$. Up to extraction, we assume $C(\vartheta_n)\to c_\infty$ as $n\to \infty$. It suffices to show $c_\infty=C_*$.
Set
 \begin{equation}\label{liu-20240330-1}
   U_*(x):=\lim\limits_{n\to\infty}\inf_{t\in[0, 1]}U_n(x,t) \quad\text{and}\quad U^*(x):=\lim\limits_{n\to \infty}\sup_{t\in [0,1]}U_n(x,t).
\end{equation}

{\rm (1)} We show that $U_*$ and $U^*$ are respectively the   viscosity super-solution and sub-solution of the time-independent Hamilton-Jacobi equation:
 \begin{equation}\label{liu-20240329-1-1}
\widehat{{\rm H}}(\nabla U, x) =-c_\infty    \,\,\text{ in }\,\,\Omega,\qquad
\nabla U_*\cdot\nu =0 \,\,\text{ on }\,\partial\Omega,
\end{equation}
where  $\widehat{\rm H}({\bf p},x)=\int_0^1 {\rm H}({\bf p},x,t){\rm d}t$ for any $({\bf p},x)\in \mathbb{R}^n\times\Omega$. We shall verify that $U_*$ is a   viscosity super-solution of \eqref{liu-20240329-1-1}, then the fact that $U^*$ is a viscosity sub-solution can be proved by the same arguments.

By the definition of viscosity super-solutions, we fix any test function $\phi\in C^1(\Omega)$ and assume that $U_*-\phi$ attains a strict local minimum point $x_*\in \overline{\Omega}$. We must prove
 \begin{equation}\label{liu-20240330-1-2}
\begin{cases}
\begin{array}{ll}
\smallskip
\widehat{{\rm H}}(\nabla\phi(x_*),x_*)\geq -c_\infty  &\text{if }\,\,x_* \in \Omega,\\
\max\{\nabla \phi\cdot \nu ,\,\, \widehat{\rm H}(\nabla\phi,x_*)+c_\infty\}\geq 0
 &\text{if }\,\,x_* \in \partial \Omega. 
\end{array}
\end{cases}
\end{equation}
We only show the first inequality in \eqref{liu-20240330-1-2} in the case $x_*\in\Omega$, then the second inequality for $x_* \in \partial \Omega$ follows by the same arguments  in the proof of Proposition \ref{liuprop-2}.

Set ${\bf p}_*:=\nabla\phi(x_*)$. Suppose on the contrary that
\begin{equation}\label{liu-20240330-2}
    \widehat{{\rm H}}({\bf p}_*,x_*)+c_\infty<0.
\end{equation}
We define the perturbed test function as
$$\phi_n(x,t)=\phi(x)+\frac{1}{\vartheta_n}\left[t\widehat{{\rm H}}({\bf p}_*,x_*)-\int_0^t {\rm H}({\bf p}_*,x_*,s){\rm d}s\right], \quad \forall (x,t)\in \Omega\times\mathbb{R},$$
which is $1$-periodic in $t$.
By direct calculations, one obtains
\begin{align*}
    &\vartheta\partial_t \phi_n+{\rm H} (\nabla \phi_n,x,t)\\
    =&\widehat{{\rm H}}({\bf p}_*,x_*)-{\rm H}({\bf p}_*,x_*,t)+{\rm H}(\nabla \phi(x),x,t)\\
    =&-C(\vartheta_n)+(\widehat{{\rm H}}({\bf p}_*,x_*)+C(\vartheta_n))+({\rm H}(\nabla \phi(x),x,t)-{\rm H}({\bf p}_*,x_*,t)).
\end{align*}
Hence, by \eqref{liu-20240330-2} and the continuity of Hamiltonian ${\rm H}$, one can choose some $r>0$ small such that $B_r(x_*)\subset \Omega$ and
\begin{equation}\label{liu-20240330-3}
    \vartheta_n\partial_t \phi_n+{\rm H} (\nabla \phi_n,x,t)<-C(\vartheta_n), \quad \forall (x,t)\in B_r(x_*)\times \mathbb{R},
\end{equation}
provided that $n$ is sufficiently large.
Recall that $U_n$ is a Lipschitz viscosity solution of \eqref{Liu01-20210716}. We can apply the comparison principle to \eqref{Liu01-20210716} and \eqref{liu-20240330-3} to derive that
$$\min_{B_r(x_*)\times \mathbb{R}} (U_n-\phi_n)\geq \min_{\partial B_r(x_*)\times \mathbb{R}} (U_n-\phi_n).$$
By letting $n\to \infty$, it follows from the definition of $U_*$ in \eqref{liu-20240330-1} that
$$\min_{B_r(x_*)} (U_*-\phi)\geq \min_{\partial B_r(x_*)} (U_*-\phi),$$
which contradicts to the fact that  $x_*$ is a strict local minimum point of $U_*-\phi$. Therefore, \eqref{liu-20240330-1-2} holds and $U_*$ is a  viscosity solution of \eqref{liu-20240329-1-1}.

{\rm (2)} We compete the proof.
We first claim that the unique  value such that 
 \begin{equation}\label{liu-20240723-1}
\widehat{{\rm H}}(\nabla U, x) = -c    \,\, \text{ in }\,\,\Omega,\qquad
\nabla U\cdot\nu =0\,\, \text{ on }\,\partial\Omega
\end{equation}
admits a  Lipschitz viscosity solution is $c=C_*$.  Indeed, noting that  $\widehat{{\rm H}}(\nabla U , x)\geq \widehat{{\rm H}}(0, x)=\int_0^1 \mu({\bf A}(x,t)){\rm d}t$, we observe from  \eqref{liu-20240723-1} that $c\leq -\int_0^1 \mu({\bf A}(x,t)){\rm d}t$ for all  $x\in\Omega$, and thus $c\leq C_*$.
Then the assertion $c\geq C_*$ can be similarly established following the arguments in Lemma \ref{liu-lem-20240329}, thus the claim is proved.
Furthermore,  the assertion in {\rm (1)} indicates that problem \eqref{liu-20240723-1} admits both viscosity super-solution and viscosity sub-solution when $c=c_\infty$.
By apply Lemma \ref{liu-20240122}, we can infer that $c_\infty$ is exactly the critical value of \eqref{liu-20240723-1}, i.e. $c_\infty=C_*$. The proof is now complete.
\end{proof}
We are in a position to prove Theorem \ref{TH-liu-20240106}.
\begin{proof}[Proof of Theorem  {\rm \ref{TH-liu-20240106}}]

By Propositions \ref{liuprop-2}, \ref{prop-20240521} and \ref{liu-20240326}, it remains to show
\begin{equation}\label{liu-20240531-3}
    \lim_{(\omega,\rho)\to (0,0) \atop \frac{\omega}{\sqrt{\rho}}\to 0}\lambda(\omega, \rho)=\underline{C}=-\int_0^1 \max_{x\in \overline{\Omega}} \mu({\bf A}(x,t)){\rm d}t.
\end{equation}
The proof of \eqref{liu-20240531-3} is postponed to the end of section \ref{monotonicity-section}, as the monotonicity result in Theorem \ref{TH1-1} is required. The proof is complete.
\end{proof}

\section{\bf Monotonicity of the principal eigenvalue}\label{monotonicity-section}



In this section, we shall establish the monotone dependent of the principal eigenvalue 
on frequency $\omega$ and diffusion rate $\rho$. 

The adjoint problem of \eqref{Liu1} can be written as
\begin{equation}\label{Liu2}
 \left\{
\begin{array}{ll}
\medskip
-\omega\partial_t {\bm \psi}-\rho {\bf D}\Delta{\bm \psi}-{\bf A}(x,t){\bm \psi}=\lambda{\bm \psi}~~&\mathrm{in}~\Omega\times \mathbb{R},\\
\medskip
  \nabla {\bm \psi}\cdot\nu =0,~~&\mathrm{on}~\partial\Omega\times\mathbb{R},\\
  {\bm \psi}(x,t)={\bm \psi}(x,t+1),&\mathrm{in}~\Omega\times \mathbb{R}.
  \end{array}
 \right.
 \end{equation}
We first prepare the following equality holds.
\begin{lemma}\label{L2}
Let ${\bm \varphi}=(\varphi_1,\cdots,\varphi_n)>0$ and ${\bm\psi}=(\psi_1,\cdots,\psi_n)>0$ denote respectively  the principal eigenfunctions of \eqref{Liu1} and \eqref{Liu2} associated with 
$\lambda(\omega,\rho)$. Then there holds 
\begin{equation}\label{W1}
\begin{split}
    2 \omega\int_0^1\!\!\!\int_\Omega  {\bm\psi}^{\rm T} \partial_t {\bm \varphi}
    =& \rho \sum_{i=1}^n d_i \int_0^1\!\!\!\int_\Omega\varphi_i\psi_i\left|\nabla\log \left(\frac{\psi_i}{\varphi_i}\right)\right|^2\\
    &+\frac{1}{2}\sum_{i,j=1}^n\int_0^1\!\!\!\int_\Omega a_{ij}( \varphi_j\psi_i-\psi_j \varphi_i)\log\left(\frac{\psi_i\varphi_j}{\varphi_i\psi_j}\right).
    \end{split}
\end{equation}
\end{lemma}
\begin{proof}
Multiply both sides of \eqref{Liu1} by $\psi_i\log(\psi_i/\varphi_i)$ and integrate the resulting equation over $\Omega\times(0,1)$, then by the boundary conditions in  \eqref{Liu1} we calculate that
\begin{align}\label{liu-20231230-1}
    &\lambda(\omega, \rho)\int_0^1\!\!\!\int_\Omega \varphi_i\psi_i\log(\psi_i/\varphi_i)\notag\\
    =&\omega\int_0^1\!\!\!\int_\Omega\psi_i\log(\psi_i/\varphi_i) \partial_t\varphi_i -\rho d_i\int_0^1\!\!\!\int_\Omega \psi_i\log(\psi_i/\varphi_i)\Delta\varphi_i\notag\\
    &-\sum_{j=1}^n\int_0^1\!\!\!\int_\Omega a_{ij}\varphi_j\psi_i\log(\psi_i/\varphi_i)\notag\\
    =&-\omega\int_0^1\!\!\!\int_\Omega \varphi_i \log(\psi_i/\varphi_i) \partial_t \psi_i-\omega\int_0^1\!\!\!\int_\Omega \varphi_i \psi_i \partial_t \log(\psi_i/\varphi_i)\\
    &+\rho d_i\int_0^1\!\!\!\int_\Omega  \log(\psi_i/\varphi_i) \nabla\varphi_i\cdot \nabla \psi_i+\rho d_i\int_0^1\!\!\!\int_\Omega  \psi_i \nabla\varphi_i\cdot \nabla \log(\psi_i/\varphi_i)\notag\\
    &-\sum_{j=1}^n\int_0^1\!\!\!\int_\Omega a_{ij}\varphi_j\psi_i\log(\psi_i/\varphi_i).\notag
\end{align}
Similarly, multiply both sides of \eqref{Liu2}  by $\varphi_i\log(\psi_i/\varphi_i)$ and integrate  over $\Omega\times(0,1)$, then 
\begin{align}\label{liu-20231230-2}
&\lambda(\omega,\rho)\int_0^1\!\!\!\int_\Omega \varphi_i\psi_i\log(\psi_i/\varphi_i)\notag\\
    =&-\omega\int_0^1\!\!\!\int_\Omega \varphi_i \log(\psi_i/\varphi_i) \partial_t \psi_i+\rho d_i\int_0^1\!\!\!\int_\Omega  \log(\psi_i/\varphi_i) \nabla\varphi_i\cdot \nabla \psi_i\\
    &+\rho d_i\int_0^1\!\!\!\int_\Omega  \varphi_i \nabla\psi_i\cdot \nabla \log(\psi_i/\varphi_i)-\sum_{j=1}^n\int_0^1\!\!\!\int_\Omega a_{ij}\psi_j\varphi_i\log(\psi_i/\varphi_i).\notag
\end{align}
Subtracting equations \eqref{liu-20231230-1} and \eqref{liu-20231230-2} and using the symmetry of matrix ${\bf A}$ yield
\begin{align}\label{liu-20240120-1}
    &\omega\int_0^1\!\!\!\int_\Omega \varphi_i \psi_i \partial_t \log(\psi_i/\varphi_i)\notag\\
    =&\rho d_i\int_0^1\!\!\!\int_\Omega  \psi_i \nabla \varphi_i\cdot\nabla \log(\psi_i/\varphi_i)-\rho d_i\int_0^1\!\!\!\int_\Omega \varphi_i \nabla\psi_i \cdot \nabla \log(\psi_i/\varphi_i)\notag\\
    &-\sum_{j=1}^n\int_0^1\!\!\!\int_\Omega a_{ij} \varphi_j\psi_i\log(\psi_i/\varphi_i)+\sum_{j=1}^n\int_0^1\!\!\!\int_\Omega a_{ji}\psi_j\varphi_i\log(\psi_i/\varphi_i)\notag\\
     =&-\rho d_i\int_0^1\!\!\!\int_\Omega  \frac{\psi_i |\nabla \varphi_i|^2}{\varphi^2_i}+2\rho d_i\int_0^1\!\!\!\int_\Omega  \nabla \varphi_i\cdot \nabla \psi_i-\rho d_i\int_0^1\!\!\!\int_\Omega  \frac{\varphi_i |\nabla \psi_i|^2}{\psi^2_i}
     \\
    &-\sum_{j=1}^n\int_0^1\!\!\!\int_\Omega a_{ij} (\varphi_j\psi_i-\psi_j\varphi_i)\log(\psi_i/\varphi_i)\notag\\
    =&-\rho d_i \int_0^1\!\!\!\int_\Omega\varphi_i\psi_i\left|\nabla\log \left(\frac{\psi_i}{\varphi_i}\right)\right|^2\notag\\
    &-\frac{1}{2}\sum_{j=1}^n\int_0^1\!\!\!\int_\Omega a_{ij}(\varphi_j\psi_i-\psi_j\varphi_i)\log\left(\frac{\psi_i\varphi_j}{\varphi_i\psi_j}\right).\notag
\end{align}
Observe that
$$\int_0^1\!\!\!\int_\Omega \varphi_i \psi_i \partial_t \log(\psi_i/\varphi_i)=-2\int_0^1\!\!\!\int_\Omega \psi_i \partial_t \varphi_i.$$
Then  summing equality \eqref{liu-20240120-1} from $i$ to $n$ yields \eqref{W1}. The proof is complete.
 \end{proof}

Based on the equality in Lemma \ref{L2}, 
we can establish 
the following estimate.
\begin{theorem}\label{liutheorem-0229}
Let $U_{\omega/\sqrt{\rho}}$ be any Lipschitz  viscosity solution of \eqref{Liu01-20210716} with $\vartheta=\omega/\sqrt{\rho}$ corresponding to  the critical value $C(\omega/\sqrt{\rho})$.
Let ${\bm \varphi}=(\varphi_1,\cdots,\varphi_n)
>0$ 
denote the principal eigenfunction of problem \eqref{Liu1}. 
Then
there holds 
\begin{align*}
\begin{split}
  \lambda(\omega,\rho)-C(\omega/\sqrt{\rho})
    \geq  \,&\rho\sum_{i=1}^n\int_0^1\!\!\!\int_\Omega d_i|\nabla \sqrt{\varphi_i\psi_i}|^2\\ &+\sum_{i=1}^n\int_0^1\!\!\!\int_\Omega \varphi_i\psi_i \left|\frac{\sqrt{\rho}}{2}\nabla\log (\frac{\varphi_i}{\psi_i})+\nabla U_{\omega/\sqrt{\rho}} \right|^2,
    \quad \forall \omega,\rho>0,
   \end{split}
\end{align*}
whereas ${\bm \psi}=(\psi_1,\cdots,\psi_n)>0$ normalized by $\int_0^1\!\!\int_\Omega {\bm \varphi}^{\rm T}{\bm \psi}=1$   denotes the principal eigenfunction of the adjoint problem \eqref{Liu2}.
\end{theorem}

\begin{proof}
Recall that ${\bm \varphi}=(\varphi_1,\cdots,\varphi_n)>0$ and ${\bm \psi}=(\psi_1,\cdots,\psi_n)>0$ are respectively the principal eigenfunctions of \eqref{Liu1} and its adjoint problem \eqref{Liu2} associated with $\lambda(\omega,\rho)$, which  can be normalized by $\int_0^1\!\!\int_\Omega {\bm \varphi}^{\rm T}{\bm \psi}=1$.  Define
\begin{equation}\label{liu-0703-1}
\alpha_i:=\sqrt{\varphi_i\psi_i}\quad\text{ and }\quad\beta_i:=-\frac{\sqrt{\rho}}{2}\log \left(\frac{\varphi_i}{\psi_i}\right).
\end{equation}
By \eqref{Liu1} and \eqref{Liu2}, we can calculate from \eqref{liu-0703-1}  that
\begin{align}\label{liu-20231231-2}
  d_i \nabla\cdot \left(\alpha_i^2 \nabla\beta_i\right)
    = &\frac{\sqrt{\rho }d_i}{2}\nabla\cdot \left(\varphi_i \nabla\psi_i-\psi_i \nabla\varphi_i \right)\notag\\
    =&\frac{\sqrt{\rho}d_i}{2} \left(\varphi_i \Delta\psi_i-\psi_i \Delta\varphi_i\right)\\
    =&-\frac{\omega}{2\sqrt{\rho }}\partial_t(\alpha_i^2)+\frac{1}{2\sqrt{\rho}}\sum_{j=1}^n a_{ij} (\varphi_j\psi_i-\psi_j\varphi_i).\notag
\end{align}
We write $U=U_{\omega/\sqrt{\rho}}$  for  simplicity, which is any Lipschitz viscosity solution of \eqref{Liu01-20210716} associated to $C(\omega/\sqrt{\rho})$, 
i.e.
\begin{equation}\label{liu-20231231-3}
   \frac{\omega}{\sqrt{\rho}} \partial_{t}U+
   {\rm H}(\nabla U,x,t)= -C(\omega/\sqrt{\rho})\quad \text{ a.e. in } \Omega\times\mathbb{R}.
\end{equation}
Multiplying both sides of \eqref{liu-20231231-2} by $U$ yields that
\begin{align*}
   d_i\int_0^1\!\!\!\int_\Omega\alpha_i^2 \nabla\beta_i\cdot \nabla U=
    -\frac{\omega}{2\sqrt{\rho}}\int_0^1\!\!\!\int_\Omega\alpha_i^2\partial_t U-\frac{1}{2\sqrt{\rho}}\sum_{j=1}^n  \int_0^1\!\!\!\int_\Omega a_{ij}(\varphi_j\psi_i-\psi_j\varphi_i)U.
    \end{align*}
Summing  both sides of the above equality over index $i$ from $1$ to $n$, we derive 
\begin{equation*}
   \sum_{i=1}^n d_i\int_0^1\!\!\!\int_\Omega\alpha_i^2 \nabla\beta_i\cdot \nabla U+\frac{\omega}{2\sqrt{\rho}}\sum_{i=1}^n\int_0^1\!\!\!\int_\Omega\alpha_i^2\partial_t U=0.
\end{equation*}
Hence, in view of $\sum_{i=1}^n\int_0^1\!\int_\Omega\alpha_i^2=1$,  by \eqref{liu-20231231-3} we deduce that
\begin{equation}\label{liu-20231231-1}
   2\sum_{i=1}^n d_i\int_0^1\!\!\!\int_\Omega\alpha_i^2 \nabla\beta_i\cdot \nabla U=C(\omega/\sqrt{\rho})+\sum_{i=1}^n\int_0^1\!\!\!\int_\Omega\alpha_i^2 {\rm H}(\nabla U,x,t).
\end{equation}
By \eqref{liu-20231231-3} and \eqref{liu-20231231-1}, direct calculations yield
\begin{align}\label{liu-20231231-4}
    &\rho\sum_{i=1}^n d_i\int_0^1\!\!\!\int_\Omega |\nabla \alpha_i|^2+\sum_{i=1}^n d_i\int_0^1\!\!\!\int_\Omega \alpha_i^2\left|\nabla\beta_i-\nabla U \right|^2 \notag\\
    =&\frac{\rho}{4}\sum_{i=1}^n d_i\int_0^1\!\!\!\int_\Omega \alpha_i^2|\nabla\log(\alpha_i^2)|^2+\sum_{i=1}^n d_i\int_0^1\!\!\!\int_\Omega \alpha_i^2|\nabla\beta_i|^2\notag\\
    &-2\sum_{i=1}^n  d_i\int_0^1\!\!\!\int_\Omega \alpha_i^2\nabla\beta_i\cdot\nabla U+ \sum_{i=1}^n  d_i\int_0^1\!\!\!\int_\Omega \alpha_i^2|\nabla U|^2 \notag\\
    =&\rho\sum_{i=1}^n d_i\int_0^1\!\!\!\int_\Omega \nabla\varphi_i\cdot\nabla\psi_i+2\sum_{i=1}^n d_i\int_0^1\!\!\!\int_\Omega \alpha_i^2|\nabla\beta_i|^2
    \\
&-2\sum_{i=1}^n  d_i\int_0^1\!\!\!\int_\Omega \alpha_i^2\nabla\beta_i\cdot\nabla U+ \sum_{i=1}^n  d_i\int_0^1\!\!\!\int_\Omega \alpha_i^2|\nabla U|^2 \notag\\
=&2\sum_{i=1}^n d_i\int_0^1\!\!\!\int_\Omega \alpha_i^2|\nabla\beta_i|^2+\rho\sum_{i=1}^n d_i\int_0^1\!\!\!\int_\Omega \nabla\varphi_i\cdot\nabla\psi_i
\notag\\
&
+ \sum_{i=1}^n  d_i\int_0^1\!\!\!\int_\Omega \alpha_i^2|\nabla U|^2-\sum_{i=1}^n\int_0^1\!\!\!\int_\Omega\alpha_i^2 {\rm H}(\nabla U,x,t) -C(\omega/\sqrt{\rho}). \notag
\end{align}

Observe from Lemma \ref{L2} that
\begin{equation}\label{liu-20231231-5}
    \begin{split}
         2\sum_{i=1}^n d_i\int_0^1\!\!\!\int_\Omega \alpha_i^2|\nabla\beta_i|^2=&\frac{\rho}{2}\sum_{i=1}^n d_i\int_0^1\!\!\!\int_\Omega \alpha_i^2|\nabla\log(\varphi_i/\psi_i)|^2\\
    =& \omega\int_0^1\!\!\!\int_\Omega  {\bm\psi}^{\rm T} \partial_t {\bm \varphi}
    -\frac{1}{2}\sum_{i,j=1}^n\int_0^1\!\!\!\int_\Omega a_{ij} \varphi_j\psi_i\log\left(\frac{\psi_i\varphi_j}{\varphi_i\psi_j}\right).
    \end{split}
\end{equation}
Multiply both sides of \eqref{Liu1} by ${\bm\psi}^{\mathrm{T}}$ and integrate over $(0,1)\times \Omega$, then
\begin{align*}
\omega\int_0^1\!\!\!\int_\Omega  {\bm \psi}^{\rm T} \partial_t {\bm \varphi}
    +\rho\sum_{i=1}^n d_i\int_0^1\!\!\!\int_\Omega \nabla\varphi_i\cdot\nabla\psi_i- \sum_{i,j=1}^n\int_0^1\!\!\!\int_\Omega a_{ij} \varphi_j\psi_i=\lambda(\omega,\rho).
   \end{align*}
This together with 
\eqref{liu-20231231-5} yields
\begin{align*}
&2\sum_{i=1}^n d_i\int_0^1\!\!\!\int_\Omega \alpha_i^2|\nabla\beta_i|^2+\rho\sum_{i=1}^n d_i\int_0^1\!\!\!\int_\Omega \nabla\varphi_i\cdot\nabla\psi_i\\
=&\lambda(\omega,\rho)+\sum_{i,j=1}^n\int_0^1\!\!\!\int_\Omega a_{ij}\varphi_j\psi_i-\frac{1}{2}\sum_{i,j=1}^n\int_0^1\!\!\!\int_\Omega a_{ij} \varphi_j\psi_i\log\left(\frac{\psi_i\varphi_j}{\varphi_i\psi_j}\right).
\end{align*}
Substituting the above equality into \eqref{liu-20231231-4} gives
\begin{equation}\label{liu-20231231-6}
    \begin{split}
        &\rho\sum_{i=1}^n\int_0^1\!\!\!\int_\Omega d_i|\nabla \alpha_i|^2+\sum_{i=1}^n\int_0^1\!\!\!\int_\Omega \alpha_i^2\left|\nabla\beta_i-\nabla U \right|^2 \\ =&\lambda(\omega,\rho)-C(\omega/\sqrt{\rho}) -\frac{1}{2}\sum_{i,j=1}^n\int_0^1\!\!\!\int_\Omega a_{ij} \varphi_j\psi_i\log\left(\frac{\psi_i\varphi_j}{\varphi_i\psi_j}\right)
        \\
        &+ \sum_{i=1}^n  d_i \int_0^1\!\!\!\int_\Omega \alpha_i^2|\nabla U|^2-\sum_{i=1}^n\int_0^1\!\!\!\int_\Omega\alpha_i^2 {\rm H}(\nabla U,x,t) +\sum_{i,j=1}^n\int_0^1\!\!\!\int_\Omega a_{ij}\varphi_j\psi_i.
    \end{split}
\end{equation}
Recall that ${\rm H}(\nabla U,x,t)$ defines the the maximal eigenvalue of matrix $${\bf A}+{\rm{diag}}\{d_1,\cdots, d_n\}|\nabla U|^2.$$
By the symmetry of matrix ${\bf A}$,
we observe 
that
\begin{align*}
    {\rm H}(\nabla U,x,t)\sum_{i=1}^n \alpha_i^2 &= {\rm H}(\nabla U,x,t) {\bm\psi}^{\rm T}{\bm\varphi}\\
    &\geq {\bm\psi}^{\rm T} ({\bf A}+{\rm{diag}}\{d_1,\cdots,d_n\}|\nabla U|^2) {\bm\varphi} \\
    &= \sum_{i,j=1}^n a_{ij} \varphi_j\psi_i+ |\nabla U|^2\sum_{i=1}^n d_i \alpha_i^2.
\end{align*}
Then we can derive from \eqref{liu-20231231-6} that
\begin{align*}
   \lambda(\omega,\rho)-C(\omega/\sqrt{\rho})
    \geq \rho\sum_{i=1}^n\int_0^1\!\!\!\int_\Omega d_i|\nabla \alpha_i|^2+\sum_{i=1}^n\int_0^1\!\!\!\int_\Omega \alpha_i^2\left|\nabla\beta_i-\nabla U \right|^2.
\end{align*}
This proves Theorem \ref{liutheorem-0229}.
\end{proof}

Based on the equality in Lemma \ref{L2}, we can also establish the following monotonicity result,
which implies
Theorem {\rm\ref{TH1-1}}.
\begin{theorem}\label{TH1}
Let $\lambda(\omega,\rho)$ denote the principal eigenvalue of \eqref{Liu1}.
Suppose that $\rho=\rho(\omega)\in C^1((0,\infty))$,
$\rho'(\omega)\geq 0$, and $\left[\rho(\omega)/\omega^2\right]'\leq 0$ 
in \eqref{Liu1}. Then $\lambda(\omega,\rho(\omega))$ is non-decreasing in $\omega$.   

 Furthermore, 
    if $\rho'(\omega)= 0$, then $(\lambda(\omega,\rho(\omega)))'= 0$ if and only if
   \begin{equation}\label{n-s-condition}
       {\bf A}{\bm \phi}=\widehat{\bf A}(x){\bm \phi}+g(t){\bm \phi},\quad \forall (x,t)\in\Omega\times\mathbb{R}
   \end{equation}
  for  some periodic function $g\in C(\mathbb{R})$, where $\widehat{\bf A}$ is the temporally averaged matrix defined in Proposition {\rm\ref{TH-liu-20240227}}
  and
  ${\bm \phi}={\bm \phi}(x)$ is the principal eigenfunction of elliptic problem \eqref{liu-20240227-1}.

In particular, for each 
$\rho>0$,
 either $\partial\lambda(\omega, \rho)/\partial \omega> 0$ for all $\omega>0$, or $\partial\lambda(\omega, \rho)/\partial \omega\equiv 0$.
\end{theorem}
\begin{proof}

In the proof, we denote ${\bm \varphi}=(\varphi_1,\cdots,\varphi_n)>0$ and ${\bm \psi}=(\psi_1,\cdots,\psi_n)>0$ by   the principal eigenfunctions of \eqref{Liu1} and \eqref{Liu2},  respectively. The proof is completed by the following two steps.

\smallskip

 {\it Step 1}. We first show that the principal eigenvalue $\lambda(\omega,\rho(\omega))$
  is non-decreasing  in $\omega\geq0$. Let $\rho=\rho(\omega)$ in \eqref{Liu1} and define $\Lambda(\omega):=\lambda(\omega,\rho(\omega))$. Differentiate both sides of \eqref{Liu1} with respect to $\omega$ and denote ${\bm \varphi}'=\frac{\partial {\bm \varphi}}{\partial\omega}$ for brevity. Then 
 \begin{equation*}
 \partial_t{\bm \varphi}+\omega \partial_t{\bm \varphi}'-\rho'(\omega){\bf D}\Delta{\bm \varphi}-\rho(\omega){\bf D}\Delta{\bm \varphi}'-{\bf A}(x,t) {\bm \varphi}'=\Lambda(\omega) {\bm \varphi}'+\Lambda'(\omega) {\bm \varphi} \quad \text{in}~~\Omega\times\mathbb{R}.
 \end{equation*}
Recall  that ${\bm \psi}$ satisfies \eqref{Liu2}. Multiply the above equation by ${\bm \psi}^{\mathrm{T}}$ and integrate the resulting equation over $\Omega\times(0,1)$, then
\begin{equation*}
   \Lambda'(\omega)\int_0^1\!\!\!\int_\Omega{\bm \psi}^{\mathrm{T}}{\bm \varphi} =\int_0^1\!\!\!\int_\Omega{\bm \psi}^{\mathrm{T}}\partial_t{\bm \varphi} +\rho'(\omega)\sum_{i=1}^n d_i \int_0^1\!\!\!\int_\Omega \nabla\varphi_i\cdot \nabla \psi_i.
\end{equation*}
By  Lemma \ref{L2} and the symmetry of matrix ${\bf A}$, we derive
\begin{align}\label{liu-20231222-1}
    &\Lambda'(\omega)\int_0^1\!\!\!\int_\Omega{\bm \psi}^{\mathrm{T}}{\bm \varphi}\notag\\
    =&\frac{\rho(\omega)}{2\omega} \sum_{i=1}^n d_i \int_0^1\!\!\!\int_\Omega\varphi_i\psi_i\left|\nabla\log \left(\frac{\psi_i}{\varphi_i}\right)\right|^2\notag+\rho'(\omega)\sum_{i=1}^n d_i \int_0^1\!\!\!\int_\Omega \nabla\varphi_i\cdot \nabla \psi_i\notag\\
    &+\frac{1}{4\omega}\sum_{i,j=1}^n\int_0^1\!\!\!\int_\Omega a_{ij}( \varphi_j\psi_i-\psi_j \varphi_i)\log\left(\frac{\psi_i\varphi_j}{\varphi_i\psi_j}\right)\notag\\
    \geq&\left[\frac{\rho(\omega)}{2\omega}-\frac{\rho'(\omega)}{4}\right] \sum_{i=1}^n d_i \int_0^1\!\!\!\int_\Omega\varphi_i\psi_i\left|\nabla\log \left(\frac{\psi_i}{\varphi_i}\right)\right|^2\\
     &+\frac{\rho'(\omega)}{4} \sum_{i=1}^n d_i \int_0^1\!\!\!\int_\Omega\varphi_i\psi_i\left|\nabla\log \left(\frac{\psi_i}{\varphi_i}\right)\right|^2+\rho'(\omega)\sum_{i=1}^n d_i \int_0^1\!\!\!\int_\Omega \nabla\varphi_i\cdot \nabla \psi_i
     \notag\\
     =&-\frac{\omega^2}{4}\left[\frac{\rho(\omega)}{\omega^2}\right]'\sum_{i=1}^n d_i \int_0^1\!\!\!\int_\Omega\varphi_i\psi_i\left|\nabla\log \left(\frac{\psi_i}{\varphi_i}\right)\right|^2\notag\\
     &+\frac{\rho'(\omega)}{4} \sum_{i=1}^n d_i \int_0^1\!\!\!\int_\Omega\varphi_i\psi_i\left|\nabla\log \left(\varphi_i\psi_i\right)\right|^2.
     \notag
\end{align}
Due to $\rho'(\omega)\geq 0$ and $\left[\rho(\omega)/\omega^2\right]'\leq 0$, \eqref{liu-20231222-1} implies $\Lambda'(\omega) \geq 0$ for all $\omega>0$. This proves the monotonicity of $\Lambda(\omega)$ and Step 1 is complete.

{\it Step 2}. We assume $\rho'(\omega)=0$ and prove $\Lambda'(\omega)= 0$ if and only if ${\bf A}{\bm \phi}=\widehat{\bf A}(x){\bm \phi}+g(t){\bm \phi}$ for all $(x,t)\in\Omega\times\mathbb{R}$ for some periodic function $g$, where  ${\bm \phi}={\bm \phi}(x)>0$ is the principal eigenfunction of \eqref{liu-20240227-1} with $\rho=\rho(\omega)$ correspoding to $\overline{\lambda}(\rho(\omega))$.
First, we assume  ${\bf A}{\bm\phi}=\hat{{\bf A}}(x){\bm\phi}+g(t){\bm\phi}$ with some periodic function $g$. Set $\overline{\bm \varphi}(x,t)=e^{\left(\frac{1}{\omega}\int_{0}^{t}g(s)\mathrm{d}s\right)}{\bm \phi}$. Then  $\overline{\bm \varphi}$ is $1$-periodic and solves
\begin{equation*}
 \begin{cases}
\begin{array}{ll}
\smallskip
\omega\partial_t \overline{\bm \varphi}-\rho(\omega) {\bf D}\Delta \overline{\bm \varphi}-{\bf A}(x,t)\overline{\bm \varphi}=\overline{\lambda}(\rho(\omega))\overline{\bm \varphi} &\text{ in }\,\,\,\Omega\times\mathbb{R},\\
\smallskip
  \nabla \overline{\bm \varphi}\cdot\nu =0 &\text{ on }\,\,\,\Omega\times\mathbb{R},\\
  \overline{\bm \varphi}(x,t)=\overline{\bm \varphi}(x,t+1) &\text{ in }\,\,\,\Omega\times\mathbb{R}.
  \end{array}
 \end{cases}
 \end{equation*}
By the uniqueness of principal eigenvalue, 
we have $\Lambda(\omega)\equiv \overline{\lambda}(\rho(\omega))$ for all $\omega>0$, and hence $\Lambda'(\omega)=0$ due to $\rho'(\omega)=0$. 

Conversely, we assume $\Lambda'(\omega_*)=0$ for some $\omega_*>0$ and show \eqref{n-s-condition} holds. 
Due to $\rho'(\omega_*)=0$,  we have $\left[\rho(\omega)/\omega^2\right]'|_{\omega=\omega_*}< 0$.
It follows from \eqref{liu-20231222-1} 
that  $\varphi_{i}\psi_{j}=\psi_{i}\varphi_{j}$ and $\varphi_i=c_i(t)\psi_i$ for all $1\leq i,j\leq n$, where $c_i(t)>0$ is some $1$-periodic function. In particular, substituting $\varphi_i=c_i(t)\psi_i$ into $\psi_{j}\varphi_{i}=\psi_{i}\varphi_{j}$ gives $c_i(t)=c_j(t)=c(t)$. We then substitute $\varphi_i=c(t)\psi_i$ in \eqref{Liu1} to derive that
$$\omega_*\partial_t{\bm \psi}+\omega_*(\log c)'{\bm \psi}-\rho(\omega_*) {\bf D}\Delta{\bm\psi}-{\bf A}{\bm\psi}=\Lambda(\omega_*){\bm\psi}\quad \text{ in } \,\,\Omega\times\mathbb{R}.$$
By \eqref{Liu2} and the symmetry of ${\bf A}$, we have
$2(\log c)'{\bm\psi}+\partial_t{\bm\psi}=0$, and thus
   $\frac{\partial\log \psi_i}{\partial t}=-2(\log c)'$
 for all  $1\leq i\leq n$. Hence, $\psi_i=f_i(x)/c^2(t)$ with some positive function $f_i(x)$. 
 Then by  \eqref{Liu2},
\begin{equation}\label{liu-20240229-1}
    -\rho {\bf D}\Delta {\bm f}-{\bf A}{\bm f}=\left(\Lambda(\omega_*)-2\omega_*(\log c)'\right){\bm f} \quad \text{ in } \,\,\Omega\times\mathbb{R}.
\end{equation}
Integrating the above equation with respect to $t$ on $(0,1)$ gives
\begin{equation*}
 \left\{
\begin{array}{ll}
\medskip
-\rho {\bf D}\Delta {\bm f}-\hat{{\bf A}}(x){\bm f}=\lambda(\omega_*,\rho){\bm f} &\mathrm{in}~\,\,\,\Omega,\\
  \nabla  {\bm f}\cdot\nu =0,~~&\mathrm{on}\,\,\,~\partial\Omega.\\
  \end{array}
 \right.
 \end{equation*}
This implies that ${\bm f}>0$ is a eigenfunction of problem \eqref{liu-20240227-1}, whence $\Lambda(\omega_*)=\overline{\lambda}(\rho(\omega_*))$ and ${\bm f}= \tilde{c}{\bm \phi}$ for some constant $\tilde{c}>0$. Substituting this into \eqref{liu-20240229-1} gives
$${\bf A}{\bm \phi}=-\rho {\bf D}\Delta {\bm \phi}+\left(2\omega_*(\log c)'-\overline{\lambda}(\rho(\omega_*))\right){\bm \phi} \quad \text{ in } \,\,\Omega\times\mathbb{R}.$$
Hence, \eqref{n-s-condition} holds 
with 
$g(t)=2\omega_*(\log c)'-\overline{\lambda}(\rho(\omega_*))$. The proof is  now complete.
\end{proof}

We are in a position to prove Theorems \ref{liutheorem-0229-1} and \ref{TH1-1}.

\begin{proof}[Proof of Theorem {\rm\ref{liutheorem-0229-1}}]
    Let $C(\omega/\sqrt{\rho})$ be the critical value of 
\eqref{Liu01-20210716} with $\vartheta=\omega/\sqrt{\rho}$.
A direct application of Theorem \ref{liutheorem-0229} yields   $\lambda(\omega,\rho)\geq C(\omega/\sqrt{\rho})$ for all $\omega,\rho>0$. It remains to show that $\vartheta\mapsto C(\vartheta)$ is non-decreasing. By Theorem \ref{TH1} we observe that $\vartheta\mapsto \lambda(\vartheta\sqrt{\rho},\rho)$ is non-decreasing. Due to $\lambda(\vartheta\sqrt{\rho},\rho)\to C(\vartheta)$ as $\rho \to 0$ as shown in Proposition \ref{prop-20240521}, this implies directly the monotonicity of $\vartheta\mapsto C(\vartheta)$. The proof is complete.
\end{proof}

\begin{proof}[Proof of Theorem {\rm\ref{TH1-1}}]
Theorem \ref{TH1-1} is a direct consequence of Theorem \ref{TH1}.
\end{proof}

Finally, we conclude this section by proving the limiting result \eqref{liu-20240531-3}, which completes the proof of Theorem \ref{TH-liu-20240106}.

\begin{proof}[Proof of  {\rm \eqref{liu-20240531-3}}]
 For each $\vartheta>0$, we apply Theorem \ref{TH1} with $\rho(\omega)=\omega^2/\vartheta^2$  to obtain that $\lambda(\omega,\rho(\omega))=\lambda(\vartheta\sqrt{\rho},\rho)$ is non-decreasing in $\rho$ for any fix $\vartheta>0$. Hence, for any $\rho_*>0$, 
 \begin{equation}\label{liu-36}
\lambda(\vartheta\sqrt{\rho}, \rho)\leq \lambda(\vartheta\sqrt{\rho_*}, \rho_*) \quad \text{ for all } \rho\leq \rho_* \text{ and }  \vartheta>0.
 \end{equation}
Proposition \ref{TH-liu-20240227}{\rm(ii)} implies $\lambda(\vartheta\sqrt{\rho_*}, \rho_*)\to \underline{\lambda}(\rho_*)$ as $\vartheta\to 0$, where
$\underline{\lambda}(\rho_*)$ is defined by Proposition \ref{TH-liu-20240227}.
Hence, by \eqref{liu-36} we derive that
 \begin{equation}\label{liu-20240531-4}
  \limsup_{  (\vartheta,\rho)\to (0,0)}\lambda(\vartheta\sqrt{\rho}, \rho)\leq \lim_{  \rho_*\to 0}\underline{\lambda}(\rho_*)=\underline{C},
\end{equation}
where we applied \cite[Theorem 1]{D2009} to problem \eqref{liu-20240227-2} to obtain the limit of $\underline{\lambda}(\rho_*)$ as $\rho_*\to 0$.

It follows from Theorem \ref{liutheorem-0229} that $\lambda(\vartheta\sqrt{\rho}, \rho)\geq  C(\vartheta)$ for all $\vartheta,\rho>0$. Hence,
by Proposition \ref{liu-20240326} we deduce that
$$ \limsup_{ (\vartheta,\rho)\to (0,0)}\lambda(\vartheta\sqrt{\rho}, \rho)\geq  \lim_{  \vartheta\to 0}C(\vartheta)=\underline{C},$$
which together with \eqref{liu-20240531-4} proves \eqref{liu-20240531-3}. The proof of Theorem \ref{TH-liu-20240106} is now complete.
\end{proof}

\medskip
\section{\bf Level sets of the principle eigenvalue}\label{Sect.5}
In this section, we will classify
structure of the level sets of  principal eigenvalue 
for problem \eqref{Liu1}. 
To this end, we first introduce some notations. 
Recall that ${\bf A}(x,t)=(a_{ij}(x,t))_{n\times n}$ with $a_{ij}$ being some continuous time-periodic function with unit period. We  define
\begin{equation}\label{def_underlineC2}
    \begin{split}
        C_*^+:=-\max_{x\in \overline{\Omega}} \mu (\widehat{\bf A}(x)), \quad
        \underline{C}^+:=-\int_0^1 \mu (\overline{\bf A}(t) ){\rm d}t,\quad
        \overline{C}:=-\mu(\overline{\widehat{\bf A}}),
    \end{split}
\end{equation}
where $\widehat{\bf A}$, $\overline{\bf A}$ and $\overline{\widehat{\bf A}}$ denote  the temporally or/and spatially averaged matrices 
with entries
$$(\widehat{\bf A})_{ij}(x)=\int_0^1 a_{ij}(x,t){\rm d}t, \quad (\overline{\bf A})_{ij}(t)=\frac{1}{|\Omega|}\int_\Omega a_{ij}(x,t){\rm d}x, \quad (\overline{\widehat{\bf A}})_{ij}=\frac{1}{|\Omega|}\int_0^1\!\!\!\int_\Omega a_{ij}(x,t){\rm d}x.$$

We first show the following inequalities associated with  quantities defined in \eqref{def_underlineC} and \eqref{def_underlineC2}.
\begin{lemma}\label{liu-20240526}
  Let $\underline{C}$ and $C_*$ be defined by
   \eqref{def_underlineC} and let  $C_*^+$,
   $\underline{C}^+$ and $\overline{C}$ be defined by \eqref{def_underlineC2}. Then there holds
$\underline{C}\leq C_*\leq \underline{C}^+,C_*^+\leq \overline{C}$.
\end{lemma}
\begin{proof}
   The fact that $\underline{C}\leq C_*$ can be observed directly by definition \eqref{def_underlineC}. We first show $C_*\leq C_*^+$.
   A direct application of Theorem \ref{liutheorem-0229-1} yields
    \begin{equation}\label{liu-20240526-1}
         \lambda(\omega,\rho)\geq C(\omega/\sqrt{\rho}), \quad \forall \omega,\rho>0,
    \end{equation}
    where $C(\omega/\sqrt{\rho})$ is defined as the critical value of Hamilton-Jacobi
equation \eqref{Liu01-20210716} with $\vartheta=\omega/\sqrt{\rho}$. 
Letting $\omega\to +\infty$ in \eqref{liu-20240526-1}, by Propositions \ref{TH-liu-20240227} and  \ref{liu-20240326} one has $\overline{\lambda}(\rho)\geq C_*$ for all $\rho>0$, for which sending $\rho\to 0$ and applying \cite[Theorem 1]{D2009} yield $C_*^+\geq C_*$.

Next we show $C_*\leq \underline{C}^+$. Note that 
\begin{equation}\label{liu-20240528-3}
    \mu({\bf A}(x,t))=\max_{{\bm \phi}\geq 0}\frac{{\bm \phi}^{\rm T}{\bf A}(x,t){\bm \phi}}{{\bm \phi}^{\rm T}{\bm \phi}},\quad \forall (x,t)\in \Omega\times \mathbb{R}.
\end{equation}
 Let $\overline{\bm \phi}(t)\geq 0$ be the principal eigenvector of $\overline{\bf A}(t)$ associated to $\mu(\overline{\bf A}(t))$. Choose $\overline{\bm \phi}(t)$ as the test vector in \eqref{liu-20240528-3} and integrate the resulting inequality over $\Omega$, then
$$ \frac{1}{|\Omega|}\int_\Omega \mu({\bf A}(x,t)){\rm d}x \geq \frac{\overline{\bm \phi}^{\rm T}(t)\overline{\bf A}(t)\overline{\bm \phi}(t)}{\overline{\bm \phi}^{\rm T}\overline{\bm \phi}}=\mu(\overline{\bf A}(t)), \quad \forall t\in\mathbb{R},$$
and thus it follows by \eqref{def_underlineC} that
$$ C_*\leq -\frac{1}{|\Omega|}\int_0^1\!\!\int_\Omega \mu({\bf A}(x,t)){\rm d}x{\rm d }t \leq -\int_0^1 \mu(\overline{\bf A}(t)){\rm d }t=\underline{C}^+.$$
This proves $C_*\leq \underline{C}^+$.

Finally, we show $\underline{C}^+\leq \overline{C}$ and $C_*^+\leq \overline{C}$.  Let $\overline{h}(\omega)$ be the principal eigenvalue of  \eqref{liu-20240525-3}, serving as the limit of $\lambda(\omega,\rho)$ as $\rho\to+\infty$. Note from \cite[Theorem 2.1]{LLS2022} 
that
\begin{equation}\label{liu-20240529-1}
    \lim_{\omega\to 0}\overline{h}(\omega)=\underline{C}^+\quad\text{and}\quad \lim_{\omega\to +\infty}\overline{h}(\omega)=\overline{C}.
\end{equation}
Due to the symmetry of matrix ${\bf A}$, \cite[Theorem 1.1]{LLS2022} implies that $\overline{h}(\omega)$ is non-decreasing in $\omega$. This together with \eqref{liu-20240529-1} yields $\underline{C}^+\leq\overline{C}$. 
To prove $C_*^+\leq \overline{C}$, we recall that $\overline{\lambda}(\rho)$ denotes the principal eigenvalue of \eqref{liu-20240227-1}. By \cite[Theorem 1]{D2009} and Corollary \ref{cor-liu-20240429} we derive
\begin{equation}\label{liu-20240529-2}
   \lim_{\rho\to 0}\overline{\lambda}(\rho)=C_*^+\quad\text{and}\quad \lim_{\rho\to +\infty}\overline{\lambda}(\rho)=\overline{C}.
\end{equation}
Since matrix ${\bf A}$ is symmetric, by the variational structure for problem \eqref{liu-20240227-1},
 $\overline{\lambda}(\rho)$ is non-decreasing  in $\rho$, so that \eqref{liu-20240529-2} implies $C_*^+\leq\overline{C}$. This completes the proof.
\end{proof}

We next establish the following double limits of principal eigenvalue $\lambda(\omega,\rho)$.
\begin{lemma}\label{liu-20240516}
Let $\underline{C}$ be defined in
   \eqref{def_underlineC} and let  $C_*^+$,
   $\underline{C}^+$ and $\overline{C}$ be defined by \eqref{def_underlineC2}. Then $\underline{C}\leq \lambda(\omega,\rho)\leq \overline{C}$ for all $\omega,\rho>0$ and the following double limits hold:
 \begin{align*}
  \lim_{(\omega,\rho)\to (+\infty,0)}\lambda(\omega,\rho)= C_*^+, \quad
\lim_{(\omega,\rho)\to (0,+\infty)}\lambda(\omega,\rho)=\underline{C}^+, \quad
\lim_{(\omega,\rho)\to (+\infty,+\infty)}\lambda(\omega,\rho)=\overline{C}.
          \end{align*}
\end{lemma}
\begin{proof}
An  application of  Theorem \ref{TH1-1} and Proposition \ref{TH-liu-20240227}-{(i)} yields that for any fixed $\rho>0$,
\begin{equation}\label{liu-20240518-5}
    \lambda(\omega,\rho)\leq \lim_{\omega\to +\infty}\lambda(\omega,\rho)=\overline{\lambda}(\rho), \quad \forall \omega>0,
\end{equation}
     where $\overline{\lambda}(\rho)$ is the principal eigenvalue of elliptic problem \eqref{liu-20240227-1}. 
     Note that $\overline{\lambda}(\rho)$ is non-decreasing  in $\rho$, so that \eqref{liu-20240529-2} implies $\overline{\lambda}(\rho)\leq \overline{C}$ for all $\rho >0$. 
     This and \eqref{liu-20240518-5} imply
     $\lambda(\omega,\rho)\leq \overline{C}$ for all $\omega,\rho>0$.
     Similarly, it follows from Theorem \ref{TH1-1} and Proposition \ref{TH-liu-20240227}-{(ii)}  that 
  \begin{equation}\label{liu-20240518-3}
      \lambda(\omega,\rho)\geq \lim_{\omega\to 0}\lambda(\omega,\rho)=\int_0^1 \lambda_0(t,\rho){\rm d}t, \quad \forall \omega, \rho>0,
  \end{equation}
    where $\lambda_0(t,\rho)$ denotes the principal eigenvalue of \eqref{liu-20240227-2}. Since $\rho\mapsto\lambda_0(t,\rho)$ is non-decreasing  due to  the symmetry of
${\bf A}$, we apply the small diffusion limit established by \cite[Theorem 1]{D2009} or \cite[Theorem 1.4]{LL2016} to problem \eqref{liu-20240227-2}, then
$$\lambda_0(t,\rho)\geq \lim_{\rho\to 0}\lambda_0(t,\rho)=-\max_{x\in \overline{\Omega}}\mu({\bf A}(x,t)) \quad\text{for all }\rho>0 \text{ and }t\in \mathbb{R},$$
which together with \eqref{liu-20240518-3} implies      $\lambda(\omega,\rho)\geq \underline{C}$ for all $\omega,\rho>0$.

Hence, it remains to prove the double limits in Lemma \ref{liu-20240516}.

{\it Step 1}.    We first show $\lambda(\omega,\rho)\to C_*^+$ as $(\omega,\rho)\to (+\infty,0)$.
 Applying 
 \cite[Theorem 1]{D2009} or \cite[Theorem 1.4]{LL2016} to  \eqref{liu-20240227-1} yields  $\overline{\lambda}(\rho)\to\underline{C}^+$ as $\rho \to 0$.
    Hence, 
    letting $\rho \to 0$ in \eqref{liu-20240518-5} gives
    \begin{equation}\label{liu-20240518-1}
        \limsup_{(\omega,\rho)\to (+\infty,0)}\lambda(\omega,\rho)\leq \lim_{\rho\to 0}\overline{\lambda}(\rho)=C_*^+.
    \end{equation}
    Fix any $\omega_1>0$. Theorem \ref{TH1-1} implies $\lambda(\omega,\rho)\geq \lambda(\omega_1,\rho)$ for all $\omega\geq \omega_1$ and $\rho>0$, and thus by Theorem \ref{Bai-He-2020} we derive that
\begin{equation}\label{liu-20240518-2}
    \liminf_{(\omega,\rho)\to (+\infty,0)}\lambda(\omega,\rho)\geq \lim_{\rho\to 0}\lambda(\omega_1,\rho)=\underline{h}(\omega_1),
    \quad\forall \omega_1>0,
\end{equation}
    where $\underline{h}(\omega_1)$ is the principal eigenvalue of \eqref{liu-20240317-1} with $\omega=\omega_1$. 
Applying \cite[Theorem 2.1]{LLS2022} to \eqref{liu-20240317-1} 
gives $ \underline{h}(\omega_1)\to -\mu(\widehat{\bf A}(x))$ as $\omega_1\to +\infty$, and hence letting $\omega_1\to+\infty$  in \eqref{liu-20240518-2} gives
$$    \liminf_{(\omega,\rho)\to (+\infty,0)}\lambda(\omega,\rho)\geq C_*^+,$$
which and \eqref{liu-20240518-1} together   complete  Step 1. 

  {\it Step 2}.  We next show $\lambda(\omega,\rho)\to \underline{C}^+$ as $(\omega,\rho)\to (0,+\infty)$. 
 Recall that $\lambda_0(t,\rho)$ denotes the principal eigenvalue of \eqref{liu-20240227-2} for each $t\in \mathbb{R}$.
   We apply Proposition \ref{TH-liu-20240429}
    to \eqref{liu-20240227-2} 
    and derive $\lambda_0(t,\rho)\to -\mu(\overline{\bf A}(t))$ as $\rho\to+\infty$ for all $t\in\mathbb{R}$. This together with \eqref{liu-20240518-3} implies
    \begin{equation}\label{liu-20240518-4}
         \liminf_{(\omega,\rho)\to (0,+\infty)}\lambda(\omega,\rho)\geq \lim_{\rho\to +\infty}\underline{\lambda}(\rho)=\underline{C}^+.
    \end{equation}
 For any given $\omega_1>0$, we apply Theorem \ref{TH1-1} once again to obtain $\lambda(\omega,\rho) \leq \lambda(\omega_1,\rho)$ for all $\omega\in (0,\omega_1)$ and $\rho>0$. Hence, by Proposition \ref{TH-liu-20240429} we deduce
 \begin{equation}\label{liu-20240718-1}
     \limsup_{(\omega,\rho)\to (0,+\infty)}\lambda(\omega,\rho)\leq \lim_{\rho\to +\infty} \lambda(\omega_1,\rho)=\overline{h} (\omega_1),
 \end{equation}
where $\overline{h} (\omega_1)$ denotes the principal eigenvalue of
\eqref{liu-20240525-3} with $\omega=\omega_1$.
 Letting $\omega_1\to 0$ in \eqref{liu-20240718-1} and applying \cite[Theorem 2.1]{LLS2022} to \eqref{liu-20240525-3}, we have 
  $$\limsup_{(\omega,\rho)\to (0,+\infty)}\lambda(\omega,\rho)\leq\underline{C}^+.$$
  This and \eqref{liu-20240518-4} together establish Step 2.

  {\it Step 3}.    We finally prove $\lambda(\omega,\rho)\to \overline{C}$ as $(\omega,\rho)\to (+\infty, +\infty)$. Since $\lambda(\omega,\rho)\leq \overline{C}$ for all $\omega,\rho>0$, it suffices to show
  \begin{equation}\label{liu-20240518-7}
      \liminf_{(\omega,\rho)\to (+\infty,+\infty)}\lambda(\omega,\rho)\geq\overline{C}.
  \end{equation}
  Indeed, for any given $\omega_1>0$, it follows by Theorem \ref{TH1-1} that
  $\lambda(\omega,\rho)\geq \lambda(\omega_1,\rho)$ for all $\omega\geq \omega_1$ and $\rho>0$, which together with Proposition \ref{TH-liu-20240429} implies
  $$\liminf_{(\omega,\rho)\to (+\infty,+\infty)}\lambda(\omega,\rho)\geq \lim_{\rho\to+\infty}\lambda(\omega_1,\rho)=\overline{h} (\omega_1), \quad \forall\omega\geq \omega_1,$$
  for which, as in Step 2, letting $\omega_1\to +\infty$ and applying \cite[Theorem 2.1]{LLS2022} prove \eqref{liu-20240518-7}. The proof is now complete.
\end{proof}

 \begin{remark}
      The quantities in \eqref{def_underlineC} and \eqref{def_underlineC2} provide some double limits of $\lambda(\omega,\rho)$ 
      in Theorem \ref{TH-liu-20240106} and Lemma \ref{liu-20240516}; see Fig. \ref{liufig1}. While  the inequalities
    $\underline{C}\leq C_*\leq \min\{\underline{C}^+,C_*^+\}$ and $\max\{\underline{C}^+,C_*^+\} \leq \overline{C}$ hold as proved by Lemma \ref{liu-20240526},   
    there is no  specified order relationship between 
    $\underline{C}^+$ and $C_*^+$.
For instance,    
if ${\bf A}$ is independent of $t$, then  
$\underline{C}^+=\overline{C}> C_*^+$, and conversely, if ${\bf A}$ is independent of $x$, then  
$C_*^+=\overline{C}>\underline{C}^+$.

\begin{figure}[htb]
    \centering
    \includegraphics[width=0.7\linewidth]{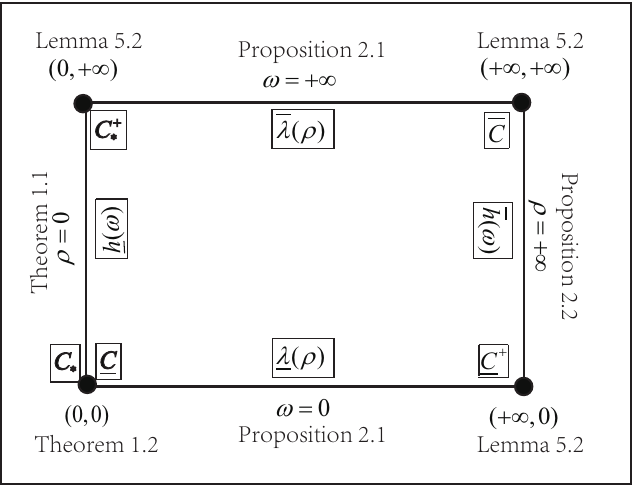}
    \caption{\small
    { The illustrations of the quantities defined in \eqref{def_underlineC} and \eqref{def_underlineC2}, which are some double limit values of the principal eigenvalue $\lambda(\omega,\rho)$ when $\omega$ tends to $0$ or $+\infty$ and/or $\rho$ tends to $0$ or $+\infty$.
    } }
  \label{liufig1}
   \end{figure}

In addition to the notations in \eqref{def_underlineC} and \eqref{def_underlineC2}, some limit values provided in Theorem \ref{Bai-He-2020}, Propositions \ref{TH-liu-20240227} and \ref{TH-liu-20240429} are  summarized in Fig. \ref{liufig1}, where
\begin{equation*}
\begin{split}
    \underline{h}(\omega)=\lim_{\rho\to 0}\lambda(\omega,\rho), \quad \overline{h}(\omega)=\lim_{\rho\to +\infty}\lambda(\omega,\rho), \quad \forall \omega>0,\\
    \underline{\lambda}(\rho)=\lim_{\omega\to 0}\lambda(\omega,\rho), \quad \overline{\lambda}(\rho)=\lim_{\omega\to +\infty}\lambda(\omega,\rho), \quad \forall \rho>0.
\end{split}
\end{equation*}
\end{remark}

Based on the above notations, we have the following result.

\begin{lemma}\label{d-ell}
For each 
$\rho>0$, let $\overline{\lambda}(\rho)$ and $\underline{\lambda}(\rho)$ be defined in Proposition {\rm \ref{TH-liu-20240227}}. 
   \begin{itemize}
       \item [{\rm (i)}] If $\underline{C}<\underline{C}^+$, then for $\ell\in (\underline{C},\underline{C}^+)$, there exists a unique $\rho_\ell>0$ such that $\underline{\lambda}(\rho_\ell)=\ell$. Moreover, $\rho_\ell$ is increasing in $\ell$, $\rho_\ell \to 0$ as $\ell\to \underline{C}$, and $\rho_\ell \to +\infty$ as $\ell\to \underline{C}^+$.
\smallskip
       \item [{\rm (ii)}] If $C_*^+<\overline{C}$, then  for $\ell\in (C_*^+,\overline{C})$, there exists a unique $\underline{\rho}_\ell \in (0,\rho_\ell]$ such that $\overline{\lambda}(\underline{\rho}_\ell)=\ell$. Moreover, $\underline{\rho}_\ell$ is increasing in $\ell$, $\underline{\rho}_\ell \to 0$ as $\ell\to C_*^+$, and $\underline{\rho}_\ell \to +\infty$ as $\ell\to \overline{C}$.
   \end{itemize}
  Moreover, $\rho_\ell=\underline{\rho}_\ell$ if and only if \eqref{n-s-condition} holds.
\end{lemma}
\begin{proof}
We shall prove part (i), and part (ii) can follow by the rather similar arguments.  A direct application of \cite[Theorem 1]{D2009} or \cite[Theorem 1.4]{LL2016} to problem \eqref{liu-20240227-2} yields  $\lambda_0(t,\rho)\to -\max_{x\in \overline{\Omega}} \mu({\bf A}(x,t))$ as $\rho \to 0$, where     $\lambda_0(t,\rho)$ denotes the principal eigenvalue of \eqref{liu-20240227-2}. 
Then applying Corollary \ref{cor-liu-20240429} to problem \eqref{liu-20240227-2}, we have $\lambda_0(t,\rho)\to -\mu(\overline{\bf A}(t))$ as $\rho \to +\infty$.
Hence, recall from the definitions in \eqref{def_underlineC} and  \eqref{def_underlineC2} that 
\begin{equation}\label{liu-33}
    \lim_{\rho\to 0}\underline{\lambda}(\rho)= \underline{C}\quad\text{and}\quad \lim_{\rho\to +\infty}\underline{\lambda}(\rho)= \underline{C}^+.
\end{equation}
By continuity, for each $\ell\in (\underline{C},\underline{C}^+)$, there exists some $\rho_\ell>0$ such that $\underline{\lambda}(\rho_\ell)=\ell$.

We next prove the uniqueness of such $\rho_\ell$. Suppose that   $\underline{\lambda}(\rho_*)=\ell$ for some $\rho_*\neq \rho_\ell$.  Without loss of generality, assume $\rho_*< \rho_\ell$. Due to the symmetry of matrix ${\bf A}(x,t)$, it follows from the variational structure of 
\eqref{liu-20240227-2} that  $\lambda_0(t,\rho)$ is increasing in $\rho$ if ${\bf A}$ is a non-constant matrix in $x$, and otherwise  $\lambda_0(t,\rho)$ is constant in $\rho$. Hence,
  $\lambda_0(t,\rho_*)\leq \lambda_0(t,\rho_\ell)$ for all  $t\in\mathbb{R}$.
 By our assumption $\underline{\lambda}(\rho_*)=\underline{\lambda}(\rho_\ell)=\ell$, it follows that
 $\lambda_0(t,\rho_*)= \lambda_0(t,\rho_\ell)$ for all  $t\in\mathbb{R}$.
This implies $\lambda_0(t,\rho)=\lambda_0(t)$ independent of $\rho$, and thus $\underline{\lambda}(\rho)\equiv\ell$ for all $\rho>0$. By \eqref{liu-33}, one has  $\underline{C}=\underline{C}^+=\ell$, contradicting our assumption $\underline{C}<\underline{C}^+$. This proves the uniqueness of $\rho_\ell$.

Since $\underline{\lambda}(\rho)$ is non-decreasing in $\rho$,  the uniqueness of $\rho_\ell$ implies that $\underline{\lambda}(\rho)$ is increasing in $\rho$, and thus by definition,  $\rho_\ell$ is increasing in $\ell$. Moreover, due to \eqref{liu-33}, it is clear that $\rho_\ell\to 0$ as $\ell\to\underline{C}$ and $\rho_\ell\to +\infty$ as $\ell\to\underline{C}^+$.

Finally, the fact that $\rho_\ell=\underline{\rho}_\ell$ if and only if \eqref{n-s-condition} holds is a direct consequence of Theorem \ref{TH1}. Lemma \ref{d-ell} is thus proved.
\end{proof}

For each value
$\ell\in (\underline{C}, \overline{C})$, 
define the level set
$$
\Sigma_\ell:=\{(\omega,\rho): \, \omega,\rho>0, \quad  \lambda(\omega,\rho)=\ell\}.
$$
Our aim is to  characterize  the topological structures of all level sets. 


\begin{theorem}\label{liu-levelset}
For any $\ell\in(\underline{C},\overline{C})$, let constants $0<\underline{\rho}_\ell\leq \rho_\ell$ be defined  by Lemma {\rm\ref{d-ell}} such that $\underline{\lambda}(\rho_\ell)=\overline{\lambda}(\underline{\rho}_\ell)=\ell$.
If $\underline{\rho}_\ell= \rho_\ell$, then
$\Sigma_\ell=\{(\omega,\rho): \, \rho=\rho_\ell, \,\, \,\omega>0 \}$; Otherwise
there exist uniquely a continuous function $\omega_\ell:{\rm dom}(\omega_\ell)\mapsto (0,+\infty)$ such that
$$\Sigma_\ell=\{(\omega,\rho): \, \omega=\omega_\ell(\rho), \,\, \,\rho\in {\rm dom}(\omega_\ell) \}.$$
 The domain ${\rm dom}(\omega_\ell)$ and
  asymptotic behaviors of function $\omega_\ell$  can be characterized as follows.


\begin{itemize}
    \item[{\rm(1)}] If
    $\ell\in (\underline{C}, \min\{C_*^+,\underline{C}^+\})$,
    then $ {\rm dom}(\omega_\ell)= (0,\rho_\ell)$ and
$\omega_\ell(\rho)\to 0$ as $\rho \nearrow \rho_\ell$. 
The asymptotic behaviors of $\omega_\ell(\rho)$ at  $\rho=0$ are 
as follows.
 \smallskip
\begin{itemize}
 \item [{\rm(i)}] If $\ell\in (\underline{C}, C_*]$, 
 then
 $\omega_\ell(\rho)\to 0$ as $\rho\searrow 0$. Furthermore, if $\ell\in (\underline{C}, C_*)$,
 there exist $0<{\underline c}<{\overline c}$ independent of $\rho$ such that
    ${\underline c} \sqrt{\rho}\leq \omega_\ell(\rho)\leq {\overline c} \sqrt{\rho}$
 for small $\rho>0$.

 \smallskip
 \item [{\rm(ii)}] If $\ell\in (C_*,\min\{C_*^+,\underline{C}^+\})$, then 
 $\omega_\ell(\rho)\to \underline{h}^{-1}(\ell)$ as $\rho\searrow 0$. 
\end{itemize}

\medskip
\item [{\rm(2)}]   If $C_*^+<\underline{C}^+$ and $\ell\in (C_*^+,\underline{C}^+)$, then
$ {\rm dom}(\omega_\ell)=(\underline{\rho}_\ell,\rho_\ell)$,
and function
 $\omega_\ell$
 satisfies
$\omega_\ell(\rho)\to 0$ as $\rho\nearrow\rho_\ell$   and $\omega_\ell(\rho)\to +\infty$ as $\rho\searrow\underline{\rho}_\ell$. 

\medskip
\item [{\rm(3)}]   If $\underline{C}^+<C_*^+$ and $\ell\in [\underline{C}^+,C_*^+)$, then $ {\rm dom}(\omega_\ell)= (0,+\infty)$,  and
 function 
 $\omega_\ell$ satisfies
$\omega_\ell(\rho)\to \underline{h}^{-1}(\ell)$ as $\rho\searrow 0$  and $\omega_\ell(\rho)\to \overline{h}^{-1}(\ell)$ as $\rho\nearrow +\infty$.

\medskip
\item[{\rm (4)}] If $ \ell \in (\max\{C_*^+,\underline{C}^+\}, \overline{C})$, then $ {\rm dom}(\omega_\ell)=(\underline{\rho}_\ell, +\infty)$, and
function 
$\omega_\ell$ satisfies 
$\omega_\ell(\rho)\to +\infty$ as $\rho\searrow\underline{\rho}_\ell$ and $\omega_\ell(\rho)\to \overline{h}^{-1}(\ell)$ as $\rho\nearrow +\infty$.
\end{itemize}
\end{theorem}

Theorem \ref{liu-levelset} yields five types of topological structures for level sets $\Sigma_\ell$,   
which are illustrated by various shaded areas in Fig. \ref{liufig2}. The pictures presented in Fig. \ref{liufig2} are focused on  the scenarios where all quantities in \eqref{def_underlineC} and \eqref{def_underlineC2} are distinct. 

 \begin{figure}[htb]
    \centering
\includegraphics[width=1\linewidth]{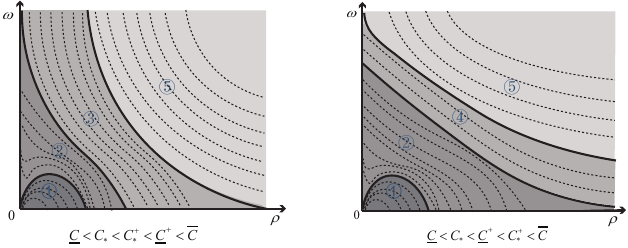}
    \caption{\small
    {A sketch of the  level  sets  of principal eigenvalue $\lambda(\omega,\rho)$ in  $\rho$-$\omega$  plane as shown in Theorem \ref{liu-levelset}, which  are illustrated by the bounded or unbounded curves presenting the graphs of function $\omega=\omega_{\ell}(\rho)$  for the cases     $\underline{C}< C_*<C_*^+<\underline{C}^+ <\overline{C}$ and $\underline{C}< C_*<\underline{C}^+<C_*^+< \overline{C}$.  
    } }
  \label{liufig2}
   \end{figure}

\begin{itemize}
  \item[\textcircled{1}]
 For
 $\ell\in (\underline{C},C_*)$,
   part (1)-{(i)} of Theorem \ref{liu-levelset}
  asserts  all level sets $\Sigma_\ell$  are characterized by bounded curves
   connecting  $(0,0)$ and $(\rho_\ell, 0)$ in $\rho$-$\omega$ plane as illustrated by  curves in  
   areas \textcircled{1}  of Fig. \ref{liufig2}.
   The estimate in part {\rm(1)}-{\rm(i)}  indicates the curve $\omega=\omega_\ell(\rho)$ has asymptotic behavior of the form   $\omega\sim \sqrt{\rho}$ at the origin. 
\smallskip
   \item [\textcircled{2}]
  For   $\ell\in (C_*,\min\{C_*^+,\underline{C}^+\})$,
   part (1)-{(ii)} in  Theorem \ref{liu-levelset}
  suggests level sets $\Sigma_\ell$  are characterized by bounded curves
   connecting   points $(\rho_\ell,0)$ and $(0,\underline{h}^{-1}(\ell))$, which are illustrated by  curves in  
   areas \textcircled{2} of Fig. \ref{liufig2}.

\smallskip
\item[\textcircled{3}]
For $C_*^+<\underline{C}^+$ and $\ell\in (C_*^+,\underline{C}^+)$, by part {\rm (2)} of Theorem \ref{liu-levelset},  the level sets $\Sigma_\ell$ will connect $(\rho_\ell,0)$ and approach asymptotically to the vertical line $\rho=\underline{\rho}_\ell$, which are illustrated by
    curves in  unbounded 
    areas \textcircled{3} in the left figure of Fig. \ref{liufig2}

 \smallskip
\item[\textcircled{4}]
For $\underline{C}^+<C_*^+$ and $\ell\in (\underline{C}^+,C_*^+)$, the level sets $\Sigma_\ell$ are illustrated by
    curves in  unbounded
    areas \textcircled{4} in the right figure of Fig. \ref{liufig2}. 
   In particular,
$\Sigma_\ell$ will connect $(0, \underline{h}^{-1}(\ell))$ and approach asymptotically to the horizontal line  $\omega=\underline{h}^{-1}(\ell)$ as proved by part {\rm (3)} of Theorem \ref{liu-levelset}.
\smallskip

 \item[\textcircled{5}]
For $\ell\in (\max\{C_*^+,\underline{C}^+\},\overline{C})$,
   part (4) of Theorem \ref{liu-levelset}
implies the level sets $\Sigma_\ell$ are
given  by unbounded curves
approaching asymptotically to both the vertical line $\rho=\underline{\rho}_\ell$ and the horizontal line  $\omega=\underline{h}^{-1}(\ell)$, see 
areas \textcircled{5} of Fig. \ref{liufig2}.

\end{itemize}

For any $\ell \in (\underline{C},{\overline{C}})$,  the topological structure of level sets $\Sigma_\ell$  is a combination of these five types, which are separated by the level sets $\Sigma_{C_*}$,   $\Sigma_{C_*^+}$ and $\Sigma_{\underline{C}}$. In particular, if $C_*=C_*^+$ and $\underline{C}^+ =\overline{C}$, then the corresponding  structure of level sets is analogue to that of scalar case   in \cite{LL2022} and spatially homogeneous case in \cite{LL2024}.

 We are in a position to prove Theorem \ref{liu-levelset}.

\begin{proof}[Proof of Theorem {\rm \ref{liu-levelset}}]

Let 
$\overline{\lambda}(\rho)$ be the principal eigenvalue of problem \eqref{liu-20240227-1}
in Proposition \ref{TH-liu-20240227},
which serves as the limit of $\lambda(\omega,\rho)$ as 
$\omega\to+\infty$. 
By our assumption, $\widehat{\bf A}(x)$ is a symmetric matrix for all $x\in \Omega$. Hence, by  the variational structure of \eqref{liu-20240227-1},   $\overline{\lambda}(\rho)$ is non-decreasing in $\rho$, and either $\overline{\lambda}'(\rho)>0$ for all $\rho>0$, or $\overline{\lambda}'(\rho)\equiv 0$.  It was shown in \cite[Theorem 1]{D2009} or \cite[Theorem 1.4]{LL2016} 
that
$\overline{\lambda}(\rho)\to C_*^+$ as $\rho\to 0$, and hence 
\begin{equation}\label{liu-12-1}
C_*^+ \leq \overline{\lambda}(\rho) \leq \overline{C}\quad \text{for all }\, \rho>0,
\end{equation}
and  $\overline{\lambda}(\rho)$ is strictly increasing in $\rho$ if $C_*^+  <\overline{C}$.

\smallskip
 {\it Step 1.} We assume   $\ell\in (\underline{C}, \min\{C_*^+,\underline{C}^+\})$  and prove part {\rm(1)}.
In view of $\underline{C}<\underline{C}^+$, let  $\rho_\ell$ be determined by Lemma \ref{d-ell}. 
Since $\underline{\lambda}(t, \rho)$ is non-decreasing in $\rho$,  the uniqueness of $\rho_\ell$ implies
\begin{equation}\label{liu-13}
 \underline{\lambda}(\rho)<\underline{\lambda}(\rho_\ell)=\ell \quad \text{for any }\, \rho\in (0,\rho_\ell). 
\end{equation}
This together with  \eqref{liu-12-1} implies that $\underline{\lambda}(\rho)<\ell\leq C_*^+ \leq \overline{\lambda}(\rho)$ for all  $\rho\in (0, \rho_\ell)$.
Applying the monotonicity 
established in Theorem \ref{TH1-1}, we find $\lambda(\omega, \rho)$ is strictly increasing in $\omega$ for any $\rho\in (0\rho_\ell)$, 
and hence by Proposition \ref{TH-liu-20240227}, there is a unique continuous  non-negative function $\omega_\ell(\rho):(0, \rho_\ell)\mapsto [0,\infty)$ such that 	
\begin{equation}\label{liu-20240404-1}
    \lambda(\omega_\ell(\rho),\rho)\equiv \ell, \quad \forall \rho \in (0, \rho_\ell).
\end{equation}
where the continuity follows from the implicit function theorem.
By \eqref{liu-13}, it is easily seen that $\omega_{\ell}(\rho_\ell)=0$ and  $\omega_{\ell}(\rho_\ell)>0$ for $\rho\in (0, \rho_\ell)$. We shall prove part {\rm(1)} by
considering the following two cases:

Case 1: We assume $\ell\in (\underline{C}, C_*]$ 
to  establish part {\rm(1)}-{\rm(i)}. We first claim $\omega_\ell(\rho)\to 0$ as $\rho\to 0$. If not, then there exists some sequence $\{\rho_n\}_{n\geq 1}$ such that $\rho_n\to 0$ and $\omega_\ell(\rho_n)\to \omega_*$ as $n\to +\infty$ for some $\omega_*\in(0,+\infty]$. Then we shall reach a contradiction by the following two cases:
{\rm(i)} If $\omega_*=+\infty$, then a direct application of Lemma \ref{liu-20240516} yields $\lambda(\omega_\ell(\rho_n),\rho_n)\to C_*^+$
as $n\to +\infty$. Note from \eqref{liu-20240404-1} that  
$\lambda(\omega_\ell(\rho_n),\rho_n)\equiv \ell$
for all $n$, so that $\ell=C_*^+$, which is a contradiction since $\ell< \min\{C_*^+,\underline{C}^+)$ 
by our assumption.
 {\rm(ii)} If $\omega_*<+\infty$, it follows by Theorem \ref{Bai-He-2020} that $\lambda(\omega_\ell(\rho_n),\rho_n)\to \underline{h}(\omega_*)$
as $n\to+\infty$. 
Since $\lambda(\omega_\ell(\rho_n),\rho_n)\equiv \ell$ for all $n$,
one has $\ell=\underline{h}(\omega_*)\leq C_*$. Due to the symmetry of ${\bf A}$, we can apply the monotonicity result in \cite[Theorem 1.1]{LLS2022} to \eqref{liu-20240317-1} and obtain that either $\underline{h}(\omega)$ is strictly increasing in $\omega$, or  $\underline{h}(\omega)$ is constant in $\omega$. Due to $\underline{h}(\omega)\to C_*$ as $\omega\to 0$ (by applying \cite[Theorem 2.1]{LLS2022}), it holds $\underline{h}(\omega_*)= C_*$, and
furthermore  $\underline{h}(\omega)\equiv C_*$ for all $\omega>0$. By $\ell=\underline{h}(\omega_*)$ one has $\ell=C_*<C_*+$ by our assumption. In particular, according to the fact that $\underline{h}(\omega)\to C_*^+$ as $\omega\to +\infty$, we derive $C_*=C_*^+$, which is a contradiction.  Therefore, $\omega_\ell(\rho)\to 0$ as $\rho\to 0$.


It remains to prove that for the case $\ell\in(\underline{C},C_*)$, it holds $c_1{\sqrt{\rho}}\leq {\omega_\ell(\rho)}
\leq c_2 \sqrt{\rho}$ for some constants $0<c_1<c_2$ independent of $\rho$. We first claim ${\omega_\ell(\rho)}
\geq c_1{\sqrt{\rho}}$  holds for some constant $c_1>0$.
If not,  then there exists some sequence $\{\underline{\rho}_n\}_{n=1}^\infty$ such that $\omega_\ell(\underline{\rho}_n)/\sqrt{\underline{\rho}_n}\to 0$ and $\underline{\rho}_n\to 0$ as $n\to+\infty$. Then a direct application of Theorem \ref{TH-liu-20240106} yields $\lambda(\omega_\ell(\underline{\rho}_n), \underline{\rho}_n)\to \underline{C}$ as $n\to+\infty$, so that $\ell=\underline{C}$, which is a contradiction.  We next show $\omega_\ell(\rho)\leq c_2 \sqrt{\rho}$ for some $c_2>0$ independent of $\rho$.
Suppose on the contrary that  there exists some sequence $\{\rho_n\}_{n=1}^\infty$ such that $\overline{\rho}_n\to 0$ and $\omega_\ell(\overline{\rho}_n)/\sqrt{\overline{\rho}_n}\to+\infty$  as $n\to+\infty$. Then
Theorem \ref{TH-liu-20240106} implies $\lambda(\omega_\ell(\overline{\rho}_n), \overline{\rho}_n)\to C_*$ as $n\to+\infty$, which contradicts $\lambda(\omega_\ell(\overline{\rho}_n), \overline{\rho}_n)\equiv \ell< C_*$. 
Therefore, 
$0<c_1\leq \omega_\ell(\rho)/\sqrt{\rho}\leq c_2$.
 Part {\rm(1)}-{\rm(i)} is thus proved. 

Case 2: We assume $\ell\in ( C_*,\min\{C_*^+,\underline{C}^+\})$  to  establish part {\rm(1)}-{\rm(ii)}. Due to $C_*<C_*^+$, 
we apply \cite[Theorem 1.1]{LLS2022} to problem \eqref{liu-20240317-1} and derive that $\underline{h}(\omega)$ is strictly increasing in $\omega$, with   $\underline{h}(\omega)$ being  defined as the limit of $\lambda(\omega,\rho)$ as $\rho \to 0$, see Theorem \ref{Bai-He-2020}. Hence, there is the unique $\omega=\underline{h}^{-1}(\ell)$ such that $\underline{h}(\omega)=\ell$. Fix any sequence  $\{\rho_n\}_{n=1}^\infty$
such that $\rho_n\to 0$ as $n\to+\infty$. If
$\omega_{\ell}(\rho_n)\to \omega_*$ for some $\omega_*\in [0,+\infty]$, then due to $\ell\in ( C_*, C_*^+)$, 
by Theorem \ref{TH-liu-20240106} and Lemma \ref{liu-20240516} we have $\omega_*\in (0,+\infty)$. 
Hence, Theorem \ref{Bai-He-2020} implies $\lambda(\omega_\ell(\rho_n),\rho_n)\to \underline{h}(\omega_*)$ as $n\to+\infty$, which together with \eqref{liu-20240404-1} implies $\ell=\underline{h}(\omega_*)$, and thus $\omega_*=\underline{h}^{-1}(\ell)$. This proves part {\rm(1)}-{\rm(ii)}.

\smallskip
 {\it Step 2.} We assume  $\ell\in (C_*^+,\underline{C}^+)$ and $C_*^+<\underline{C}^+$ to  prove part {\rm(2)}.
In view of $\underline{C}\leq C_*<\underline{C}^+$, let
$\rho_\ell>0$ be given by Lemma \ref{d-ell}.
Note that $\underline{\lambda}(\rho_\ell)=\ell$ as given in Lemma \ref{d-ell}. Let ${\bm \phi}>0$ be the  principal eigenfunction of 
\eqref{liu-20240227-1} with $\rho=\rho_\ell$. If ${\bf A}{\bm \phi}=\widehat{\bf A}(x){\bm\phi}+g(t){\bm\phi}$ for some $1$-periodic function $g$,
Theorem \ref{TH1}{\rm(ii)} yields $\lambda(\omega,\rho_\ell)\equiv \ell$ for all $\omega>0$; Otherwise, $\lambda(\omega,\rho_\ell)$ is strictly increasing in $\omega>0$, so that $\overline{\lambda}(\rho_\ell)>\underline{\lambda}(\rho_\ell)= \ell$, which implies
 $\underline{\rho}_\ell< \rho_\ell$. For each $\rho\in (\underline{\rho}_\ell, \rho_\ell)$,  by \eqref{liu-13} and the monotonicity of $\overline{\lambda}(\rho)$,  it holds that $\underline{\lambda}(\rho)< \overline{\lambda}(\underline{\rho}_\ell)<\overline{\lambda}(\rho)$. Hence, by  the monotonicity of $\omega\mapsto\lambda(\omega,\rho)$ in Theorem \ref{TH1-1}, there is a unique continuous  non-negative function $\omega_\ell(\rho):(\underline{\rho}_\ell, \rho_\ell]\to [0,\infty)$ such that	
 \begin{equation}\label{liu-20240404-3}
    \lambda(\omega_\ell(\rho),\rho)\equiv \ell, \quad \forall \rho \in (\underline{\rho}_\ell, \rho_\ell).
\end{equation}
Clearly,  $\omega_{\ell}(\rho_\ell)=0$ and  $\omega_{\ell}(d)>0$ for $\rho\in (\underline{\rho}_\ell, \rho_\ell)$.
To part {\rm(2)}, it remains to show $\omega_{\ell}(\rho)\to+\infty$ as $\rho\searrow  \underline{\rho}_\ell$.
If not, then there exist some $\omega_*\geq 0$ and sequence $\{\rho_n\}_{n=1}^{\infty}$ such that $\rho_n \searrow \underline{\rho}_\ell$ and $\omega_{A}(\rho_n)\to \omega_*$    as $n\to+\infty$.
By \eqref{liu-20240404-3}, one has $\lambda(\omega_*,\underline{\rho}_\ell)=\ell$, and thus  $\overline{\lambda}(\underline{\rho}_\ell)=\lambda(\omega_*,\underline{\rho}_\ell)$. Applying Theorem \ref{TH1-1} once again yields  $\lambda(\omega, \underline{\rho}_\ell)\equiv \ell$ for all $\omega>0$, which contradicts $\underline{\lambda}(\underline{\rho}_\ell)< \ell$ as stated in \eqref{liu-13}.  This proves part {\rm(2)}. 

\smallskip
 {\it Step 3.} We assume  $\ell\in (\underline{C}^+,C_*^+)$ and $\underline{C}^+<C_*^+$ to  prove part {\rm(3)}. We observe that
$$\underline{\lambda}(\rho)\leq \underline{C}^+<\ell<C_*^+\leq \overline{\lambda}(\rho), \quad \forall \rho>0.$$
Hence, Theorem \ref{TH1-1} implies that $\lambda(\omega,\rho)$ is strictly increasing in $\omega$, and
there is a unique continuous  function $\omega_\ell(\rho):(0, +\infty)\to [0,\infty)$ such that	
$\lambda(\omega_\ell(\rho),\rho)\equiv \ell$ for all $\rho>0$.

Let $\underline{h}(\omega)$ be defined in Theorem \ref{Bai-He-2020}, serving as the limit of $\lambda(\omega,\rho)$ as $\rho\to 0$. Note that
\begin{equation}\label{liu-20240525-2}
   \lim_{\omega\to 0}\underline{h}(\omega)= C_*\leq \underline{C}^+<\ell<C_*^+=\lim_{\omega\to +\infty}\underline{h}(\omega).
\end{equation}
In view of $C_*\leq \underline{C}^+<C_*^+$, as in Step 1,  we may apply \cite[Theorem 1.1]{LLS2022} 
to derive that $\underline{h}(\omega)$ is strictly increasing in $\omega$. This  together with \eqref{liu-20240525-2} yields $\underline{h}^{-1}(\ell)$ exists. By similar arguments as in Case 2 of Step 1, we find   $\omega_\ell(\rho) \to \underline{h}^{-1}(\ell)$ as $\rho \to 0$.

It remains to show $\omega_\ell(\rho) \to \overline{h}^{-1}(\ell)$ as $\rho \to +\infty$, where $\overline{h}(\omega)$ is the limit of $\lambda(\omega,\rho)$ as $\rho\to+\infty$; see Proposition \ref{TH-liu-20240429}. Similar to \eqref{liu-20240525-2}, we find
\begin{equation*}
   \lim_{\omega\to 0}\overline{h}(\omega)= \underline{C}^+<\ell<C_*^+\leq \overline{C}=\lim_{\omega\to +\infty}\overline{h}(\omega).
\end{equation*}
Hence,  by applying \cite[Theorem 1.1]{LLS2022} to problem \eqref{liu-20240525-3} yields that $\overline{h}(\omega)$ is strictly increasing in $\omega$, so that $\overline{h}^{-1}(\ell)$ exists uniquely. Given any sequence  $\{\rho_n\}_{n=1}^\infty$
such that $\rho_n\to +\infty$ as $n\to+\infty$. Assume
$\omega_{\ell}(\rho_n)\to \omega_*$ for some $\omega_*\in [0,+\infty]$, then due to $\ell\in ( \underline{C}, \overline{C})$, 
by Theorem \ref{TH-liu-20240106} and Lemma \ref{liu-20240516} we have $\omega_*\in (0,+\infty)$.
Hence, Proposition \ref{TH-liu-20240429} implies $\lambda(\omega_\ell(\rho_n),\rho_n)\to \overline{h}(\omega_*)$ as $n\to+\infty$.
In view of $\lambda(\omega_\ell(\rho),\rho)\equiv \ell$, this implies  $\ell=\overline{h}(\omega_*)$, and thus $\omega_*=\overline{h}^{-1}(\ell)$. This completes Step 3.

\smallskip
\noindent {\it Step 4.} We assume  $ \ell \in (\max\{C_*^+,\underline{C}^+\}, \overline{C})$ and   prove part {\rm(4)}. Since $C_*^+<\overline{C}$, let $\underline{\rho}_\ell>0$ be defined by Lemma \ref{d-ell}.
Due to $\overline{\lambda}(\underline{\rho}_\ell)=\ell$, the monotonicity of $\overline{\lambda}(\rho)$ implies $\overline{\lambda}(\rho)>\ell$ for all $\rho>\underline{\rho}_\ell$, and thus
$\underline{\lambda}(\rho)\leq \underline{C}^+<\ell< \overline{\lambda}(\rho)$ for all $\rho>\underline{\rho}_\ell$. This together with
Theorem \ref{TH1-1} implies that $\lambda(\omega,\rho)$ is increasing in $\omega$, and thus
there is a unique positive continuous  function $\omega_\ell(\rho):(\underline{\rho}_\ell, +\infty)\mapsto [0,\infty)$ such that	
$\lambda(\omega_\ell(\rho),\rho)\equiv \ell$ for all $\rho\in (\underline{\rho}_\ell,+\infty)$.
Then by the same arguments in Steps 2 and 3, we can show
$\omega_{\ell}(\rho)\to+\infty$ as $\rho\to  \underline{\rho}_\ell$ and $\omega_{\ell}(\rho)\to \overline{h}^{-1}(\ell)$ as $\rho\to  +\infty$. Therefore, part (4) is proved.

The proof of Theorem \ref{liu-levelset} is now complete.
\end{proof}

For the temporally constant case, i.e. matrix {\bf A} is independent of $t$,  it can be observed that $\underline{C}=C_*=C_*^+$.  The variational structure of \eqref{Liu1} indicates  that principal eigenvalue is monotone non-increasing
in $\rho$. In contrast, Theorem \ref{liu-levelset} turns out to imply some non-monotone dependence of principal eigenvalue
on 
$\rho$ when $\underline{C}<C_*$.


\begin{corollary}
    \label{liucor-1}
For each $x\in\Omega$ and $\omega>0$, let $h(x,\omega)$ be the principal eigenvalue of
\eqref{liu-20240317-1}.
Assume   $\underline{C}<C_*$, then the followings hold.

\begin{itemize}
\item [{\rm(i)}]  If   $C_*=\overline{C}$,
then for each $\omega>0$, $\lambda(\omega,\rho)$ attains its global minimum exactly
at both $\rho=0$ and $\rho=+\infty$.
\smallskip

\item [{\rm(ii)}] If   $C_*<\overline{C}$ and  for each $\omega>0$, function $h(x,\omega)$ admits finite number of strict local minimal points, all of which are  non-degenerate, i.e. the associated Hessian matrices are positive defined. 
 Then  for any $\omega>0$,   $\lambda(\omega, \rho)$ attains a local minimum at $\rho=0$,
and there exists some $\omega_*>0$
such that for each $0<\omega<\omega_*$,
there exist two positive constants
$0<\underline{\rho}<\overline{\rho}$ such that
\begin{itemize}
\item [{\rm(a)}]
$\lambda(\omega, \underline{\rho})=\lambda(\omega, \overline{\rho})
=C_*$,
\item [{\rm(b)}]
$\lambda(\omega, \rho)>C_*$ for all
$\rho\in (0, \underline{\rho})\cup (\overline{\rho}, +\infty)$, and
\item [{\rm(c)}]
{$\lambda(\omega, \rho)
< C_*$
for some $\rho\in (\underline{\rho}, \overline{\rho})$.}

\end{itemize}
 In particular, $\lambda(\omega,\rho)$ attains its 
 global minimum in $(\underline{\rho}, \overline{\rho})$.


\end{itemize}
\end{corollary}

Corollary \ref{liucor-1} implies that
for temporally varying environment, the
principal eigenvalue $\lambda(\omega,\rho)$ may not
depend on 
diffusion rate $\rho$
monotonically, which is illustrated in Fig. \ref{liufig3}. It could occur that $\lambda(\omega,\rho)$ is monotone decreasing for some ranges of
$\rho$, and the global minimum of $\lambda(\omega,\rho)$ is attained at some intermediate value of $\rho$. 
 This suggests the effect of diffusion rate  on the principal eigenvalue of \eqref{Liu1}
can be rather complicate in contrast  to  temporally constant case.

\begin{figure}[htb]
  \centering
  \includegraphics[width=0.95\linewidth]{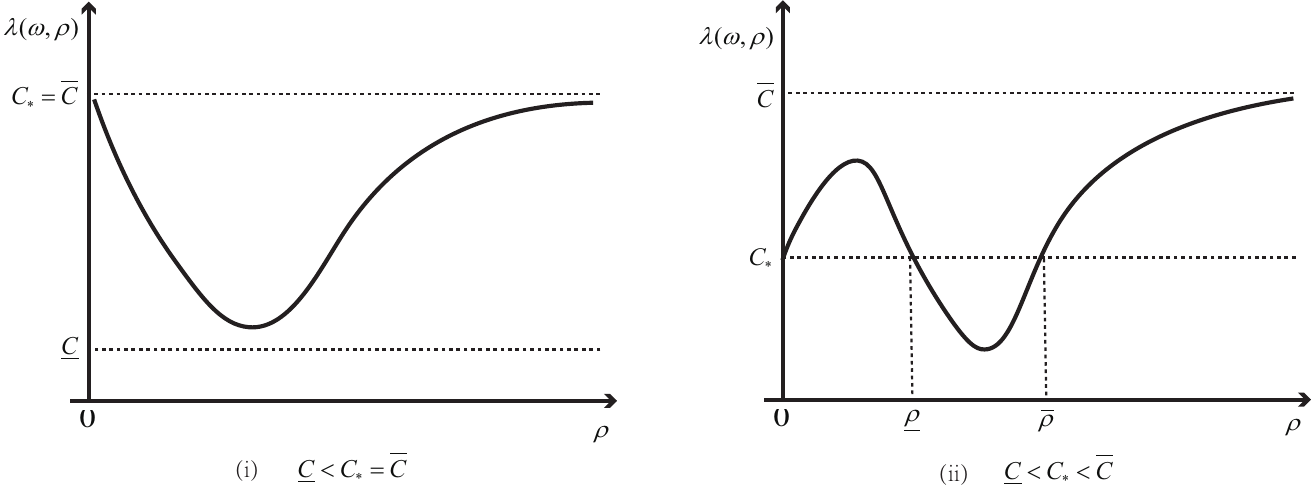}
  \caption{\small
  {
  Graphs of $\lambda(\omega,\rho)$
  as functions of $\rho$, with fixed $\omega$, for the cases given in Theorem  \ref{liucor-1}.
  These figures are for illustration purpose only, as the exact shapes of the graphs could be more complex.
  }
}
\label{liufig3}
 \end{figure}
In what follows,
 we are devoted into proving 
 Corollary \ref{liucor-1}. 
We first establish the following lower bound of $\lambda(\omega, \rho)$,
which  implies that for each fixed $\omega>0$, $\lambda(\omega, \rho)$ attains a local minimum at $\rho=0$.

\begin{lemma}\label{liulem-4}
For any $\omega>0$ and $x\in\overline{\Omega}$, let $h(x, \omega)$ be the principal eigenvalue of  problem \eqref{liu-20240317-1}. Suppose the assumptions in  Corollary {\rm \ref{liucor-1}}-{\rm(ii)} hold. 
Then for any $\omega>0$,   there exist some $\eta,\rho_0>0$ independent of $\rho>0$ such
that
\begin{equation}\label{eq:lowerbound}
\lambda(\omega, \rho)>\underline{h}(\omega)+\eta  \sqrt{\rho}, \quad \forall \rho<\rho_0,
\end{equation}
where $\underline{h}(\omega)=\min_{x\in\overline{\Omega}}h(x,\omega)$ is the limit of $\lambda(\omega,\rho)$ as $\rho\to 0$ given by Theorem {\rm \ref{Bai-He-2020}}.
\end{lemma}
\begin{proof}
Fix any $\omega>0$ throughout the proof. By our assumption,
 let $x_k\in \Omega$ ($1\leq k\leq N$) be the local minimal point such that  $h(x_k,\omega)=\underline{h}(\omega)$ and $h(x,\omega)>\underline{h}(\omega)$
 for all $x\in \overline\Omega\setminus\{x_1,\cdots, x_N\}$, as well as  each point
 $x_k$ is a non-degenerate critical point.
Hence, we can choose $v\in C^2(\overline\Omega)$ such that for every $1\leq k\leq N$,
$$\nabla v(x_k)=0, \quad\Delta v(x_k)>0, \quad \nabla v\cdot\nu>0 \text{ on } \partial\Omega,$$ and
\begin{equation}\label{eq:liu-temp_1-1}
   h(x,\omega)-\underline{h}(\omega)-\overline{d}|\nabla v|^2
    =\left\{
    \begin{array}{ll}
    \smallskip
        +,
     &x\not=x_k,\quad k=1,\cdots, N,\\
    0,   &x=x_k,\quad k=1,\cdots, N,
    \end{array} \right.
    \end{equation}
    where $\overline{d}=\max\{d_1,\cdots,d_n\}$.
Denote by ${\bm \phi}=(\phi_1,\cdots,\phi_n)>0$ the principal eigenvector of problem \eqref{liu-20240317-1}   corresponding to principal eigenvalue $h(x,\omega)$. 
We define
$$
\overline{\bm \varphi}(x,t):
={\bm \phi} (x,t)\exp\left\{-\frac{v(x)}{\sqrt{\rho}}\right\}, \quad \forall (x,t)\in \overline\Omega\times \mathbb{R}.
$$
Next, we will claim there exist  constants $\eta>0$ and $\rho_0>0$ such that for all $\rho<\rho_0$,  
\begin{equation}\label{liu-50-1}
\begin{cases}
\begin{array}{ll}
\smallskip
\omega\partial_{t}\overline{\bm \varphi}-\rho {\bf D}\Delta\overline{\bm \varphi}-{\bf A}\overline{\bm \varphi}\geq (\underline{h}(\omega)+\eta \sqrt{\rho})\overline{\bm \varphi}  &\text{in }\,\,\Omega\times \mathbb{R},\\
\smallskip
\nabla \overline{\bm \varphi}\cdot\nu\geq 0 &\text{on }\,\partial\Omega\times \mathbb{R},\\
\overline{\bm \varphi}(x,t)=\overline{\bm\varphi}(x, t+1) &\text{in }\,\,\Omega\times\mathbb{R}.
\end{array}
\end{cases}
\end{equation}
Then \eqref{eq:lowerbound} is a direct consequence of the comparison principle 
established in \cite[Sect. 2]{BH2020}.

By \eqref{liu-20240317-1}, direct calculations yield 
\begin{align}\label{liu-20240602-2}
    &\omega\partial_t \overline{\varphi}_i-\rho d_i\Delta \overline{\varphi}_i
-\sum_{j=1}^n a_{ij}\overline{\varphi}_j-\underline{h}(\omega)\overline{\varphi}_i\notag\\
=&\omega(\partial_t \log \phi_i)\overline{\varphi}_i-\rho d_i \frac{\Delta \phi_i}{\phi_i}\overline{\varphi}_i +2\sqrt{\rho}d_i (\nabla\log \phi_i\cdot \nabla v) \overline{\varphi}_i+\sqrt{\rho} d_i \overline{\varphi}_i \Delta v \notag\\
&- d_i |\nabla v|^2  \overline{\varphi}_i-\overline{\varphi}_i\sum_{j=1}^n a_{ij}\frac{\phi_j}{\phi_i}-\underline{h}(\omega)\overline{\varphi}_i\\
=& \left[h(x,\omega)-\underline{h}(\omega)-d_i|\nabla v|^2\right]\overline{\varphi}_i\notag
\\
&+\sqrt{\rho}d_i\left[\Delta v+2\nabla\log \phi_i\cdot \nabla v-\sqrt{\rho} \frac{\Delta \phi_i}{\phi_i}\right]\overline{\varphi}_i, \quad \forall (x,t)\in \Omega\times \mathbb{R}.\notag
\end{align}

By continuity, there exists some $\delta>0$
such that $\Delta v(x)\ge \Delta v(x_k)/2$
for $x\in B_\delta(x_k)$, $k=1,\cdots, N$. 
In view of $\nabla v(x_k)=0$, we may further choose $\delta$ small if necessary such that $|\nabla\log \phi_i\cdot \nabla v|\leq \Delta v(x_k)/4$
for all $x\in B_\delta(x_k)$. Then by \eqref{eq:liu-temp_1-1} and \eqref{liu-20240602-2}, there exists some $\rho_1>0$ small such that  whenever $\rho\leq \rho_1$, it holds that
for all $1\leq k\leq N$, 
\begin{equation}\label{liu-48}
  \omega\partial_t \overline{\bm\varphi}-\rho {\bf D}\Delta \overline{\bm\varphi}
-{\bf A}\overline{\bm\varphi}
-\underline{h}(\omega)\overline{\bm\varphi}
\geq \frac{\Delta v(x_k)}{4} \underline{d} \sqrt{\rho}\overline{\bm\varphi}
\quad \text{for}\,\,(x,t)\in B_\delta(x_k)\times \mathbb{R}.
\end{equation}

By \eqref{eq:liu-temp_1-1},  there exists some constant $\eta=\eta(\delta)>0$ such that
$$
 h(x,\omega)-\underline{h}(\omega)-\overline{d}|\nabla v|^2 \geq \eta, \quad \forall x\not\in B_\delta(x_k),\,\,k=1,\cdots, N.$$
Hence, by \eqref{liu-20240602-2} we may choose some constant  $\rho_2>0$ such that for all $\rho<\rho_2$ and  $1\leq k\leq N$,
\begin{equation}\label{eq:liu-temp22}
\omega\partial_t \overline{\bm\varphi}-\rho {\bf D}\Delta \overline{\bm\varphi}
-{\bf A}\overline{\bm\varphi}
-\underline{h}(\omega)\overline{\bm\varphi}
\geq \frac{\eta}{2}\overline{\bm\varphi} \quad  \text{for }\,\, x\not\in B_\delta(x_k), \ t\in \mathbb{R}. 
\end{equation}

Combining with \eqref{liu-48} and \eqref{eq:liu-temp22}, we have verified that the first inequality in \eqref{liu-50-1} holds with $\rho\leq \min\{\rho_1,\rho_2\}$.
 For $x\in\partial\Omega$ and
$t\in \mathbb{R}$, observe that
\begin{equation*}
\nabla \overline{\bm\varphi}\cdot\nu=\overline{\bm\varphi} \left[ \nabla\log {\bm\phi}\cdot \nu-\frac{1}{\sqrt{\rho}}
\nabla v\cdot\nu\right].
\end{equation*}
Due to $\nabla v\cdot\nu>0$, the boundary condition in \eqref{liu-50-1} can be verified by
choosing $\rho$ small.
 Therefore, \eqref{liu-50-1} holds and Lemma \ref{liulem-4} is proved.
\end{proof}

We conclude this section by proving Corollary {\rm \ref{liucor-1}}.
\begin{proof}[Proof of Theorem {\rm \ref{liucor-1}}]
   We first assume $\underline{C}<C_*=\overline{C}$ and prove part {\rm(i)}.
Since $C_*=\overline{C}$, Lemma \ref{liu-20240526} implies $C_*=C_*^+=\underline{C}^+=\overline{C}$, and thus
by  Lemma \ref{liu-20240516} we derive 
$$\lim_{\rho\to 0 }\lambda(\omega, \rho)=\lim_{\rho\to +\infty }\lambda(\omega, \rho)=\overline{C}, \qquad \forall \omega>0.$$
It suffices to show $\lambda(\omega, \rho)<\overline{C}$ for all $\omega, \rho>0$.
Suppose on the contrary that $\lambda(\omega_*,\rho_*)=\overline{C}$ for some $\omega_*,\rho_*>0$. Let $\overline{\lambda}(\rho)$ be the principal eigenvalue of elliptic problem \eqref{liu-20240227-1}.  By \eqref{liu-20240518-5},  $\lambda(\omega_*,\rho_*)\leq \overline{\lambda}(\rho_*)\leq \overline{C}$. Due to $\lambda(\omega_*,\rho_*)=\overline{C}$, the monotonicity of $\omega\mapsto \lambda(\omega,\rho_*)$ in Theorem \ref{TH1-1} implies $\lambda(\omega,\rho_*)\equiv \overline{\lambda}(\rho_*)=\overline{C}$ for all $\omega>0$. In particular, Proposition \ref{TH-liu-20240227}{\rm (ii)} yields
$$\overline{C}=\lim_{\omega\to 0}\lambda(\omega,\rho_*)=\underline{\lambda}(\rho).$$
By the monotonicity of $\rho\mapsto \underline{\lambda}(\rho)$,  we see that $\underline{\lambda}(\rho)\equiv \overline{C}$ for all $\rho>0$, and hence
$\underline{C}=\lim_{\rho\to 0}\underline{\lambda}(\rho)=\overline{C}$.
 This contradicts our assumption. Part {\rm(i)} thus  follows.

We next show  part {\rm (ii)}. Under the assumptions in part {\rm (ii)}, Lemma \ref{liulem-4} implies that $\lambda(\omega, \rho)$ attains a local minimum at $\rho=0$.  It remains to prove parts {\rm (a)}-{\rm(c)}. 
    Let $\rho_{C_*}$ and $\omega_{C_*}$ be determined by Lemma \ref{d-ell} and Theorem \ref{liu-levelset} with $\ell=C_*$.  
  Set
  \begin{equation}\label{liu-66}
   \omega_*:=  \max\limits_{\rho\in[0,\rho_{C_*}]}\omega_{C_*}(\rho).
    \end{equation}
   Then $\omega_*>0$ due to $\omega_{C_*}(0)=\omega_{C_*}(\rho_{C_*})=0$. Fix any $\omega\in (0,\omega_*)$. We define
  \begin{equation}\label{liu-67}
      \begin{split}
         \underline{\rho}:=\inf\{\rho\in(0, \rho_{C_*}): \, \omega_{C_*}(\rho)= \omega\},\quad
         \overline{\rho}:=\sup\{\rho\in(0, \rho_{C_*}): \, \omega_{C_*}(\rho)= \omega\}.
      \end{split}
  \end{equation}
      Since 
      $\omega<\omega_*$, it hold that $0<\underline{\rho}<\overline{\rho}$ and $\lambda(\omega, \underline{\rho})=\lambda(\omega, \overline{\rho})=C_*$, which proves part {\rm (ii)}-{\rm(a)}.

      On the other hand,
  it follows from \eqref{liu-67} 
   that 
   \begin{equation}\label{liu-67-1}
       \omega_{C_*}(\rho)<\omega\quad \text{and}\quad \lambda(\omega,\rho)\neq C_*,\quad \forall \rho\in (0,\underline{\rho})\cup (\overline{\rho},\infty).
   \end{equation}
The monotonicity of $\omega\mapsto\lambda(\omega,\rho)$ in Theorem \ref{TH1-1} implies  
$\lambda(\omega, \rho)> \lambda(\omega_{C_*(\rho)}, \rho)=C_*$ for  $\rho\in (0,\underline{\rho})\cup (\overline{\rho},\infty)$. This  
proves part {\rm (ii)}-{\rm(b)}.

  It suffices to show part {\rm (ii)}-{\rm(c)}, i.e.  $\lambda(\omega,\rho)< C_*$ for some $\rho\in(\underline{\rho},\overline{\rho})$.   If not,  then $\lambda(\omega,\rho)\geq C_*$ for all $\rho\in(\underline{\rho},\overline{\rho})$. In view of $\lambda(\omega_{C_*}(\rho),\rho)= C_*$, the monotonicity in Theorem \ref{TH1-1} implies  $\omega_{C_*}(\rho)\leq \omega$ for all $\rho\in(\underline{\rho},\overline{\rho})$. This together with \eqref{liu-67-1} gives  $\omega_{C_*}(\rho)\leq \omega$ for all $\rho\in(0,  \rho_{C_*})$, which contradicts $\omega<\omega_*$ and \eqref{liu-66}. Therefore, 
  part {\rm (ii)}-{\rm(c)} holds.

 The proof is now complete.
\end{proof}

  \bigskip
   \noindent{\bf Acknowledgement.} This work is partially supported by the NSFC grant (12201041) 
   and Beijing Institute of Technology Research Fund Program for Young Scholars (XSQD-202214001).


\bigskip
\bigskip
\baselineskip 18pt
\renewcommand{\baselinestretch}{1.2}

\end{document}